\documentclass[10pt,oneside,notitlepage]{amsart}

\usepackage[utf8]{inputenc}
\usepackage[T1]{fontenc}
\usepackage{amsmath}
\usepackage{amsthm}
\usepackage{amsfonts}
\usepackage{amssymb}
\usepackage{mathtools}
\usepackage{latexsym}
\usepackage[hidelinks]{hyperref}
\usepackage[backgroundcolor = white, linecolor = red, bordercolor =
red]{todonotes}
\usepackage[all,cmtip]{xy}
\usepackage[only,llbracket,rrbracket]{stmaryrd}
\usepackage{mathabx}
\usepackage{xparse}
\usepackage[shortlabels]{enumitem}
\usepackage{subfig}

\theoremstyle{plain}
\newtheorem{theorem}{Theorem}
\newtheorem{corollary}[theorem]{Corollary}
\newtheorem{lemma}[theorem]{Lemma}
\newtheorem{proposition}[theorem]{Proposition}

\theoremstyle{remark}
\newtheorem{remark}[theorem]{Remark}
\newtheorem{example}[theorem]{Example}

\numberwithin{equation}{section}
\numberwithin{theorem}{section}
\numberwithin{figure}{section}

\newcounter{AssumptionCounter}    % Counter for the abstract assumptions

\newcommand{\Dual}[1]{#1'} % Dual space

\newcommand{\R}{\mathbb{R}}

\newcommand{\ta}{\widetilde{a}}     % Extended bilinear form
\newcommand{\tA}{\widetilde{A}} % Extended Riesz isometry
\newcommand{\tV}{\widetilde{V}}
\newcommand{\tv}{\tilde{v}}

\newcommand{\calA}{\mathcal{A}}

\newcommand{\calE}{\mathcal{E}}
\newcommand{\calI}{\mathcal{I}}
\newcommand{\calP}{\mathcal{P}}

\newcommand{\Cav}{C_{\mathtt{av}}}
\newcommand{\Cqo}{C_{\mathtt{qo}}} % Quasi-optimality constant
\newcommand{\Cloc}{C_{\mathtt{loc}}}
\newcommand{\Ccol}{C_{\mathtt{col}}}
\newcommand{\Cobl}{C_{\mathtt{ort}}}

\newcommand{\Cstab}{C_{\textsf{stab}}}
\newcommand{\Csplt}{C_\mathtt{splt}}
\newcommand{\CL}{C_{\mathtt{\overline{com}}}}
\newcommand{\CNL}{C_{\mathtt{\underline{com}}}}

\DeclareMathOperator{\Res}{\mathrm{res}} % Residual
\DeclareMathOperator{\Osc}{\mathrm{osc}} % Oscillation
\DeclareMathOperator{\Ncf}{\mathrm{ncf}} % Nonconformity measure
\DeclareMathOperator{\Est}{\mathrm{est}} % abstract estimator
\newcommand{\Index}{\mathcal{Z}} % Index set for partition of unity

\newcommand{\scp}[3][]{\ensuremath{\langle #2,#3\rangle_{#1}}}
\newcommand{\norm}[2][]{\|#2\|_{#1}}

\newcommand{\normtr}[1]{ \lvert\!\lvert\!\lvert{#1} \rvert\!\rvert\!\rvert}

\newcommand{\grid}{\mathcal{T}}  
\newcommand{\sides}{\mathcal{F}}
\newcommand{\vertices}{\mathcal{V}}
\newcommand{\normal}{\mathsf{n}}
\newcommand{\tangent}{\mathsf{t}}
\newcommand{\poly}{\mathbb{P}}

\newcommand{\jump}[1]{\left\llbracket#1\right\rrbracket}

\NewDocumentCommand{\mean}{sO{}m}{%
  \IfBooleanTF{#1}
    {\meanext{#3}}
    {\meanx[#2]{#3}}%
}

\NewDocumentCommand{\meanext}{m}{%
  \sbox0{%
    \mathsurround=0pt % just for safety
    $\left\{\vphantom{#1}\right.\kern-\nulldelimiterspace$%
  }%
  \sbox2{\{}%
  \ifdim\ht0=\ht2
    \{\kern-.45\wd2 \{#1\}\kern-.45\wd2 \}%
  \else
    \left\{\kern-.5\wd0\left\{#1\right\}\kern-.5\wd0\right\}%
  \fi
}

\NewDocumentCommand{\meanx}{om}{%
  \sbox0{\mathsurround=0pt$#1\{$}%
  \sbox2{\{}%
  \ifdim\ht0=\ht2
    \{\kern-.45\wd2 \{#2\}\kern-.45\wd2 \}%
  \else
    \mathopen{#1\{\kern-.5\wd0 #1\{}
    #2
    \mathclose{#1\}\kern-.5\wd0 #1\}}
  \fi
}

\newcommand{\CR}{{C\!R}}
\renewcommand{\MR}{{M\!R}}
\newcommand{\mr}{\mathsf{MR}}
\newcommand{\HCT}{H\!C\!T}

\newcommand{\COip}{\mathtt{C0}}
\allowdisplaybreaks

\begin{document}

% Title
\title[A~posteriori estimates for quasi-optimal nonconforming methods]{Strictly equivalent a~posteriori error estimators for quasi-optimal nonconforming methods}

% Authors
\author[C.~Kreuzer]{Christian Kreuzer}
\address{TU Dortmund \\ Fakult{\"a}t f{\"u}r Mathematik \\ D-44221 Dortmund \\ Germany}
\email{christian.kreuzer@tu-dortmund.de}

\author[M.~Rott]{Matthias Rott}
\address{TU Dortmund, Fakult{\"a}t f{\"u}r Mathematik \\ D-44221 Dortmund \\ Germany}
\email{matthias.rott@tu-dortmund.de}

\author[A.~Veeser]{Andreas Veeser}
\address{Universit\`{a} degli Studi di Milano, Dipartimento di
	Matematica\\ 20133 Milano\\ Italy}
\email{andreas.veeser@unimi.it}

\author[P.~Zanotti]{Pietro Zanotti}
\address{Universit\`{a} degli Studi di Milano, Dipartimento di
	Matematica\\ 20133 Milano\\ Italy}
\email{pietro.zanotti@unimi.it}

% Keywords

% Subject classification

\begin{abstract}
We devise a posteriori error estimators for quasi-optimal nonconforming finite element methods approximating symmetric elliptic problems of second and fourth order. These estimators are defined for all source terms that are admissible to the underlying weak formulations. More importantly, they are equivalent to the error in a strict sense. In particular, their data oscillation part is bounded by the error and, furthermore, can be designed to be bounded by classical data oscillations.  The estimators are computable, except for the data oscillation part. Since even the computation of some bound of the oscillation  part is not possible in general, we advocate to handle it on a case-by-case basis. We illustrate the practical use of two estimators obtained for the Crouzeix-Raviart method applied to the Poisson problem with a source term that is not a function and its singular part with respect to the Lebesgues measure is not aligned with the mesh.    
\end{abstract}

\maketitle

\section{Introduction}
\label{sec:introduction}
%
%
% relevance of nonconforming methods
Nonconforming finite element methods are a well-established technique for the approximate solution of partial differential equations (PDEs). Classical nonconforming elements, like the ones of Morley \cite{Morley:68} and Crouzeix-Raviart \cite{Crouzeix.Raviart:73}, were originally proposed as valuable alternatives to conforming ones for the biharmonic and the Stokes equations. Later on, other nonconforming techniques like discontinuous Galerkin and $C^0$ interior penalty
methods have proved to be competitive for a wide range of problems.  

% nature of nonconforming methods
Nonconforming methods weaken the coupling between elements so that at least some discrete functions are not admissible to the weak formulation of the PDE of interest. On the one hand, this increases flexibility for approximation and accommodating desired structural properties,  but on the other hand it complicates the theoretical analysis and, without suitable measures, excludes data that is admissible in the weak PDE. 

% aposteriori analysis ...
An a~posteriori error analysis aims at devising a quantity, called estimator, that, ideally, is computable and equivalent to the error of the approximate solution. Such an estimator can be used to asses the quality of the approximate solution and, if it splits into local contributions, to guide adaptive mesh refinement.
% ... and state of the art  for nonconforming methods
To outline the structure of most available results, consider a weak formulation of the form
\begin{equation*}
	u \in V \text{ such that } \forall v \in V \; a(u,v) = \scp{f}{v},
\end{equation*}
where the bilinear form $a$ is symmetric and $V$-coercive, and apply some nonconforming method based upon a discrete space $V_h \not\subset V$. Then the error $\norm{u-u_h}$ and the estimator $\Est_h$ are related by the following equivalence,  cf.\ \cite{Dari.Duran.Padra:95, Karakashian.Pascal:03, Carstensen:05, Vohralik:07, Ainsworth.Rankin:08, Brenner:15, Carstensen.Nataraj:22, Carstensen.Graessle.Nataraj:24}:
\begin{equation}
	\label{intro-spoiled-equivalence}
	\norm{u-u_h}^2 + \Osc_h^2
	\eqsim
	\Est_h^2
	=
	\Ncf_h^2 + \eta_h^2 + \Osc_h^2,
\end{equation} 
where $\Ncf_h$ is an approximation of the distance of $u_h$ to $V$, $\eta_h$ is a PDE-specific part of the estimator depending on $u_h$, and $\Osc_h$ is the so-called oscillation, typically involving only data. Notice the presence of $\Osc_h$ on both sides, which spoils the ideal equivalence of error and estimator. This flaw is due to the fact that in general the error cannot bound the oscillation $\Osc_h$ on a given mesh and this may even persist under adaptive refinement; cf.\  Remark~\ref{R:strict-equivalence} below. % \cite[Section~3.8]{Kreuzer.Veeser:21} or \cite[???]{Bonito.Canuto.Nochetto.Veeser:24}.
Often and related, $\eta_h$  cannot be bounded by the error alone. 

% strictly equivalent estimators as main result
In contrast to \eqref{intro-spoiled-equivalence}, this article establishes  the strict equivalence
\begin{equation}
	\label{intro-strict-equivalence}
	\norm{u-u_h}^2
	\eqsim
	\Est_h^2
	=
	\Ncf_h^2 + \eta_h^2 + \Osc_h^2,
\end{equation} 
where $\Osc_h$ and often also $\eta_h$ differ from their counterparts in \eqref{intro-spoiled-equivalence} and are defined for all $f \in V'$. Moreover, the oscillation $\Osc_h$ in \eqref{intro-strict-equivalence} can be designed to be bounded by the one in \eqref{intro-spoiled-equivalence}. Observe that now the error dominates both oscillation $\Osc_h$ and $\eta_h$. We achieve this improvement by generalizing the new approach of \cite{Kreuzer.Veeser:21}, used also in the recent survey \cite{Bonito.Canuto.Nochetto.Veeser:24}, to nonconforming methods. To shed some light on the proof, let us outline our approach.

% overview of the main steps of our aposteriori analyis
%
% error-residual relationship and orthogonal decomposition
We start by identifying a residual $\Res_h$ and, similarly to existing approaches, obtain 
\begin{equation*}
	\norm{u-u_h}^2
	=
	\norm[\Dual{(V+V_h)}]{\Res_h}^2
	=
	\inf_{v \in V} \norm{u_h -v}^2
	+
	\norm[\Dual{V}]{\Res_h^\mathtt{C}}^2
\end{equation*}
with the conforming part $\Res_h^\texttt{C} := \Res_h{}_{|V}$ of the residual. Note that the infinite dimension of $V$ obstructs the computation of both terms on the right-hand side. For the first term, %depending only on the discrete function $u_h$,
averaging operators allow deriving indicators $\Ncf_h$ that are strictly equivalent, computable, and readily split into local contributions.

% localization and decomposition of conforming residual
The second term with the conforming residual $\Res_h^\mathtt{C}$ is more delicate in view of the dual norm and the fact that data, say only the source $f$, is typically taken from an infinite-dimensional space. Considering just for a moment a simplifying special case, we address these two issues by the following two steps:
\begin{equation}
	\label{intro-conf-res}
	\norm[\Dual{V}]{\Res_h^\mathtt{C}}^2
	\eqsim
	\sum_{z\in\vertices} 	\norm[\Dual{V_z}]{\Res_h^\mathtt{C}}^2
	\eqsim
	\sum_{z\in\vertices}  \norm[\Dual{V_z}]{\calP_z\Res_h^\mathtt{C}}^2
	+
	\sum_{z\in\vertices} 	\norm[\Dual{V_z}]{f -\calP_z f}^2 ,
\end{equation}
where $\vertices$ are the vertices of the underlying mesh, $V_z$ are local counterparts of $V$, and $\calP_z$ are local projections onto finite-dimensional counterparts $D_z$ of $\Dual{V_z}$.

% importance of near orthogonality and role of quasi-optimality
The general version of the first step in \eqref{intro-conf-res} is given in Lemma~\ref{L:localizating-conf-res}. It hinges on (near) orthogonality properties of the conforming residual $\Res_h^\mathtt{C}$, which originate in the nonconforming discrete test functions. If the nonconforming method is  quasi-optimal in the sense of
\begin{equation}
\label{eq:apriori-new}
	\norm{u_h-u}
	\leq
	\Cqo \inf_{v_h \in V_h} \norm{u-v_h},
\end{equation}
they are verified in Lemma~\ref{L:conforming-residual}, else they have to be measured by an additional indicator in the spirit of Remark~\ref{R:role-of-qopt}.

% approximate residual and oscillation
The second step in \eqref{intro-conf-res} further decomposes the residual, adapting the ideas in \cite{Kreuzer.Veeser:21} to new features arising from nonconformity. Thanks to $\dim D_z < \infty$, classical techniques allow deriving strictly equivalent and computable indicators $\eta_h$ for the sum with the approximate residuals $\calP_z\Res_h^\mathtt{C}$, while the computability issue due to infinite dimensional data is isolated in the second sum, which coincides with $\Osc_h^2$ in \eqref{intro-strict-equivalence}. The latter decomposition is also useful for adaptivity; see Remark~\ref{R:conf-res-splitting-and-adaptivity} below.
% or \cite{Kreuzer.Veeser:19, Bonito.Canuto.Nochetto.Veeser:24}.
It is also the crucial step for the difference between the two equivalences \eqref{intro-spoiled-equivalence} and \eqref{intro-strict-equivalence}. For the latter, the projections $\calP_z$ have to be stable in a $\Dual{V}$-like manner, while classical techniques (implicitly) use projections that are stable only in proper subspaces; see Remark~\ref{R:strict-equivalence}.

% computability of oscillation, surrogate oscillation
We advocate to first apply the outlined approach and then address the computability of the oscillation $\Osc_h$ on a case-by-case basis, exploiting all the available information on the structure and regularity of the given data. We postpone a more detailed discussion to the Remarks~\ref{R:computability}, \ref{R:neglecting-data-osc}, and \ref{R:class-osc-as-surrogate} and only mention here that an a~posteriori analysis based upon a `more computable' oscillation may lead to indicators $\eta_h$ that cannot be bounded by the error alone, as in the case of several classical techniques.

% generality and flexibility of approach + DPP(V) + no variable diffusion etc
To illustrate the generality of the outlined approach, we apply it to various quasi-optimal nonconforming methods covering arbitrary approximation order as well as second- and fourth-order problems. The approach also allows for estimator simplifications thanks to nonconformity as in \cite{Dari.Duran.Padra:95,Dari.Duran.Padra.Vampa:96}. We restrict ourselves to equations with data-free symmetric principal terms in order to avoid the technical complications which are addressed in \cite[Section~4]{Bonito.Canuto.Nochetto.Veeser:24} in the case of conforming methods. Furthermore, for the sake of the simplicity, we do not cover estimators  based on flux equilibration; the related changes are however illustrated for conforming methods in \cite[Section~4.4]{Kreuzer.Veeser:21} (lowest-order case) and in \cite[Section~4.9]{Bonito.Canuto.Nochetto.Veeser:24} (higher-order cases). 

\subsection*{Organization}
In Section~\ref{sec:qo-methods}, we
briefly recall the framework for quasi-optimal nonconforming methods
from~\cite{Veeser.Zanotti:18}. Section~\ref{sec:posteriori-analysis} provides the framework for the a~posteriori analysis of this class of methods, based on the principles in~\cite{Kreuzer.Veeser:21}. Section~\ref{sec:Poisson} illustrates the application of the abstract results to lowest- and higher-order quasi-optimal methods for the Poisson problem, while Section~\ref{sec:Biharmonic} concerns lowest-order methods for the biharmonic problem. Finally, we illustrate the practical use of two derived estimators  in Section~\ref{sec:numerics}.

\subsection*{Notation}
Let $X$, $Y$ be Hilbert spaces. We equip the dual space $\Dual{X}$ with the norm $\norm[\Dual{X}]{\cdot} = \sup_{x \in X, \norm[X]{x} \leq 1} \scp{\cdot}{x}$, where \(\scp{\cdot}{\cdot}\) is the dual pairing on \(\Dual{X}\times
X\) and $\norm[X]{\cdot}$ is the norm on $X$. We denote the norm of a bounded linear operator $\mathcal{L}: X \to Y$  by $\normtr{\mathcal{L}}$. Moreover, we use standard notation for Lebesgue and Sobolev spaces.

\section{Quasi-optimal nonconforming methods}
\label{sec:qo-methods}
In this section we summarize the framework for quasi-optimal nonconforming methods introduced in \cite{Veeser.Zanotti:18}. We report only the notions and results that are useful to our subsequent developments.

Let $V$ be a Hilbert space with scalar product $a(\cdot,\cdot)$ and induced norm \(\norm{\cdot}:= \norm[V]{\cdot} := \sqrt{a(\cdot,\cdot)}\). As in \cite{Veeser.Zanotti:18}, we focus on PDEs whose weak formulation reads %as
\begin{equation}\label{Eq:AnaPro}
  \text{for \(f\in \Dual{V}\), find $u \in V$ such that } 
 \forall v\in V\quad a(u, v) = \scp{f}{v}.
\end{equation}
Thanks to the Riesz representation theorem, the operator \(A: V \to \Dual{V}\) induced by $a$ and defined 
as \(Av := a(v,\cdot)\) is an isometry. Hence, for any source term $f \in \Dual{V}$, there exists a unique solution $u = A^{-1} f$ of \eqref{Eq:AnaPro} with
\(\norm{u}=\norm[\Dual{V}]{f}\). 

Let $V_h$ be a finite-dimensional, linear space, which is possibly \emph{nonconforming}, i.e.\  not necessarily a subspace of $V$. We assume
that $a$ extends to a scalar product $\ta$ on the space $\tV := V + V_h$. The norm induced by $\widetilde{a}$ extends the one induced by $a$, therefore we still denote it by $\norm{\cdot}$. Consequently, the operator \(\tA:\tV\to\Dual{\tV}\) extending \(A\) via \(\tA \tilde{v}=\ta(\tilde{v},\cdot)\) is an isometry as well. Writing  $\Pi_V:\tV\to V$ for the \(\ta\)-orthogonal projection onto $V$, the \emph{distance to conformity} of an element \(\widetilde{v}\in\tV\) is given by the distance to its projection
\begin{align}\label{df:PiV}
  \norm{\widetilde{v}-\Pi_V \widetilde{v}}=\inf_{v\in V}\norm{\widetilde{v} - v}.
\end{align}

To discretize \eqref{Eq:AnaPro}, we need a nondegenerate bilinear form $a_h: V_h \times V_h \to \R$ and a functional $f_h \in \Dual{V_h}$, which approximate their counterparts $a$ and $f$ in a sense to be specified. According to \cite[Theorem~4.7]{Veeser.Zanotti:18}, a necessary condition for the quasi-optimality \eqref{eq:apriori-new} is \emph{full stability}, i.e.\  the mapping $V' \ni f \mapsto f_h \in V_h'$ induced by the discretization of the source term must be bounded. This, in turn, implies that we necessarily have $f_h = \calE^*_h f$, where $\calE^*_h: \Dual{V} \to \Dual{V_h}$ is the adjoint of some linear \emph{smoothing} operator $\calE_h: V_h \to V$. Therefore, we restrict our attention to nonconforming discretizations of \eqref{Eq:AnaPro} which read as follows:
\begin{equation} 
\label{Eq:DisEq}
  \text{for \(f\in \Dual{V}\)}, \text{find $u_h \in V_h$ such that}~\forall v_h\in V_h\quad
   a_h(u_h, v_h) = \scp{f}{\calE_h v_h}. 
\end{equation}

Since $a_h$ is nondegenerate and $\dim V_h < \infty$, the operator \(A_h: V_h \to \Dual{V_h}\), defined by \(A_h(v_h) := a_h(v_h, \cdot)\), is invertible. Then, for any source term $f \in \Dual{V}$, there exists a unique solution of the discrete problem \eqref{Eq:DisEq} with $u_h = A_h^{-1} \calE^*_h f =: M_h f$. Henceforth, we assume that the evaluations of $a_h$ and $f$ in the discrete problem are computationally realizable, so that $u_h$ is computable (in exact arithmetic). This entails that the evaluations of $f$ for all test functions in $\calE_h(V_h)$ are known. We call the linear operator $M_h: \Dual{V} \to V_h$ the \emph{nonconforming method}, i.e.\  the mapping of the source term $f$ to the approximation $u_h$ of $u$, whereas $P_h: V \to V_h$ defined as
\begin{equation}
\label{Eq:approximation-operator}
P_h := M_hA = A_h^{-1} \calE_h^* A
\end{equation}
is the associated approximation operator, i.e.\  the mapping of $u$ to $u_h$. Figure~\ref{Fig:LaplSch} summarizes the functional framework.

\begin{figure}[h]
	\[\xymatrixcolsep{4pc}
	\xymatrixrowsep{4pc}
	\xymatrix{
		\Dual{V}
		\ar[d]^{\calE^*_h}\ar[rd]^{M_h}
		& \ar[l]^{A} V
		\ar[d]^{P_h} \\
		\Dual{V}_h 
		& \ar[l]^{A_h}V_h}\]
	\caption{Commutative diagram summarizing the construction of a fully stable nonconforming method $M_h$.}
	\label{Fig:LaplSch}
\end{figure}
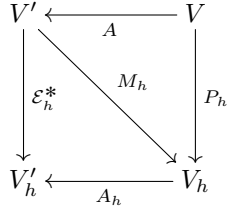

% generalized Galerkin orthogonality
The definitions \eqref{Eq:AnaPro} and \eqref{Eq:DisEq} of exact and approximate solution directly lead to the \emph{generalized Galerkin orthogonality}
\begin{equation}
\label{Eq:galerkin-orthogonality}
	\forall v_h \in V_h \quad  a_h(u_h, v_h) = a(u,\calE_h v_h).
\end{equation}
Note the presence of the smoothing operator $\calE_h$ and the related possible difference between the test functions in the bilinear forms $a$ and $a_h$.

% quasi-optimality´
In view of \cite[Theorem~4.14]{Veeser.Zanotti:18}, the compatibility condition
\begin{equation}
\label{Eq:AlgCons}
 \forall u\in V\cap V_h,~v_h\in V_h\quad a_h(u, v_h) = a(u,\calE_h v_h),
\end{equation}
called \emph{full algebraic consistency}, is necessary and sufficient for the quasi-optimality of a fully stable nonconforming method. More precisely, the following result holds true.

\begin{proposition}[Quasi-optimality]
\label{P:QuasiOptMethods}
For the fully stable method $M_h$, induced by the discretization \eqref{Eq:DisEq}, the following statements are equivalent: 
\begin{enumerate}
	\item \label{QuasiOptMethods-QuasiOptimality} $M_h$ is quasi-optimal, i.e.\  \eqref{eq:apriori-new} is verified for some constant $\Cqo\geq 1$.
	\item \label{QuasiOptMethods-FullConsistency} $M_h$ is fully consistent, i.e.\  \eqref{Eq:AlgCons} is verified.
	\item \label{QuasiOptMethods-ConsistencyMeasure} There is a constant $\delta_h \geq 0$ verifying
	\begin{equation*} \label{Eq:QuasiGO}
	\forall w_h\in V_h\quad 
	\sup_{v_h\in V_h}\frac{a_h(w_h,v_h)-a(\Pi_V w_h, \calE_h v_h)}
	{\norm[\Dual{V_h}]{A_h^* v_h}}
	\le \delta_h\|w_h-\Pi_V w_h\|.
	\end{equation*} 
\end{enumerate} 
Moreover, if \eqref{QuasiOptMethods-QuasiOptimality}-\eqref{QuasiOptMethods-ConsistencyMeasure} are verified, then the constant $\Cqo$ from \eqref{eq:apriori-new} is such that
\begin{equation*}
\label{Eq:quasi-optimality-constant}
\max\{\Cstab, \delta_h\}\le \Cqo\le
\sqrt{\Cstab^2+ \delta_h^2}
\end{equation*}
with 
\begin{equation*}
\label{Eq:stability-constant}
\Cstab:= \sup_{v_h \in V_h}
\dfrac{\norm{\calE_h v_h}}{\norm[\Dual{V_h}]{A_h^*v_h}}.
\end{equation*}
\end{proposition}

\begin{proof}
See \cite[Theorem~4.14]{Veeser.Zanotti:18} and, for the gap between the barriers for $\Cqo$, \cite[Remark~3.5]{Veeser.Zanotti:18}.	
\end{proof}

We refer to the constants $\delta_h$ and $\Cstab$ in Proposition~\ref{P:QuasiOptMethods} as the \emph{consistency measure} and the \emph{stability constant} of the method $M_h$, respectively.

\begin{remark}[Overconsistency]
\label{R:overconsistent-methods}	
Quasi-optimal methods satisfying $\delta_h = 0$ are called \emph{overconsistent} in \cite{Veeser.Zanotti:19b}, hinting at the fact that this condition does not need to be satisfied for consistent quasi-optimality (or convergence for a suitable sequence of discrete spaces $V_h$). Such methods are closer to conforming methods in the following sense. First, their quasi-optimality constant equals the stability constant, $\Cqo = \Cstab$. Second,  the error function is $\ta$-orthogonal to $\calE_h(V_h)$ because the generalized Galerkin orthogonality \eqref{Eq:galerkin-orthogonality} can be rewritten as
\begin{equation}
\label{Gal-orth-under-overconsistency}
 \forall v_h \in V_h
\quad
 \ta(u-u_h,\calE_h v_h) = 0.
\end{equation}
\end{remark}

\section{Abstract a~posteriori analysis}
\label{sec:posteriori-analysis}
%
%
% goal of a posteriori analysis
The goal of an a~posteriori analysis is to devise a quantity, called \emph{error estimator}, that is equivalent to the typically unknown error, that splits into local contributions, so-called \emph{indicators}, and that is computable. Such an estimator allows assessing the quality of the current approximate solution and its indicators may be used as input for adaptive mesh refinement.
 
% purpose and organization of this section
This section proposes a framework for the a~posteriori analysis of
quasi-optimal nonconforming methods. Conceptually speaking, it
proceeds in three steps. The first step determines the
\emph{residual}, along with a convenient decomposition, and relates it
to the error by means of a dual norm. This complements the previous
frameworks used in \cite{Karakashian.Pascal:03,Vohralik:07} and is an
alternative to \cite{Carstensen:05}.
The second step provides a
\emph{localization}, i.e.\ it splits the global residual norm into
local contributions. Finally, the third step addresses computability,
generalizing the approach \cite{Kreuzer.Veeser:21} to nonconforming
methods. 

\subsection{The residual of nonconforming errors}
\label{sec:err=res}
We start by looking for a quantity that is equivalent to the error and given in terms of available information, viz.\ problem data $f$ and approximate solution $u_h$ in the setting of Section~\ref{sec:qo-methods}.

% conforming case
In the special case of a conforming Galerkin discretization, i.e.\ for $V_h \subseteq V$, $a_h = a_{|V_h \times V_h}$, and $\calE_h = \mathrm{id}_V$ in \eqref{Eq:DisEq}, the residual $f-Au_h \in \Dual{V}$ satisfies
\begin{equation}
\label{err-res-for-conf-case}
 \norm{ u-u_h } 
 =
 \sup_{v \in V, \norm{v} \leq 1} a(u-u_h,v) 
 =
 \norm[\Dual{V}]{f-Au_h}
\end{equation}
thanks to \eqref{Eq:AnaPro} and the fact that $A$ is a linear  isometry. 

% towards a nonconforming residual
In the general case, the possibility of $u_h \not\in V$ and the identity
\begin{equation}
\label{var-char-of-err}
 \norm{u-u_h}  = \sup_{\tv \in \tV, \norm{\tv} \leq 1} \ta(u-u_h, \tv)
\end{equation}
thus suggest to suitably extend the residual $f - Au_h \in \Dual{V}$ in \eqref{err-res-for-conf-case} to a functional in $\Dual{\tV}$ allowing for $u_h  \not\in V$.  While $\tA u_h \in \Dual{\tV}$ is such an extension for the term $Au_h \in \Dual{V}$, the extension for $f \in \Dual{V}$ is less obvious. Mimicking the role of \eqref{Eq:AnaPro} in \eqref{err-res-for-conf-case} for the given nonconforming situation, we observe
\begin{equation}
\label{extended-var-eq}
 \forall \tv \in \tV
\quad
 \ta(u,\tv) = \ta(u,\Pi_V\tv) = \scp{f}{\Pi_V\tv}.
\end{equation}
Note that these conditions are meaningful for $u \in \tV$ and the original condition $u \in V$ is encoded in the new conditions
\begin{equation*} 
  \forall v^\perp \in V^\perp \;\; \ta(u,v^\perp) = 0
\end{equation*}
where $V^\perp$ denotes the $\ta$-orthogonal complement of $V$ in $\tV$. These observations suggest to take $\scp{f}{\Pi_V\cdot} \in \Dual{\tV}$ as extension for $f \in \Dual{V}$ and motivate the following residual notions: given the solution $u_h$ of the discrete problem \eqref{Eq:DisEq}, we define its residual as
\begin{equation}
\label{Eq:residual-nonconforming-method}
 \Res_h := \Res_h(f,u_h) := \scp{f}{\Pi_V \cdot} - \tA u_h \in \Dual{\tV},
\end{equation}
which can be decomposed into a conforming and nonconforming component by
\begin{subequations}
\label{Eq:residual-components}
\begin{align}
 \Res_h^{\mathtt{C}}
 &:=
 \Res_h^{\mathtt{C}}(f,u_h)
 :=
 \Res_h{}_{|V} = f - \tA u_h \in \Dual{V},
\\
 \Res_h^{\texttt{NC}}
 &:=
 \Res_h^{\mathtt{NC}}(u_h)
 :=
 \Res_h{}_{|V^\perp} = -\tA u_h \in \Dual{\big(V^\perp\big)}.
\end{align}
\end{subequations}
Thus, we have $\Res_h = \scp{\Res_h^{\mathtt{C}}}{\Pi_V \cdot} + \scp{\Res_h^{\texttt{NC}}}{(\mathrm{id}_{\tV} - \Pi_V)(\cdot)}$ and, upon identifying a component with its respective term in this relationship, it may be viewed as an element of $\Dual{\tV}$.

The following lemma states a well-known orthogonal decomposition of the error, which is usually formulated not completely in terms of residuals.
\begin{lemma}[Error, residual and their decompositions]
\label{L:err-res-decomps}
For the respective solutions $u$ and $u_h$ of \eqref{Eq:AnaPro} and \eqref{Eq:DisEq}, we have
\begin{gather*}
 \norm{u-u_h}^2
 =
 \norm[\Dual{\tV}]{\Res_h}^2
 =
 \norm[\Dual{V}]{\Res_h^{\mathtt{C}}}^2
  + \norm[\Dual{(V^\perp)}]{\Res_h^{\mathtt{NC}}}^2,
\\
 \norm{u-\Pi_V u_h}
 =
 \norm[\Dual{V}]{\Res_h^{\mathtt{C}}},
\quad\text{and}\quad
 \norm{u_h - \Pi_V u_h}
 =
 \norm[\Dual{(V^\perp)}]{\Res_h^{\mathtt{NC}}}.
\end{gather*}
\end{lemma}

\begin{proof}
Let us write $e_h := u - u_h$ for the error function and observe $\Pi_V u = u$. In light of \eqref{var-char-of-err},
\begin{gather*}
 \norm{u - \Pi_V u_h}
 =
 \sup_{v \in V, \norm{v} \leq 1} \ta(e_h, v),
\quad\text{and}\quad
 \norm{u_h - \Pi_V u_h}
 =
 \sup_{v^\perp \in V^\perp, \norm{v^\perp} \leq 1} \ta(e_h, v^\perp),
\end{gather*}
the identities for these errors follow from \eqref{extended-var-eq} and the fact that $\tA$ is a linear isometry. The missing identity is then a consequence of orthogonality
\begin{equation*}
 \ta(u - \Pi_V u_h, u_h - \Pi_V u_h)
 =
 \ta(\Pi_V e_h, e_h - \Pi_V e_h)
 =
 0, 
\end{equation*}
which expresses also the orthogonality of the residual components in $\Dual{\tV}$ endowed with the scalar product $\ta(\tA^{-1}\cdot, \tA^{-1}\cdot)$.
\end{proof}

\begin{remark}[Extended weak formulation]
\label{R:ex-weak-form}
An equivalent formulation of the weak formulation \eqref{Eq:AnaPro} in the extended space $\tV$ reads as follows: for $f \in \Dual{V}$, find $u \in \tV$ such that 
\begin{equation*}
 \forall v \in V \quad a(u,v) = \scp{f}{v} 
\qquad \text{and} \qquad
 \forall v^\perp \in V^\perp \quad \ta(u,v^\perp) = 0.
\end{equation*}
This formulation is an alternative departure point for the definition of the residual $\Res_h$ and its components $\Res_h^{\mathtt{C}}$ and $\Res_h^{\mathtt{NC}}$.
\end{remark}

\begin{remark}[Decomposition features]
\label{R:decomp-features}
Orthogonal decompositions are `constant-free' and therefore desirable
in an a~posteriori error analysis. The orthogonal decompositions in
Lemma~\ref{L:err-res-decomps} are essentially characterized by the
property that one component coincides with the conforming
subcase. This prepares the ground for using techniques available
for conforming methods and will be exploited in the following steps
concerning localization and computability. 
\end{remark}

In what follows, we shall refer to the two parts  $\Res_h^{\mathtt{C}}$ and $\Res_h^{\mathtt{NC}}$ of the residual, respectively, as \emph{conforming} and \emph{nonconforming residual} for short and handle them differently.

\subsection{Indicators for the nonconforming residual}
\label{sec:ind-for-nc-res}
%
% intro
We start with the nonconforming residual norm $\norm[\Dual{(V^\perp)}]{\Res_h^{\mathtt{NC}}}$, which, according to Lemma~\ref{L:err-res-decomps}, equals the distance to conformity $\norm{u_h-\Pi_V u_h}$ of the approximate solution $u_h$. 

% splitting of extended energy norm
In applications, the abstract space $V$ is a space of functions over some domain $\Omega$ and the extended bilinear form $\widetilde{a}$ is defined in terms of an integral over $\Omega$. Thus, given a partition $\grid$ of $\Omega$, the extended energy norm can be written as
\begin{equation}
	\label{energy-norm-localization}
	\| \cdot \|^2
	=
	\sum_{T \in \grid} | \cdot |_T^2,
\end{equation}
where the local contributions $| \cdot |_T$ are in general only seminorms.

% non-computability and global nature of $\Pi_V$
The distance to conformity $\norm{u_h-\Pi_V u_h}$ therefore readily splits into the local quantities $| u_h - \Pi_V u_h |_T$, $ T \in \grid$. Note however that the presence of the projection $\Pi_V$ and the infinite dimension of $V$ typically entail that these quantities are not computable and depend on $u_h$ in a global manner. We therefore look for an equivalent alternative. For this purpose, the following proposition individuates a simple algebraic condition. It coincides with \cite[Lemma~2.2]{Carstensen.Nataraj:22}, which is proved by a different argument and seems to be used in the context of the smoothing operator $\calE_h$ in the discrete problem \eqref{Eq:DisEq}; for such a usage, it is instructive to take Remark~\ref{R:dG-smoother-ass} into account. 

\begin{proposition}[Approximating the distance to conformity]
\label{P:inds-for-dist-to-conf}
If 
\begin{equation}
\label{A:equiv-to-dist-to-conf}
\stepcounter{AssumptionCounter}
\tag{H\arabic{AssumptionCounter}}
\begin{aligned}
 \calA_h:V_h \to V &\text{ is a bounded linear operator with }
 \\
 	&\forall w_h \in V \cap V_h
 	\quad
 	\calA_h w_h = w_h,
\end{aligned} 
\end{equation}
then there exists a constant $\Cav > 0$ such that
\begin{equation*}
%\label{equiv-for-dist-to-conf}
 \forall v_h \in V_h
\quad
 \Cav  \| v_h - \calA_h v_h\|
 \leq
 \| v_h - \Pi_V v_h\|
 \leq
 \| v_h - \calA_h v_h \|.
\end{equation*}
\end{proposition} 

\begin{proof}
The upper bound is an immediate consequence of $\calA v_h \in V$, while the lower bound essentially follows from the proof of \cite[Theorem~3.2]{Veeser.Zanotti:18} about the quasi-optimality of a nonconforming method. We provide a sketch for the sake of completeness. Thanks to the compatibility condition in \eqref{A:equiv-to-dist-to-conf}, the following extension of $\calA_h$ to $\widetilde{V}$ is well-defined:
\begin{equation*}
		\widetilde{\calA}_h \widetilde{v}
		:=
		v + \calA_h v_h,
		\quad
		\widetilde{v} = v + v_h \text{ with } v \in V, v_h \in V_h.
\end{equation*}
Since $V$ and $V_h$ are closed subspaces of $\widetilde{V}$, the extension $\widetilde{\calA}_h$ is bounded, too; cf.\ \cite[Lemma~3.1]{Veeser.Zanotti:18}, swapping the roles of $V$ and $V_h$.  Given $v_h \in V_h$, we can thus write
\begin{equation*}
		v_h - \calA_h v_h 
		=
		( \mathrm{id}_{\widetilde{V}} - \widetilde{\calA}_h  ) (v_h - v)
\end{equation*}
with an arbitrary $v \in V$ and immediately get
\begin{equation*}
		\norm{ v_h - \calA_h v_h }
		\leq
		\norm{ \mathrm{id}_{\widetilde{V}} - \widetilde{\calA}_h } \inf_{v \in V} \norm{v_h - v}
		=
		\norm{ \mathrm{id}_{\widetilde{V}} - \widetilde{\calA}_h }  \norm{v_h - \Pi_V v_h}.
\end{equation*}
Notably,  \cite[Lemma~3.4]{Veeser.Zanotti:18} shows that $\norm{ \mathrm{id}_{\widetilde{V}} - \widetilde{\calA}_h }^{-1}$ is actually the largest constant $\Cav$ such that the claimed lower bound holds.
\end{proof}

Observe that the upper bound in Proposition \ref{P:inds-for-dist-to-conf}
comes with constant $1$, while the concrete value of $\Cav$ in the
lower bound is left unspecified. We will provide more information on
$\Cav$ in the applications. 

\begin{remark}[Computable realization of the distance to conformity]
\label{R:dist-to-conf-averaging}
An operator $\calA_h$ verifying \eqref{A:equiv-to-dist-to-conf} provides a computable realization of the distance to conformity whenever it is itself computable and each $|\cdot|_T$ can be computationally evaluated on $V_h - \calA_h (V_h)$. Note that, even if $\calA_h$ is a local operator,  the indicators are still of global nature due to their dependence on $u_h$. In the applications below, $\calA_h$ will be an \emph{averaging operator}. 
\end{remark}

%\begin{remark}[Localization of the distance to conformity]
%	\label{R:localization-distance-conformity}
%	In principle, it is desirable that also the distance to conformity of $u_h$ can be split into local contributions, since it enters into our error bounds via Lemma~\ref{L:err-res-decomps} and Lemma~\ref{L:conforming-residual}. Still, we do not investigate this aspect here. Indeed, the technique proposed below would not be helpful to this purpose, because we invoke a partition of unity in $V$ and the distance to conformity of $u_h$ involves a function that is orthogonal to $V$. We rather address the issue case by case via a suitable definition of the quantity $J(u_h)$ from assumption \eqref{A:distance-to-conformity}, cf.\ the examples in sections~\ref{sec:Poisson} and~\ref{sec:Biharmonic}. 
%\end{remark}

%The distance to conformity of $u_h$ can be proven to be equivalent to computationally accessible quantities (like properly scaled jumps) in specific settings, cf. the examples in sections~\ref{sec:Poisson} and~\ref{sec:Biharmonic}. Therefore, we do not further investigate it at this abstract level and simply formulate the following assumption.
%\begin{equation}
%	\label{A:distance-to-conformity}
%	%
%	\stepcounter{AssumptionCounter}
%	\tag{H\arabic{AssumptionCounter}}
%	%
%	\begin{minipage}{0.85\hsize}
%		There are a computationally accessibly quantity $J(u_h)$ and a constant $ C_\Pi > 0$ such that $C_\Pi^{-1} J(u_h)^2 \leq \norm{u_h - \Pi_V u_h}^2 \leq C_\Pi J(u_h)^2$.
%	\end{minipage}
%\end{equation}

%
%
\subsection{Localization of the conforming residual}
\label{sec:local-resid}
%
%
% localization and negative norms
Our next goal is to split the conforming residual norm $\norm[\Dual{V}]{\Res_h^{\mathtt{C}}} = \sup_{v \in V,\norm{v}\leq1} \scp{\Res_h^{\mathtt{C}}}{v}$ into an equivalent `sum' of local counterparts. This step is underlying most techniques. It is complicated by the fact that functions in $V$ and the norm $\norm[]{\cdot}$ typically involve derivatives.
Indeed, first, the partition of unity with characteristic functions underlying \eqref{energy-norm-localization} cannot be directly applied. Instead, as for conforming methods, we use sufficiently smooth partitions of unity to split the test function $v$ into local contributions.
Second, the effect of the cut-off in the norm $\norm\cdot$ has to be stabilized with the help of some orthogonality property of the residual.

% particularities of nonconforming methods
For nonconforming methods, it is convenient to formulate such orthogonality properties with the help of the smoothing operator $\calE_h$. Indeed, the overconsistency of Remark~\ref{R:overconsistent-methods} then leads to exact orthogonality, while, in general, quasi-optimality turns out to be a useful ingredient.

%The a~posteriori analysis of conforming Galerkin methods usually continues with the introduction of some interpolation operator in the residual, which is made possible by Galerkin orthogonality. For the nonconforming discretization \eqref{Eq:DisEq}, we still have a notion of Galerkin orthogonality,
%We use it to introduce the smoothing operator $\calE_h$ into the conforming component of the residual. This step possibly gives rise an additional term, which is kept under control, for quasi-optimal methods, by the consistency measure $\delta_h$. Thus, we arrive at the following upper bound. 

%
\begin{lemma}[Near orthogonality of conforming residual]
\label{L:conforming-residual}
If the discretization \eqref{Eq:DisEq} induces a quasi-optimal method $M_h$, then we have, for all $v \in V$ and $K > 0$,
\begin{equation}
\label{Eq:conforming-residual}
%\begin{split}
 \scp{\Res_h^{\mathtt{C}}}{v}
 \leq 
 \inf_{v_h\in V_h, \norm{v_h} \leq K \norm{v}} \scp{\Res_h^{\mathtt{C}}}{v - \calE_h v_h}
%\\
  + K \delta_h \normtr{A_h} \norm{u_h - \Pi_V u_h} \norm{v}.
% \end{split}
\end{equation}
\end{lemma}

%\begin{remark}[Overconsistent quasi-optimal methods]
%	\label{R:overconsistent-methods}	
%	Overconsistent methods \cite{Veeser.Zanotti:19b} are a class of quasi-optimal methods (including conforming Galerkin methods), characterized by the condition $\delta_h = 0$. For such methods, the second term on the right-hand side of \eqref{Eq:conforming-residual} vanishes, meaning that we can introduce the smoothing operator $\calE_h$ in the conforming part of the residual without giving rise to any additional term. 
%\end{remark}

\begin{proof}
Let $v_h \in V_h$ be such that $\norm{v_h} \leq K \norm{v}$. We add and subtract $\calE_h v_h$ in the action of the conforming %component of the
residual
\begin{equation*}
 \scp{\Res_h^{\mathtt{C}}}{v} 
 =
 \scp{\Res_h^{\mathtt{C}}}{v - \calE_h v_h} 
 +
 \scp{\Res_h^{\mathtt{C}}}{\calE_h v_h}. 
\end{equation*}
For the second term on the right-hand side, we use the method \eqref{Eq:DisEq}, the $\ta$-orthogonal projection $\Pi_V$, and the quasi-optimality of the method through Proposition~\ref{P:QuasiOptMethods}\eqref{QuasiOptMethods-ConsistencyMeasure}, to derive
\begin{equation*}
 \scp{\Res_h^{\mathtt{C}}}{\calE_h v_h} 
 =
 a_h(u_h, v_h) - a(\Pi_V u_h, \calE_h v_h)
 \leq
 \delta_h \norm[\Dual{V_h}]{A_h^* v_h} \norm{u_h-\Pi_V u_h}. 
\end{equation*}
We thus conclude by recalling that the operator norm of the adjoint $A_h^*$ equals the operator norm of $A_h$. 
\end{proof}

Concerning the partition of unity, we assume the following properties:
\begin{equation}
\label{A:partition-of-unity}
\stepcounter{AssumptionCounter}
\tag{H\arabic{AssumptionCounter}}
\begin{minipage}{0.85\hsize}
 There are a finite index set $\Index$, functions $(\Phi_z)_{z\in \Index}$, 
 subspaces $(V_z)_{z \in \Index}$ of $V$, a bounded linear operator $\calI_h: V \to V_h$ and constants $\Cloc,\Ccol>0$ such that
{ \begin{enumerate}[label = (\roman*)]
\item for $z \in \Index$ and $v \in V$, we have $\sum_{z\in \Index} v\Phi_z=v$ with $v\Phi_z\in V_z$,
\item for $v_z \in V_z$, we have
 $\|\sum_{z\in \Index}v_z\|^2\le \Ccol^2 \sum_{z\in\Index}\|v_z\|^2$, and
\item for $v \in V$, we have
$ \sum_{z\in \Index}\norm{(v-\calE_h\calI_h v)\Phi_z}^2\le \Cloc^2\norm{v}^2$. 
\end{enumerate}}
\end{minipage}
\end{equation}
The term `localization'  in the title of this subsection hints to the fact that in applications the functions $\Phi_z$, $z \in \Index$, have compact supports.  We therefore shall refer to the subspaces $V_z$, $z \in \Index$, as \emph{local test spaces}. Properties (ii) and (iii) prescribe some sort of stability of the partition of unity, where the form of (iii) allows countering the effect of the cut-off.

\begin{lemma}[Localizing the conforming residual norm]
\label{L:localizating-conf-res}
If  the discretization \eqref{Eq:DisEq} induces a quasi-optimal method $M_h$, then \eqref{A:partition-of-unity} implies that 
\begin{equation*}
 \frac{1}{\Ccol^2} \sum_{z\in\Index} \norm[\Dual{V}_z]{\Res_h^{\mathtt{C}}}^2
 \le
 \norm[\Dual{V}]{\Res_h^{\mathtt{C}}}^2
 \le
 2 \Cobl^2 \delta_h^2 \norm{ u_h - \Pi_V u_h}^2
 + 
 2 \Cloc^2 \sum_{z\in\Index} \norm[\Dual{V}_z]{\Res_h^{\mathtt{C}}}^2 
\end{equation*} 
holds for some constant $\Cobl \leq \normtr{A_h} \normtr{\calI_h} $ arising through near orthogonality properties of the residual. 
\end{lemma}

Recall that the distance to conformity of the approximate solution $u_h$ is localized in Section~\ref{sec:ind-for-nc-res}. Furthermore, for overconsistent methods, we have $\delta_h=0$ meaning that the upper bound does not mix with the nonconforming residual; in this case, the factor 2 in the constant  $2\Cloc^2$ of the upper bound can be removed.
  
\begin{proof}
We begin by verifying the lower bound. Given $z \in \Index$, let $v_z \in V_z$ be the Riesz representer of $\Res_h^{\mathtt{C}}{}_{|{V_z}}$ in $V_z$ entailing $\norm{v_z} = \norm[\Dual{V_z}]{\Res_h^{\mathtt{C}}}$ and $\scp{\Res_h^{\mathtt{C}}}{v_z} = \norm[\Dual{V_z}]{\Res_h^{\mathtt{C}}}^2$. Defining $v := \sum_{z\in \Index} v_z \in V$, we infer
\begin{equation*}
 \sum_{z\in \Index} \norm[\Dual{V_z}]{\Res_h^{\mathtt{C}}}^2
 =
 \sum_{z\in \Index} \scp{\Res_h^{\mathtt{C}}}{v_z}
 =
 \scp{\Res_h^{\mathtt{C}}}{v}
 \leq
 \norm[\Dual{V}]{\Res_h^{\mathtt{C}}}\norm{v}.
\end{equation*}
Thus, by invoking (ii) in \eqref{A:partition-of-unity}, we obtain the lower bound
\begin{equation*}
 \sum_{z\in \Index} \norm[\Dual{V_z}]{\Res_h^{\mathtt{C}}}^2
 \leq 
 \Ccol^2 \norm[\Dual{V}]{\Res_h^{\mathtt{C}}}^2.
\end{equation*}

To prove the upper bound, let $v \in V$ with $\norm{v} \leq 1$ and set $v_h := \calI_h v \in V_h$. The assumptions (i) and (iii) in \eqref{A:partition-of-unity} provide
\begin{equation*}
 \scp{\Res_h^{\mathtt{C}}}{v - \calE_h v_h}
 =
 \sum_{z\in \Index} \scp{\Res_h^{\mathtt{C}}}{(v - \calE_h \calI_h v)\Phi_z} %\\
 \leq
 \Cloc \left(
  \sum_{z\in \Index} \norm[\Dual{V_z}]{\Res_h^{\mathtt{C}}}^2
 \right)^{\frac{1}{2}}. 
\end{equation*}
We use this estimate in Lemma~\ref{L:conforming-residual} with $K = \normtr{\calI_h}$, then take the supremum over $v$ and conclude the upper bound
\begin{equation*}
 \norm[\Dual{V}]{\Res_h^{\mathtt{C}}}^2
 \leq
 2 \delta_h^2 \normtr{A_h}^2 \normtr{\calI_h}^2 \norm{ u_h - \Pi_V u_h}^2
 +
 2 \Cloc^2 \sum_{z\in \Index} \norm[\Dual{V_z}]{\Res_h^{\mathtt{C}}}^2. \qedhere
\end{equation*}
\end{proof}

\subsection{Computability and conforming residual norm}
\label{sec:comp-and-conf-res-norm}
%
%
% computability problem and approach
Lemma~\ref{L:localizating-conf-res} may suggest to employ the local
quantities $\norm[\Dual{V_z}]{\Res_h^{\mathtt{C}}} =
\norm[\Dual{V_z}]{f-\tA u_h}$, $z \in V_z$, as indicators. However,
these quantities neither are computable nor can be bounded by a
computable expression if we only know that $f$ is taken from
the infinite-dimensional space $\Dual{V}$; cf.\
\cite[Lemma~2]{Kreuzer.Veeser:21}. We therefore adapt the approach in
\cite{Kreuzer.Veeser:21} to the nonconforming methods in
Section~\ref{sec:qo-methods}, thus isolating this computability problem
in data terms.

% overview of the projection approach
More precisely, for each $z \in \Index$, using a suitable \emph{local
projection} $\calP_z$, we split the local conforming residual
$(f-\tA u_h)_{|V_z}$ into two parts with the following properties:
The first part $(\calP_z \Res_h^{\mathtt{C}})_{|V_z}$ is a
local  finite-dimensional approximation of the conforming residual, while
the second part $(\Res_h^{\mathtt{C}}-\calP_z \Res_h^{\mathtt{C}})_{|V_z} =(f - \calP_z f)_{|V_z}$ only depends on data and is of
oscillatory nature. Doing so, $\norm[\Dual{V_z}]{\calP_z \Res_h^{\mathtt{C}}}$ is equivalent to a computable quantity and, therefore, the above computability problem gets
isolated in the second part and may be handled on a case-by-case basis,
depending on the knowledge of $f$. A further advantage of this
splitting is outlined in Remark~\ref{R:conf-res-splitting-and-adaptivity} below. 

% ingredients for constructing the local projections
To construct a local projection $\calP_z$, we discretize the local
test and dual spaces $V_z$ and $\Dual{V_z}$ by respective local,
finite-dimensional spaces $S_z$  and $D_z$.  We shall refer to the
elements of these discrete spaces as \emph{simple (test)
  functions} and \emph{simple functionals}, respectively. The simple test functions
may be locally nonconforming, i.e.\ $S_z \not\subseteq V_z$ is
allowed, but have to be globally conforming, i.e.\ $S_z \subseteq
V$. Similar to \cite{Dari.Duran.Padra:95,Dari.Duran.Padra.Vampa:96},
this additional freedom is exploited in
Theorem~\ref{T:CR-error-estimator-2} below to provide an estimator,
which is simpler than the one in
Theorem~\ref{T:CR-error-estimator-1} based on locally conforming
simple test functions. In view of possibly nonconforming simple test
functions, we consider the local space $\tV_z := V_z + S_z$ and
require $D_z \subseteq \Dual{\tV_z}$. Thus, the simple functionals are
locally 
nonconforming only if there are locally nonconforming simple test functions.
The operator $\calP_z$ will then project $\tV_z'$ onto $D_z$
orthogonal to $S_z$ and we shall build on the following assumption: 
\begin{equation}
\label{A:construction-Pz-local}
\stepcounter{AssumptionCounter}
\tag{H\arabic{AssumptionCounter}}
\begin{minipage}{0.9\hsize}
There are subspaces $(D_z)_{z \in\Index}$, $(S_z)_{z\in\Index}$, and constants $\CL,\CNL > 0$, such that, for all $z \in \Index$, 
\begin{enumerate}[label = (\roman*)]
\item $S_z \subseteq V$ and $D_z \subseteq \Dual{\tV_z} $, where $\tV_z := V_z + S_z$,
\item $\tA (V_h)_{|\tV_z} \subseteq D_z$,
\item $\dim S_z \leq \dim D_z < \infty$, %and \(D_z\) is unisolvent on \(S_z\),
\item for $\chi \in D_z$, we have
 $\norm[\Dual{V_z}]{\chi} \leq \CL \norm[\Dual{S_z}]{\chi}$, 
\item for $g \in \Dual{V}$, we have
 $\sum_{z\in \Index} \norm[\Dual{S_z}]{g}^2
 \leq
 \CNL^2 \sum_{z\in \Index} \norm[\Dual{V_z}]{g}^2$.
\end{enumerate}
\end{minipage}
\end{equation}
If we consider the case of locally conforming simple test functions,
i.e.\ $S_z \subseteq V_z$, the conditions (i) and (ii) are equivalent
to  
\begin{equation}
\label{simplification-for-conf-simple-test-fcts}
  \tA(V_h)_{|V_z} \subseteq D_z \subseteq\Dual{V_z}
\end{equation}
and (v) holds with $\CNL=1$  since \(\norm[\Dual{S_z}]{g}^2 \leq \norm[\Dual{V_z}]{g}^2\).

In any case, the simple functionals have to cover locally the part of
the residual depending on the approximate solution. This is a
crucial structural condition and the remaining conditions (iii)-(v) concern the relationship between simple functionals and simple
test functions in order to ensure well-posedness and stability of the
aforementioned operators $\calP_z$, $z\in \Index$. Condition~(iv) is
equivalent to the inf-sup condition 
\begin{equation}
	\label{Eq:infsup-local}
	\adjustlimits{\inf}_{\chi \in D_z}{\sup}_{s \in S_z}
	\dfrac{\scp{\chi}{s}}{\norm[\Dual{V_z}]{\chi}\norm{s}}
	\geq \dfrac{1}{\CL} > 0 
\end{equation} 
and implies $\dim D_z \leq \dim S_z$. Consequently, a necessary condition for \eqref{A:construction-Pz-local} is $\dim S_z = \dim D_z$. In Sections \ref{sec:Poisson} and \ref{sec:Biharmonic} below, we exemplify how to choose simple functionals and local test functions in concrete settings.
 
The next lemma establishes the local projections, along with their stability properties and their interplay with the extended %differential
operator $\tA$ and the discrete space $V_h$.

\begin{proposition}[Local projections]
\label{P:construction-Pz-local}
Suppose assumption~\eqref{A:construction-Pz-local} holds. Then the variational equations
\begin{equation}
\label{Eq:construction-Pz-local}
% \text{find }\calP_z g \in D_z \text{ s.th.\ }
 \forall s \in S_z \qquad \scp{\calP_z g}{s} = \scp{g}{s},
\quad
 z \in \Index,
\end{equation}
define linear projections $\calP_z: \Dual{\tV_z} \to D_z$, satisfying the invariances
\begin{subequations}
\begin{align}
\label{Pz-invariance}
 \forall z \in \Index, v_h\in V_h
 \qquad
 \big( \calP_z\tA v_h \big)_{|V_z} = \big( \tA v_h \big)_{|V_z}
\intertext{and the collective stability bound}
\label{Eq:boundedness-Pz-local}
 \forall g \in \Dual{V}
\qquad
 \sum_{z\in\Index}\norm[\Dual{V_z}]{\calP_z g}^2
 \le
 \CL^2 \CNL^2 \sum_{z\in\Index}\norm[\Dual{V_z}]{g}^2.
\end{align}
Hereafter, we also write simply $\calP_z g$ instead of $\calP_z (g_{|\Dual{\tV_z}})$ for $g \in \Dual{V}$.
\end{subequations}
\end{proposition}

\begin{proof}
% Pz well-defined
Fix any $z \in \Index$. Thanks to the rank-nullity theorem and condition~(iii) in \eqref{A:construction-Pz-local}, the well-posedness of the linear variational problem for $\calP_z$ follows from its uniqueness. To show the latter,  let $\chi \in D_z$ with  $\scp{\chi }{s} = 0$ for all $s \in S_z$, that is \ $\chi_{|S_z} = 0$. Condition~(iv) in \eqref{A:construction-Pz-local} then implies $\chi_{|V_z} = 0$. By its linearity, $\chi$ vanishes on the sum $V_z + S_z$, meaning $\chi = 0$ in the light of (i) in \eqref{A:construction-Pz-local}. Consequently, \eqref{Eq:construction-Pz-local} has a unique solution in $D_z$ and $\calP_z$ is the linear projection from $\Dual{\tV_z}$ onto $D_z$ with the orthogonality property
\begin{equation}
\label{Pz-orthogonality}
 \forall s \in S_z \qquad \scp{g - \calP_z g}{s} = 0.
\end{equation}
% invariance for \tA(V_h)
The invariance \eqref{Pz-invariance} then follows from the observation that, for \(v_h\in V_h\), the difference $\calP_z\tA v_h - (\tA v_h)_{|\tV_z}$ is in the image $D_z$ of $\calP_z$ thanks to (ii) in \eqref{A:construction-Pz-local}.
	
% stability
Finally, given any $g \in \Dual{V}$, the collective stability bound  \eqref{Eq:boundedness-Pz-local} readily follows by combining (iv) and (v) in \eqref{A:construction-Pz-local} and the orthogonality \eqref{Pz-orthogonality} of $\calP_z$:	
\begin{equation*}
 \sum_{z \in \Index} \norm[\Dual{V_z}]{\calP_z g}^2 
 \leq
 \CL^2  \sum_{z \in \Index} \norm[\Dual{S_z}]{\calP_z g}^2
 =
 \CL^2  \sum_{z \in \Index} \norm[\Dual{S_z}]{g}^2
 \leq
 \CL^2 \CNL^2  \sum_{z \in \Index} \norm[\Dual{V_z}]{g}^2.\qedhere
\end{equation*}
\end{proof}

The use of local projections is motivated by difficulties with the computability of the local quantities $\norm[\Dual{V_z}]{\Res_h^{\mathtt{C}}}$. The next remark is therefore of primary interest.

\begin{remark}[Computing the local projections]
\label{R:computing-local-projections}
Let $z \in \Index$ and $g \in \Dual{V}$. From an algebraic viewpoint,
the projection $\calP_z g$ from \eqref{Eq:construction-Pz-local} can
be computed by solving a linear system of equations. Moreover, we have
for the $\Dual{S_z}$-norm of $\calP_z g$ that
\begin{equation*}
 \|\calP_z g\|_{S_z'}^2=\mathbf{g}_z^T \mathbf{A}_z^{-1} \mathbf{g}_z,
\end{equation*}
where $\mathbf{g}_z$ is the column vector representing the action of $g$ on some basis of $S_z$ and $\mathbf{A}_z$ is the matrix representing the extended scalar product $\ta$ on $S_z \times S_z$ with respect to the same basis.
\end{remark}

The reader may skip the following two remarks as they are not used in what follows. They, however, give instructive information on the approximation qualities of the local projections and how they are achieved.

\begin{remark}[Near-best approximation of the local projections]
\label{R:qopt-of-Pzs}
The local projections provide the collective near-best approximation
\begin{equation}
\label{Pz_qopt}
 \sum_{z \in \Index} \norm[\Dual{V_z}]{g - \calP_z g}^2
 \leq
 (2+ 2\CL^2 \CNL^2)
 \inf_{\chi \in \tA(V_h)} \sum_{z \in \Index} \norm[\Dual{V_z}]{g - \chi}^2
\end{equation}
for all $g \in \Dual{V}$. This follows from the local invariances \eqref{Pz-invariance} for any fixed $\chi \in \tA(V_h)$, the linearity of the local projections, and their collective stability \eqref{Eq:boundedness-Pz-local}.
\end{remark}

\begin{remark}[Local projections as nonconforming approximations]
\label{R:nonconformity_of_Pz}
%
% nonconforming discretization and qo-nconf methods
Each projection $\calP_z$ can be viewed as a \emph{nonconforming Petrov-Galerkin approximation} of the identity $\mathrm{id}_{\Dual{V_z}}$ on $\Dual{V_z}$. Equally, in Section~\ref{sec:qo-methods}, $M_h$ is such an approximation of $A^{-1}$. It is therefore in order to compare the approach based upon \eqref{A:construction-Pz-local} with the one in \cite{Veeser.Zanotti:18} used in Section~\ref{sec:qo-methods}. To this end, we note that (ii) of \eqref{A:construction-Pz-local} is not of interest in this context, and ignore that \cite{Veeser.Zanotti:18} does not properly cover the collective or product setting with different continuous trial and test spaces of the local projections.
% Furthermore, we observe that the continuous operator $\mathrm{id}_{\Dual{V_z}}$ assigns a double role to $\calP_z$: it is both method and approximation operator; cf.\ \eqref{Eq:approximation-operator}. 

% common feature
A common feature of both approaches is that they invoke the same type of extended spaces. Indeed, the model $\tV = V + V_h$ is repeated in the context of \eqref{A:construction-Pz-local}, on the test side, by $\tV_z = V_z + S_z$, which we may write also as $\tV_z = V_z \oplus S_z^\mathrm{NC}$, where $S_z^\mathrm{NC}$ is a complement of $S_z \cap V_z$ in $S_z$, i.e.\  we have $S_z = (S_z \cap V_z) \oplus S_z^\mathrm{NC}$.
Similarly, for the more involved trial side, we can view $\Dual{\tV_z}$ as the direct sum of $\Dual{V_z}$ and the nonconforming part $(D_z{}_{|S_z^\mathrm{NC}})$ of $D_z$. In fact, we have $\Dual{\tV_z} = \Dual{ (V_z \oplus S_z^\mathrm{NC}) } \eqsim \Dual{V_z} \oplus \Dual{(S_z^\mathrm{NC})} \eqsim \Dual{V_z} \oplus (D_z{}_{|S_z^\mathrm{NC}})$  and $D_z \eqsim (D_z{}_{|V_z}) \oplus (D_z{}_{|S_z^\mathrm{NC}})$.

% without smoother but quasi-optimal
There is, however, an important and, at the first look, perhaps striking difference. The nonconforming discrete problems \eqref{Eq:construction-Pz-local} do not invoke smoothers but exploit simple restriction. Recall that classical nonconforming methods relying on simple restriction cannot be quasi-optimal, i.e.\ cannot provide near-best approximation; cf.\ \cite{Veeser.Zanotti:18}. Nevertheless, the use of restriction in \eqref{Eq:construction-Pz-local} does not obstruct the collective near-best approximation \eqref{Pz_qopt} thanks to the fact that all simple test functions are globally conforming. In addition, restriction is preferable over smoothing in \eqref{Eq:construction-Pz-local} as it simplifies the implementation of the local projections. The global conformity of all simple test functions implies also that each extended space $\Dual{\tV_z}$, like its counterpart $\Dual{V_z}$, can be viewed as a subspace of  $\Dual{V}$, meaning that the simple functionals are globally conforming.  Hence, the extension of the local projections to the extended local spaces is `built-in', while the extension of the approximation operator $P_h$ in \eqref{Eq:approximation-operator} has to be constructed.
\end{remark}

The next lemma applies the local projections in order to obtain the
announced splitting of the conforming residual. Therein, an indicator
$\eta_h$ is \emph{computationally $C$-quantifiable} whenever it is
equivalent to a computable quantity $\bar{\eta}_h$ up to a
multiplicative constant $C>0$, i.e., for all $f \in \Dual{V}$, we have $C^{-1} \bar{\eta}_h \leq \eta_h \leq C \bar{\eta}_h$.

\begin{lemma}[Conforming residual splitting]
\label{L:splitting-of-conf-res}
Assumption~\eqref{A:construction-Pz-local} ensures
\begin{equation*}
 \Csplt^2 \left( \eta_h^2 + \Osc_h^2 \right)
 \leq
 \sum_{z\in\Index} \norm[\Dual{V}_z]{\Res_h^{\mathtt{C}}}^2
 \leq
 2 \left( \eta_h^2 + \Osc_h^2 \right)
\end{equation*}
with some $\Csplt \geq (1 + 2\CL^2\CNL^2)^{-1/2}$ and the indicators
\begin{equation}
\label{Eq:indicator-oscillation-local}
 \eta_h^2
 :=
 \sum_{z\in\Index} \norm[\Dual{V_z}]{\calP_z \Res_h^{\mathtt{C}}}^2  
\qquad \text{and} \qquad
 \Osc_h^2
 :=
 \sum_{z\in\Index} \norm[\Dual{V_z}]{f - \calP_z f}^2,
\end{equation}
where the indicator $\eta_h$ is computationally quantifiable up to $\max\{\CL,\CNL^{-1}\}$ whenever $\Res_h^{\mathtt{C}} = f - \tA u_h$ can be computationally evaluated for all simple test functions.
\end{lemma}

\begin{proof}
The upper bound follows from the triangle inequality and the special case $(\calP_z \tA u_h){}_{|V_z} = \tA u_h{}_{|V_z}$ of the invariance \eqref{Pz-invariance}:
\begin{equation}
\label{danger-of-overest}
 \norm[\Dual{V_z}]{\Res_h^{\mathtt{C}}}
 \le
 \norm[\Dual{V_z}]{\calP_z\Res_h^{\mathtt{C}}}
  +
 \norm[\Dual{V_z}]{\Res_h^{\mathtt{C}}-\calP_z \Res_h^{\mathtt{C}}}
 =
 \norm[\Dual{V_z}]{\calP_z\Res_h^{\mathtt{C}}}
 +
 \norm[\Dual{V_z}]{f-\calP_z f}.
\end{equation}
The lower bound follows from using the last equality in the opposite direction and applying the collective stability~\eqref{Eq:boundedness-Pz-local} % and near-best approximation \eqref{Pz_qopt} 
of the local projections,
\begin{equation*}
 \sum_{z\in \Index}\norm[\Dual{V_z}]{\calP_z \Res_h^{\mathtt{C}}}^2
 +
 \norm[\Dual{V_z}]{f - \calP_z f}^2
 \leq
 (2 + 3\CL^2\CNL^2)
 \sum_{z\in \Index}\norm[\Dual{V}_z]{\Res_h^{\mathtt{C}}}^2. 
\end{equation*}

To show the quantifiability of $\eta_h$, observe that $\bar{\eta}_h^2 := \sum_{z \in \Index} \norm[\Dual{S_z}]{\calP_z \Res_h^{\mathtt{C}}}^2$ is computable thanks to Remark~\ref{R:computing-local-projections} and the assumed knowledge on $\Res_h^{\texttt{C}}$. Furthermore, it is equivalent to $\eta_h$ thanks to (iv) and (v) of \eqref{A:construction-Pz-local}.
\end{proof}

It is important to note that, thanks to the structural condition (ii) in \eqref{A:construction-Pz-local}, the oscillation $\Osc_h$ is not just of oscillatory nature, but a \emph{data oscillation}, i.e., in the context of  \eqref{Eq:AnaPro}, it depends only on the data $f$ and not on the discrete solution $u_h$.

\begin{remark}[Improving $\Csplt$]
\label{R:constant_in_low_bd}
In the analysis \cite[Section~4]{Bonito.Canuto.Nochetto.Veeser:24} for a conforming method, one has the counterpart of the improved inequality
\begin{equation*}
 \Csplt \geq \frac{1}{\sqrt{2} \CL \CNL}
\end{equation*}
thanks to  \cite[Lemma~2.1]{Szyld:06}. For the corresponding case when
all simple test functions are locally conforming, i.e.\ $S_z \subseteq
V_z$ for all $z \in \Index$, this improvement follows from applying
the cited lemma on each projection $\calP_z$. Similarly, it follows
also if the local projections arise from a global one, viz.\ there is
a projection on $P_h$ on $\Dual{V}$ such that $P_z{}_{|V_z} =
P_h{}_{|V_z}$ for all $z \in \Index$. In fact, we can then apply the
cited lemma on $P_h$, endowing $\Dual{V}$ with the square root of
$\sum_{z \in \Index} \norm[\Dual{V_z}]{\cdot}^2$. 
\end{remark}

\subsection{Deriving a strictly equivalent error estimator}
\label{sec:err-est}
We now combine the various ingredients and previous results to obtain
the main result of our abstract a~posteriori analysis: a guideline to
derive an estimator bounding the error from above and below without
any additional terms. 
\begin{theorem}[Abstract estimator]
\label{T:error-estimator}
Let $M_h$ be a quasi-optimal method, induced by the discretization
\eqref{Eq:DisEq}. Suppose that \eqref{A:equiv-to-dist-to-conf},
\eqref{A:partition-of-unity}, and \eqref{A:construction-Pz-local} hold
and, given tuning constants $C_1, C_2 >0$, define an abstract
estimator by 
\begin{multline*}
  \Est_h^2
  :=
  (1 + C_1^2 \delta_h^2) \Ncf_h^2 + C_2^2 (\eta_h^2 + \Osc_h^2),  \text{ where }
\\
\begin{aligned}
 \Ncf_h^2
 &:= 
  \norm{ u_h - \calA_h u_h }^2,  &&\text{(approximate distance to conformity)}
\\
 \eta_h^2
  &:= 
 \textstyle \sum_{z\in\Index} \norm[\Dual{V_z}]{\calP_z f- \tA u_h}^2 ,  &&\text{(approximate conforming residual)}
\\
 \Osc_h^2
 &:= 
 \textstyle \sum_{z\in\Index} \norm[\Dual{V_z}]{f - \calP_z f}^2, &&\text{(data oscillation)}
\end{aligned}
\end{multline*}
and the projections $\calP_z$ %_{z \in \Index}$
are defined by \eqref{Eq:construction-Pz-local}. Then this estimator quantifies the error by
\begin{equation*}
%\label{Eq:error-estimator-local}
 \underline{C} \Est_h
 \leq
 \norm{u-u_h}
 \leq
 \overline{C}\Est_h
\end{equation*}
where the equivalence constants $\underline{C}$, $\overline{C}$ depend only on the norm $\normtr{A_h}$ of the discrete operator $A_h$, the constants arising from the assumptions \eqref{A:equiv-to-dist-to-conf}, \eqref{A:partition-of-unity}, and \eqref{A:construction-Pz-local}, and the tuning constants $C_1$, $C_2$.
\end{theorem} 

\begin{proof}
We first show the upper bound. Combining Lemmas~\ref{L:err-res-decomps}, \ref{L:localizating-conf-res}, \ref{L:splitting-of-conf-res}, and Proposition~\ref{P:inds-for-dist-to-conf} leads to
\begin{align*}
 \norm{u-u_h}^2
  &=
  \norm[\tV']{\Res_h}^2
  =
  \norm{u_h - \Pi_V u_h}^2  + \norm[V']{\Res_h^\mathtt{C}}^2 
\\
  &\le
   (1 + 2  \Cobl^2 \delta_h^2 ) \Ncf_h^2 + 2 \Cloc^2 \sum_{z \in \Index} \norm[\Dual{V_z}]{\Res_h^\mathtt{C}}^2
\\
  &\le
  (1 + 2  \Cobl^2 \delta_h^2 ) \Ncf_h^2 + 4 \Cloc^2 (\eta_h^2 + \Osc_h^2).
\end{align*}
This implies the upper bound for some
\begin{equation*}
		\overline{C}
		\leq
		\max\left\{ 1, \frac{\sqrt{2}\Cobl}{C_1}, \frac{2\Cloc}{C_2} \right\}.
\end{equation*}
Regarding the lower bound, we apply the same auxiliary results to derive
\begin{align*}
 \Est_h^2
 &\leq
 (1 + C_1^2 \delta_h^2) \Cav^{-2} \norm{u_h - \Pi_V u_h}^2
 +
 C_2^2  \Csplt^{-2} \Ccol^2 \norm[\Dual{V}]{\Res_h^\mathtt{C}}^2
\\
 &\leq
 \max \big\{ (1 + C_1^2 \delta_h^2) \Cav^{-2}, C_2^2  \Csplt^{-2} \Ccol^2 \big\}
  \norm{u-u_h}^2,
\end{align*}
whence the lower bound holds for some
\begin{equation*}
	\underline{C}
	\geq
	\max \big\{ (1 + C_1^2 \delta_h^2) \Cav^{-2}, C_2^2  \Csplt^{-2} \Ccol^2 \big\}^{-1}.\qedhere
\end{equation*}	
\end{proof}

Several remarks about the properties of the estimator $\Est_h$ in
Theorem~\ref{T:error-estimator} are in order. We start by discussing
the computability of $\Est_h$, compare
with typical results from literature, highlight its features for
adaptivity,  and conclude with a remark about how to choose the simple
functionals. 

\begin{remark}[On computability]
\label{R:computability}
In Remark~\ref{R:dist-to-conf-averaging} and
Lemma~\ref{L:splitting-of-conf-res}, we have already observed that,
under mild assumptions,  the indicators $\Ncf_h$ and $\eta_h$ are
computable (in exact arithmetic), exploiting their finite-dimensional
nature. The remaining oscillation indicator $\Osc_h$ cannot be
computationally bounded, not to mention computed, in general, viz.\
knowing only that the data $f $ is from the infinite-dimensional space
$\Dual{V}$; cf.\ \cite[Lemma~2 and Corollary~5]{Kreuzer.Veeser:21}. We
may neglect it (see Remark~\ref{R:neglecting-data-osc} for a
discussion of this option) or resort to some approximation of $\Osc_h$,
which is derived case-by-case, possibly exploiting additional
information on the regularity and/or structure of $f$. The
significance of Theorem~\ref{T:error-estimator} then of course hinges
on the quality of those approximations. In particular, to facilitate
the approximation, it may be convenient to waive the lower bound with
the data oscillation $\Osc_h$ and replace it by a so-called
\emph{surrogate oscillation}, i.e.\ a quantity that provides an upper
bound of $\Osc_h$ and can be considered computable; cf.\
\cite[Section~7.1]{Cohen.DeVore.Nochetto:12}.  
\end{remark}

\begin{remark}[Neglecting data oscillation?]
\label{R:neglecting-data-osc}
Let  us compare the splitting of the conforming residual in
Lemma~\ref{L:splitting-of-conf-res} to the following formula in
numerical integration: 
\begin{equation*}
\label{mid-point-with-error}
	\int_0^1 v 
	=
	v(\tfrac{1}{2}) + \int_0^1 v(x) - v(\tfrac{1}{2}) \, dx.
\end{equation*}
Similarly to Lemma~\ref{L:splitting-of-conf-res} combined with
Remark~\ref{R:computing-local-projections}, but with an equality
instead of an equivalence, this formula relates an exact
non-computable value, a computable approximation and a representation
of the error of the approximation. Since in numerical integration the
error is often not considered on a computational level,  one may be
tempted to also neglect the oscillation $\Osc_h$. This approach
amounts to use in computations simply the estimator 
\begin{equation*}
	\widetilde{\mathrm{est}}{}_h^2
	:=
	(1 + C_1^2 \delta_h^2) \Ncf_h^2 + C_2^2  \eta_h^2.
\end{equation*}
Note that this estimator depends on the data $f$ only through the discrete solution $u_h$ and the local projections $\calP_z$, $z \in \Index$. It therefore cannot see components of $f$ orthogonal to $\calE_h(V_h) + \bigcup_{z\in\Index} S_z$, which may constitute the main part of the error; cf.\ \cite[Section~5.3.1]{Bonito.Canuto.Nochetto.Veeser:24}. Hence, a reliable use of $	\widetilde{\mathrm{est}}{}_h$ 
requires to verify a smallness assumption of the type $\Osc_h \ll
\widetilde{\mathrm{est}}_h$. Apart from the fact that such a smallness
condition would a~priori restrict the admissible data, it is not clear
to us that its verification is a simpler task than the alternatives
outlined in Remark~\ref{R:computability}. Notice also that, in an
adaptive context, the estimator $\widetilde{\mathrm{est}}{}_h$ does
not provide any information to adapt  to the aforementioned orthogonal
components of $f$ and the smallness assumption will be needed for the
initial and all subsequent meshes.  
\end{remark}

\begin{remark}[Strict equivalence and error-dominated oscillation]
\label{R:strict-equivalence}
The strict equivalence between error and the estimator $\Est_h$ and
the related error-dominated oscillation $\Osc_h$ are key features of
Theorem~\ref{T:error-estimator}. In fact, if we derive a~posteriori
bounds  for \eqref{Eq:AnaPro} and \eqref{Eq:DisEq} by means of
classical techniques,  the equivalence between error and estimator  is
spoiled. 

% defect of classical approach
To explain this defect in more detail, denote the classical estimator and oscillation
by  $\widehat{\Est}_h$ and $\widehat{\Osc}_h$, respectively. The classical oscillation requires $f \in L^2(\Omega)
\subsetneq\Dual{V}$ or other regularity assumptions beyond $\Dual{V}$. The
modified equivalence then reads 
\begin{equation}
\label{spoiled-equivalence}	
  \norm{u-u_h}  \leq C_1 \big( \widehat{\Est}_h + \widehat{\Osc}_h \big)
 \quad\text{and}\quad
\widehat{\Est}_h \leq C_1 \big( \norm{u-u_h} + \widehat{\Osc}_h \big),
\end{equation}
where, depending on the estimator type, $\widehat{\Osc}_h$ may be not
present in one of the bounds. Since the classical oscillation
$\widehat{\Osc}_h$ does not vanish whenever the error does, a bound of
the type $\widehat{\Osc}_h \leq C \norm{u-u_h}$ does not
hold. Furthermore,  the fact that the classical oscillation $\widehat{\Osc}_h$ is
formally of higher order does not exclude that, under adaptive refinement,
it converges slower than the error in certain cases; cf.\ 
\cite[Section~6.4]{Cohen.DeVore.Nochetto:12} and
\cite[Section~3.8]{Kreuzer.Veeser:21}. Consequently, a strict
equivalence cannot be obtained from \eqref{spoiled-equivalence} and
the defect may not mitigate under adaptive refinement. In contrast to
this, Theorem~\ref{T:error-estimator} ensures $\Osc_h \leq C
\norm{u-u_h}$ for some constant $C$. 

% new twist
The crucial step for this improvement is the splitting of the
conforming residual by means of the triangle inequality in
\eqref{danger-of-overest}. Therein, the splitting is realized with the
local projections that  are collectively stable without changing the
involved norm thanks to conditions (iv) and (v) of
\eqref{A:construction-Pz-local}. Classical techniques split,
explicitly or  implicitly, the conforming residual in the upper and/or
lower bound in a less stable manner, passing to some norm that is strictly stronger than $\norm[\Dual{V}]{\cdot}$.
This  difference in the stability of the splitting of the
conforming residual is the new twist. In view of the discussion in
\cite[Remark~4.19]{Bonito.Canuto.Nochetto.Veeser:24} and
\cite{Kreuzer.Veeser.Zanotti:24}, it is also a necessary ingredient
for the strict equivalence. 
\end{remark} 

\begin{remark}[Classical oscillation as surrogate]
\label{R:class-osc-as-surrogate}
A classical oscillation $\widehat{\Osc}_h$  from
Remark~\ref{R:strict-equivalence} may be used as a surrogate
oscillation, as defined in Remark~\ref{R:computability}; see Theorems~\ref{T:CR-error-estimator-1}, \ref{T:CR-error-estimator-2}, \ref{T:dG-error-estimator}, \ref{TH:mainC0}, and \ref{TH:mainMR} below.  In this context, it is
worth mentioning that the application of
Theorem~\ref{T:error-estimator} is nevertheless advantageous.  Indeed,
Theorem~\ref{T:error-estimator} ensures the lower bound $\eta_h \leq C
\norm{u-u_h}$ with some constant $C$, while, for certain classical
techniques, the corresponding lower bound gets spoiled, i.e.\ one has
only $\widehat{\eta}_h \leq C (\norm{u-u_h} + \widehat{\Osc}_h)$ for
the counterpart $\widehat{\eta}_h$ of $\eta_h$, with the disadvantages
outlined in Remark~\ref{R:strict-equivalence};
cf. \cite[Remark~4.43]{Bonito.Canuto.Nochetto.Veeser:24}. 
\end{remark}

\begin{remark}[Conforming residual splitting and adaptivity]
\label{R:conf-res-splitting-and-adaptivity}
We briefly discuss a further advantage of the splitting of the conforming residual in an adaptive context. The quantity $\sum_{z \in \Index} \norm[\Dual{V_z}]{\Res_h^\mathtt{C}}^2 = \sum_{z \in \Index} \norm[\Dual{V_z}]{f - \tA u_h}^2$ is of infinite-dimensional and, although localized, of global nature. The infinite dimension arises through the data $f$, while the global dependence through the approximate solution $u_h$.  These features lead to complications  in adaptivity. Roughly speaking, infinite dimension in absence of additional information entails that a given pattern for local refinement cannot be expected to uniformly reduce the error, while global dependence complicates the behavior of indicators under local refinement. After the splitting in Lemma~\ref{L:splitting-of-conf-res}, the indicator $\eta_h$ is of finite-dimensional but global nature and the data oscillation $\Osc_h$ of  local, but infinite-dimensional  nature. In other words: the two difficulties get separated by this splitting. This separation, which relies on the structural condition (ii) in \eqref{A:construction-Pz-local}, can be exploited in the design and analysis of adaptive algorithms; cf., e.g., \cite[Section~5]{Bonito.Canuto.Nochetto.Veeser:24}. 
\end{remark}

\begin{remark}[Choosing the simple functionals]
\label{R:choosing-simple-fcts}
The crucial structural condition (ii) of \eqref{A:construction-Pz-local} prescribes a minimal size for the sets $D_z$, $z \in \Index$, of simple functionals. Furthermore, in the light of the computational issues with $\Osc_h$, it appears convenient that its formal convergence order is larger than the one for the error and that it is smaller than available classical oscillations. It will then play, at least asymptotically for sufficiently smooth sources $f$, a less important role. Of course, the larger the sets $D_z$, $z \in \Index$, the greater the computational cost of the estimator.  All this suggests to take the simple functionals not much larger than necessary for the aforementioned three conditions.
\end{remark}

\section{Estimators for the Poisson problem}
\label{sec:Poisson}
In this section, we first exemplify the application of the guidelines developed Section~\ref{sec:posteriori-analysis} in a simple case and then start to illustrate its generality. For the first purpose, we consider the Poisson problem
\begin{align}\label{Eq:Poisson}
	-\Delta u= f\quad\text{in}~\Omega
\qquad\text{and}\qquad
	u=0\quad\text{on}~\partial\Omega,
\end{align}
together with a first-order Crouzeix-Raviart method, which is quasi-optimal and overconsistent. Covering the use of locally conforming and locally nonconforming simple test functions, we derive and discuss two different strictly equivalent a~posteriori error estimators, highlighting various aspects of the abstract framework.
 
To illustrate the generality of the guidelines, we then consider quasi-optimal discontinuous Galerkin methods of arbitrary but fixed order for the second-order problem \eqref{Eq:Poisson}. This will be complemented in Section~\ref{sec:Biharmonic} by dealing with nonconforming methods for a fourth-order problem.

\subsection{Domain, mesh, and polynomials}
\label{sec:doma-mesh-polyn}
Let $\Omega$ be  a polyhedral Lipschitz domain $\Omega\subset \R^d$, \(d\ge 2\). We denote by $\normal$ the outward pointing unit normal vector field of $\partial\Omega$ and, correspondingly, $\partial_{\normal} v := \nabla v \cdot \normal$ stands for the normal derivative of a sufficiently smooth function $v$.

% notation related to the mesh
Let $\grid$ be a simplicial face-to-face (conforming) mesh of the domain $\Omega$. We denote, respectively, the set of its faces, of its interior faces, of its vertices, and of its interior vertices by $\sides$, $\sides^i$, $\vertices$, and $\vertices^i$. The mesh-size function $h:\Omega \to
[0,\infty)$ of $\grid$ is given by 
\begin{equation*}
\label{Eq:meshsize}
	h_{|\mathrm{int}(T)} := h_T \quad \text{for $T \in \grid$},
\qquad \text{and} \qquad 
	h_{|F} := h_F, \quad \text{for $F \in \sides$},
\end{equation*}
where $h_T$ and $h_F$ are the respective diameters of $T$ and $F$. We extend the normal field $\normal$ on the skeleton $\Sigma := \bigcup_{F \in \sides} F$ by setting
\begin{equation*}
\label{Eq:normal-field}
	\normal_{|F} := n_{F} \quad \text{for $F \in \sides^i$},
\end{equation*}
where $n_{F}$ is an arbitrary but fixed unit vector normal to the
interior face $F$. The patches around a simplex $T \in \grid$, a face
$F \in \sides$ and a vertex $z \in \vertices$ are the respective sets 
\begin{equation*}
\label{Eq:patches}
\begin{gathered}
	\omega_T := \bigcup_{T' \in \grid, T' \cap T \neq \emptyset} T',
\qquad
	\omega_F :=  \bigcup_{T' \in \grid, F \subseteq \partial T'} T',
\quad
 \omega_z := \bigcup_{T' \in \grid, z \in T'} T',
\\
\quad\text{and}\quad
 \omega_z^+ := \bigcup_{T' \in \grid, T ' \cap \omega_z \neq \emptyset} T'.
\end{gathered}
\end{equation*}
% where the unions are taken over all simplices $T' \in \grid$
% satisfying the respective conditions. 
For convenience, we introduce a dedicated notation also for diameter, induced mesh and set of interior faces of the above patches around a vertex, namely 
\begin{equation*}
\label{Eq:patch-z}
\begin{gathered}
	h_z := \mathrm{diam}(\omega_z), \qquad
	\grid_z := \{T \in \grid \mid z \in T\}, \qquad
	\sides_z^i := \{F \in \sides^i \mid z \in F\},
\\
  \grid_z^+ := \{T \in \grid \mid T \cap \omega_z \neq \emptyset \}, \qquad
  \sides_z^{i+} := \{F \in \sides^i \mid F \cap \omega_z \neq \emptyset \}.
 \end{gathered}
\end{equation*}

% Second, notation related to (piecewise) polynomials
For $\ell \in \mathbb N_0$, we denote by $\poly_\ell(T)$ and $\poly_\ell(F)$
the sets of polynomials with total degree not larger than $\ell$ on a
simplex $T \in \grid$ and a face $F \in \sides$, respectively. For
$\ell < 0$, we use the convention $\poly_\ell(T) := \{0\} =:
\poly_\ell(F)$. Given $k \in \mathbb{N}_0$, we consider the following
sets of piecewise polynomials on $\grid$ 
\begin{equation*}
	\label{Eq:piecewise-polynomials}
	S^k_\ell := \{v \in H^k(\Omega) \mid \forall T \in \grid \quad v_{|T} \in \poly_\ell(T)\}
	\qquad \text{and} \qquad
	\mathring{S}^k_\ell := S^k_\ell \cap \mathring{H}^k(\Omega),
\end{equation*} 
with the convention $H^0(\Omega) := L^2(\Omega) =: \mathring H^0(\Omega)$. We
make use of the Courant basis functions $\{\Psi_z\}_{z \in
  \vertices}$ of the first-order space $S^1_1$, which are determined by the condition
\begin{equation}
	\label{Eq:nodal-basis-S11}
	\Psi_z(y)=\delta_{zy},\quad y,z\in\vertices,
\end{equation}
and of the %following
bubble functions %$\Psi_T \in \mathring{S}^1_{d+1}$ and $\Psi_F \in S^1_d$
\begin{equation}
	\label{Eq:bubbles}
	\Psi_T:=
	%\frac{(2d+1)!}{d!|T|}
	\prod_{z\in\vertices,z\in T}\Psi_z  \in \mathring{S}^1_{d+1}
	\qquad\text{and}\qquad
	\Psi_F :=
	%\frac{(2d-1)!}{(d-1)!|F|}
	\prod_{z\in\vertices,z\in F}\Psi_z \in S^1_d
\end{equation}
associated with each element $T \in \grid$ and face $F \in \sides$. 
 
% Third: jumps and averages
For a piecewise smooth function $v:\Omega \to \R$, the jump
$\jump{v}:\Sigma \to \R$ of $v$ is defined on the skeleton $\Sigma$ of
$\grid$ as 
\begin{equation*}
	\jump{v}_{|F} := v_{|T_1} - v_{|T_2} \quad
	\text{for $F\in\sides^i$},
	\qquad\text{and}\qquad
	\jump{v}_{|F} := v_{|T}\quad\text{for $F\in \sides \setminus \sides^i$},
\end{equation*}
where, for $F\in\sides^i$, the simplices $T_1, T_2 \in \grid$ are such that $T_1 \cap T_2 = F$ and $\normal_{F}$ points from $T_1$ to $T_2$, and for $F \in \sides \setminus \sides^i$, $T \in \grid$ is such that $T \supset F$.
Similarly, the average $\mean{v}:\Sigma \to \R$ of $v$ is
defined as 
\begin{equation*}
	\mean{v}_{|F} := \frac{v_{|T_1} + v_{|T_2}}{2}\quad
	\text{on $F\in\sides^i$},
	\qquad\text{and}\qquad
	\mean{v}_{|F} := v_{|T}\quad\text{on $F\in \sides \setminus \sides^i$}.
\end{equation*}

% Finally: hidden constants
Finally, we write \(a\lesssim b\) if there exists a positive constant \(c\) such that
\(a\le c\,b\) and \(a\eqsim b\) when in addition \(b\lesssim a\).
The hidden constants possibly depend on the space dimension $d$, the
polynomial degree $\ell$ and  the shape constant 
\begin{align}\label{Eq:shape-const}
	\gamma_\grid:=\max_{T\in\grid}\frac{h_T}{\rho_T},
\end{align}
of the mesh $\grid$, where $\rho_T$ is the diameter of the largest ball in $T$. 

\subsection{Crouzeix-Raviart method as model example}
\label{sec:CR1}
%
% intro
This section is devoted to the a~posteriori analysis of  the quasi-optimal Crouzeix-Raviart method from \cite[section~3.2]{Veeser.Zanotti:19b}. Exemplifying the guidelines in Section~\ref{sec:posteriori-analysis}, two strictly equivalent error estimators of hierarchical type are derived in Theorems~\ref{T:CR-error-estimator-1} and~\ref{T:CR-error-estimator-2}. The first estimator involves face-bubble functions, while the second one does not, in line with the results in \cite{Dari.Duran.Padra:95,Dari.Duran.Padra.Vampa:96}. 

% quasi-optimal CR method
We start by recalling the aforementioned Crouzeix-Raviart method. The lowest-order finite element space of
Crouzeix-Raviart~\cite{Crouzeix.Raviart:73} on \(\grid\) is defined as 
\begin{equation} \label{Eq:CR1}
	\CR_1:= \left\{ v \in S^0_{1}\mid \forall F\in\sides ~\int_F\jump{v} = 0 \right\},
\end{equation}
and is not contained in $\mathring{H}^1(\Omega)$,  the trial and test space  of the standard weak formulation of the Poisson problem~\eqref{Eq:Poisson}. The corresponding discrete problem reads
\begin{equation}
\label{Eq:CR-qopt}
 \begin{multlined} 
    \text{for \(f\in H^{-1}(\Omega)\)},~\text{find $u_h \in \CR_1$ such that}
 \\
    \forall v_h\in \CR_1\quad
    \int_\Omega \nabla_h u_h \cdot \nabla_h v_h = \scp{f}{\calE_\mathsf{CR} v_h}.
\end{multlined}
\end{equation}
Hereafter,  $\nabla_h$ denotes the broken gradient given by
\[
 (\nabla_h v)_{|T}:= \nabla (v_{|T})
 \quad\text{for }
 T\in\grid \text{ and  suitable } v.
\]
The \emph{smoothing operator} $\calE_{\mathsf{CR}}: \CR_1 \to \mathring{H}^1(\Omega)$ is defined
in \cite[Proposition~3.3]{Veeser.Zanotti:19b} and satisfies the following properties: for all $v_h \in
\CR_1$, we have
\begin{subequations} 
\label{Eq:CR-smoother}
\begin{alignat}{2}
    \label{Eq:CR-smoother-moments}
    \forall F &\in \sides &\qquad &\int_F \calE_{\mathsf{CR}}v_h = \int_F v_h,
 \\  \label{Eq:CR-smoother-stability}
    \forall T&\in\grid &\qquad
    &h_T^{-1} \norm[L^2(T)]{v_h - \calE_{\mathsf{CR}} v_h} + \|\nabla \calE_\mathsf{CR}
    v_h\|_{L^2(T)}\lesssim \|\nabla_h v_h\|_{L^2(\omega_T)},
\end{alignat}
which are related to consistency and stability of the method, respectively; see also below. Note that the bound \eqref{Eq:CR-smoother-stability} entails that $\calE_\mathsf{CR}$ is a local operator but may enlarge the support of its argument. 
\end{subequations}

% specify setting and consistency properties
In order to apply Section~\ref{sec:posteriori-analysis}, we first note that the Poisson problem \eqref{Eq:Poisson} and the method \eqref{Eq:CR-qopt} fit into the abstract framework of Section~\ref{sec:qo-methods} with
\begin{equation}
\label{setting-for-Poisson-and-CR1}
\begin{gathered}
 \tV = \mathring{H}^1(\Omega) + \CR_1
\quad\text{with}\quad
  V = \mathring{H}^1(\Omega),
\quad%\text{and}\quad
 V_h = \CR_1,
\\
 \ta(v,w) = \int_\Omega \nabla_h v \cdot \nabla_h w,
\quad
 \norm{v} =\|\nabla_h v\|_{L^2(\Omega)} ,
%\quad
% v,w\in \mathring{H}^1(\Omega) + \CR_1,
\\
 a_h = \ta_{|\CR_1\times\CR_1},
\quad\text{and}\quad
 \calE_h = \calE_\mathsf{CR}.
\end{gathered}
\end{equation}
The conservation \eqref{Eq:CR-smoother-moments} of face means by the
smoother $\calE_\mathsf{CR}$ and the fact that both, discrete and continuous
bilinear forms, are restrictions of $\ta$ imply 
\begin{alignat}{2}
	\label{Eq:CR-smoother-overconsistency}
	\forall v_h,w_h &\in \CR_1 &\quad
	&\int_\Omega \nabla_h w_h \cdot \nabla_h(v_h - \calE_{\mathsf{CR}} v_h) = 0.
\end{alignat}
Consequently, the consistency measure vanishes, i.e. $\delta_h = 0$, and the method is overconsistent; cf.\ Remark~\ref{R:overconsistent-methods}. Hence, 
Proposition~\ref{P:QuasiOptMethods} and the smoother stability~\eqref{Eq:CR-smoother-stability} ensure that the method is  quasi-optimal with the constant $\Cqo = \normtr{\calE_\mathsf{CR}}$  depending on the shape constant $\gamma_\grid$ and the dimension $d$.

% intro to (H*) and verifying (H1)
We next verify the main assumptions of Section~\ref{sec:posteriori-analysis} for the setting \eqref{setting-for-Poisson-and-CR1}. The derivations of the two estimators will differ only in the verification of \eqref{A:construction-Pz-local} for the construction of the local projections. Let us start with \eqref{A:equiv-to-dist-to-conf} concerning the distance to conformity. According to Proposition~\ref{P:inds-for-dist-to-conf}, we need a linear operator that is invariant on the intersection
\( %V \cap V_h =
\mathring{H}^1(\Omega)\cap\CR_1=\mathring{S}^1_1\). In line with Remark~\ref{R:dist-to-conf-averaging}, a possible choice is the averaging operator $\mathcal{A}_\mathsf{CR}:
\CR_1 \to \mathring{S}^1_1$ defined by 
\begin{equation}
\label{Eq:CR1-averaging}
	\mathcal{A}_\mathsf{CR} v_h(z) := \dfrac{1}{\#\grid_z} \sum_{T \in \grid_z}v_{h|T}(z),
\qquad
 z \in \vertices^i,\; v_h\in\CR_1,
\end{equation}
see \cite{Brenner:96,Carstensen.Graessle.Nataraj:24}.

\begin{lemma}[Approximate $\|\nabla_h\cdot\|_{L^2}$-distance to $\mathring{H}^1$ by averaging]
\label{L:CR-distance-to-conformity}
In the setting \eqref{setting-for-Poisson-and-CR1}, the operator
$\mathcal{A}_\mathsf{CR}$ in \eqref{Eq:CR1-averaging} satisfies assumption
\eqref{A:equiv-to-dist-to-conf}  and \(\Cav\gtrsim 1\). In particular,
 for
all $v_h \in 
  \CR_1$ we have that
  \begin{align*}
		\norm[L^2(\Omega)]{\nabla_h(v_h - \mathcal{A}_\mathsf{CR} v_h)}
		\lesssim
		\inf_{v\in \mathring{H}^1(\Omega)}\norm[L^2(\Omega)]{\nabla_h(v_h - v)}
		\leq
		\norm[L^2(\Omega)]{\nabla_h(v_h - \mathcal{A}_\mathsf{CR} v_h)}.
	\end{align*}
\end{lemma}

\begin{proof}
Clearly,  $\mathcal{A}_\mathsf{CR}$ satisfies \eqref{A:equiv-to-dist-to-conf} and so the second bound of the claimed equivalence holds. For its first bound, which is equivalent to  $\Cav\gtrsim1$, we have
\begin{equation*}
    \norm[L^2(\Omega)]{\nabla_h(v_h - \mathcal{A}_\mathsf{CR} v_h)}^2
    \lesssim
    \int_\Sigma \dfrac{|\jump{v_h}|^2}{h};
  \end{equation*}
see, e.g., \cite[Theorem~2.2]{Karakashian.Pascal:03}. To conclude, we adapt ideas from~\cite[Theorem 10]{Achdou.Bernardi.Coquel:03} to prove  that the jumps are bounded by the distance to $\mathring{H}^1(\Omega)$.  Let $F \in \sides$ be any face and, for any element \(T\in\grid\) with  \(\partial T\supset F\), define \(\phi_T\in H^1(T)\) as the weak solution of the Neumann problem 
  \begin{align*}
    -\Delta \phi_T=0~\text{in}~T,\qquad \partial_\normal
    \phi_T=\jump{v_h}~\text{on}~F, \quad\text{and}\quad \partial_\normal
    \phi_T=0~\text{on}~\partial T\setminus F,
  \end{align*}
 with $\int_T \phi_T = 0$. Note that the existence of $\phi_T$ hinges on $\int_F \jump{v_h} = 0$ and thus on $v_h \in \CR_1$.  For all \(v\in \mathring{H}^1(\Omega)\), integration by parts yields
\begin{align*}
	\int_F |\jump{v_h}|^2
	=
	\sum_{T \in \grid, T \supset F} \int_{T} \nabla \phi_T \cdot \nabla(v_h-v)
	\le 
	\sum_{T \in \grid, T \supset F} \norm[L^2(T)]{\nabla \phi_T}\norm[L^2(T)]{\nabla (v_h-v)}. 
\end{align*}
For any $T \in \grid$ with $\partial T \supset F$, the norm of $\phi_T$ can be further estimated by
\begin{align*}
		\norm[L^2(T)]{\nabla \phi_T}^2 
		= 
		\int_F \jump{v_h}\phi_T
		\lesssim
		h_F^{\frac12}\norm[L^2(F)]{\jump{v_h}}\norm[L^2(T)]{\nabla\phi_T}
\end{align*}
with the help of the scaled trace theorem, a Poincar\'e inequality thanks to \(\int_T \phi_T=0\), and $h_T  \lesssim h_F$. Inserting this bound into the previous one shows
	\begin{equation*}
		\int_{F} \dfrac{|\jump{v_h}|^2}{h} \lesssim \inf_{v \in \mathring{H}^1(\Omega)} \norm[L^2(\omega_F)]{\nabla_h(v_h-v)}^2.
	\end{equation*}
	We conclude by summing over all faces, because the number of patches $\omega_F$ containing a given simplex is bounded by $d+1$.
\end{proof}

The proof of Lemma~\ref{L:CR-distance-to-conformity} reveals that properly scaled jumps of $v_h$ on the mesh skeleton can be considered as an alternative indicator for the distance to $\mathring{H}^1(\Omega)$. 
\begin{corollary}[Approximate $\|\nabla_h\cdot\|_{L^2}$-distance to $\mathring{H}^1$ by jumps]
\label{C:CR-jump-distance-to-conformity}
  For \(v_h\in\CR_1\), we have that 
	\begin{equation*}
		\inf_{v\in \mathring{H}^1(\Omega)}\norm[L^2(\Omega)]{\nabla_h(v_h - v)}^2 
		\eqsim
		\int_\Sigma \dfrac{|\jump{v_h}|^2}{h}.
	\end{equation*}
\end{corollary}

% verifying (H2)
We next establish assumption \eqref{A:partition-of-unity}. Recalling that the Courant basis functions 
\eqref{Eq:nodal-basis-S11} form a partition of unity in $H^1(\Omega)$, we take
\begin{equation}
\label{Eq:CR-partition-of-unity}
 \Index = \vertices,
\qquad
 \Phi_z = \Psi_z,
\qquad
 V_z = \mathring{H}^1(\omega_z),
\qquad
 \calI_h = \calI_\mathsf{CR}
\end{equation}
where $\calI_\mathsf{CR}: \mathring{H}^1(\Omega) \to \CR_1$ is the Crouzeix-Raviart interpolant; see, e.g., \cite[Section~2.1]{Brenner:15}.

\begin{lemma}[Partition of unity in $H^1$ for $\mathsf{CR}$]
\label{L:CR-partition-of-unity}	
In the  setting  \eqref{setting-for-Poisson-and-CR1}, the choices  \eqref{Eq:CR-partition-of-unity} satisfy assumption \eqref{A:partition-of-unity} with constants $\Cloc,\Ccol \lesssim 1$.
\end{lemma} 

\begin{proof}
Property~(i) in \eqref{A:partition-of-unity} is valid because $\Psi_z \in W^{1,\infty}(\Omega)$ and $\operatorname{supp} \Psi_z = \omega_z$ for all $z \in \vertices$ as well as $\sum_{z \in \vertices} \Psi_z = 1$. Property~(ii) follows from  the fact that each given element $T \in \grid$ is contained in $(d+1)$ stars $\omega_z$, $z \in\vertices$, the triangle and the Cauchy-Schwarz inequality. Indeed, given any $v_z \in \mathring{H}^1(\omega_z)$, $z \in \vertices$, we have
\begin{align*}
 \left\| \nabla \left(\sum_{z\in\vertices} v_z \right) \right\|_{L^2(T)}^2
 &=
 \left\| \sum_{z\in\vertices \cap T} \nabla v_z \right\|_{L^2(T)}^2
 \leq
 \left( \sum_{z\in\vertices \cap T} \| \nabla v_z  \|_{L^2(T)} \right)^2
\\
 &\leq
  (d+1) \sum_{z\in\vertices \cap T} \| \nabla v_z  \|_{L^2(T)}^2,
\end{align*}
and summing over $T \in \grid$ yields $\Ccol  \leq \sqrt{d+1}$. Finally, in order to verify property~(iii), let $v \in \mathring{H}^1(\Omega)$ and $z \in \vertices$. The scaling properties $\norm[L^\infty(\Omega)]{\Psi_z} = 1$ and $\norm[L^\infty(\Omega)]{\nabla \Psi_z} \eqsim h_z^{-1}$ of the Courant basis functions and a triangle inequality entail
\begin{equation*}
		\begin{split}
			\norm[L^2(\Omega)]{\nabla \big( (v - \calE_{\mathsf{CR}} \calI_\mathsf{CR}v)\Psi_z \big)}
			&\lesssim 
			h_z^{-1}\left( \norm[L^2(\omega_z)]{v- \calI_\mathsf{CR} v}
			+
			\norm[L^2(\omega_z)]{\calI_\mathsf{CR} v - \calE_{\mathsf{CR}} \calI_\mathsf{CR} v}\right) \\
			&\quad + \norm[L^2(\omega_z)]{\nabla (v - \calE_{\mathsf{CR}} \calI_\mathsf{CR} v)}.
		\end{split}
\end{equation*}
Then, the stability properties of the smoother $\calE_{\mathsf{CR}}$, see
\eqref{Eq:CR-smoother-stability},  and of  the interpolation
$\calI_\mathsf{CR}$ (see \cite[eq. (2.14)]{Brenner:15}) imply 
\begin{equation*}
		\sum_{z \in \vertices} \norm[L^2(\Omega)]{\nabla((v - \calE_{\mathsf{CR}} \calI_\mathsf{CR}v)\Psi_z)}^2
		\lesssim 
		\sum_{z \in \vertices} \norm[L^2(\omega_z^+)]{\nabla v}^2
		\lesssim
		\norm[L^2(\Omega)]{\nabla v}^2.
\end{equation*}
Here the second inequality follows from the observation that the number of enlarged
stars $\omega_z^+$ containing a given simplex is bounded in terms of
the shape constant \(\gamma_\grid\) and $d$. Hence $\Cloc \lesssim 1$. 
\end{proof}

% verifying (H3)
We finally turn to assumption \eqref{A:construction-Pz-local} enabling
the construction of the local projections.  In order to choose the
simple functionals, we follow Remark~\ref{R:choosing-simple-fcts}. To
this end, we observe that, for a discrete trial function $v_h\in\CR_1$ and test function $v \in
\mathring{H}^1(\Omega)$, element-wise integration by parts provides the
representation 
\begin{equation}
\label{Eq:action-laplacian}
	\int_\Omega \nabla_h v_h \cdot \nabla v
	=
	-\sum_{T\in\grid}\int_{T}\Delta v_h v
	+
	\sum_{F\in\sides^i} \int_{F} \jump{\nabla v_h }\cdot\normal v 
\end{equation}
where we do not exploit \(\Delta ({v_h}_{|T})=0\) for $v_h \in \CR_1$
and $T \in \grid$, in order to hint at the higher-order
case. Furthermore, we note that the classical oscillation in the case
at hand arises from approximating with functionals of the form 
\begin{equation*}
	\mathring{H}^1(\Omega) \ni v \mapsto \int_\Omega \Bar{f} v,
\end{equation*}
 where $\Bar{f} \in L^\infty(\Omega)$ is given by $\Bar{f}_{|T} :=
 |T|^{-1}\int_T f$ for suitable $f$. For local test functions
 $v \in \mathring{H}^1(\omega_z)$, $z \in \vertices$, the latter type of functionals and those of the second sum in the right-hand side of \eqref{Eq:action-laplacian}
 are given by constants associated with the elements of the set
 $\grid_z \cup \sides_z^i$. We thus define  
\begin{align}
\label{Eq:CR-simple-functionals}
	\begin{aligned}
		\widehat{D}_z = \Big\{\chi\in H^{-1}(\Omega)&\mid  \scp{\chi}{v} =
		\sum_{T\in\grid_z} \int_T r_T v +
		\sum_{F\in\sides_z^i}\int_F r_F v
		\\
		&\phantom{\mid} \text{with}~r_T\in \R,\,  T \in \grid_z, 
		\text{ and }
		r_F\in\R,\, F\in\sides_z^i\Big\},
	\end{aligned}
\end{align} 
where the degrees of freedom are decoupled in contrast to \eqref{Eq:action-laplacian}, and take
\begin{subequations}
\label{simple-CR1-conf}
\begin{equation}
\label{Eq:CR-conf-simple-functionals}
 D_z
 =
 \widehat{D}_z{}_{|\mathring{H}^1(\omega_z)}
\end{equation}
as simple functionals. The space $\widehat{D}_z$ is introduced to allow for a better comparison with the announced second approach to \eqref{A:construction-Pz-local}.
The fact that each member of the set $\grid_z  \cup \sides_z^i$  corresponds to a degree of freedom in the given $D_z$ suggests to establish an analogous correspondence for the simple test functions. A possible choice is
\begin{equation}
\label{Eq:CR-simple-test-functions}
	S_z
	=
	\mathrm{span}\{ \Psi_T \mid T \in \grid_z \} \oplus   \mathrm{span}\{\Psi_F \mid F \in \sides_z^i\}, 
\end{equation}
where $\Psi_T$ and $\Psi_F$ are the bubble functions from \eqref{Eq:bubbles}.
\end{subequations}
These choices for simple functionals and test functions are essentially used also in \cite{Kreuzer.Veeser:21} for the conforming method based upon the space $\mathring{S}^1_1$. The only difference is that here we pair simple functionals and test functions locally instead of globally.
\begin{lemma}[Local projections for $\mathsf{CR}$@Poisson with face bubbles]
\label{L:CR-local-projections-1}
In the settings \eqref{setting-for-Poisson-and-CR1} and \eqref{Eq:CR-partition-of-unity},  the choices \eqref{simple-CR1-conf} verify assumption \eqref{A:construction-Pz-local} with \(\CL\lesssim1\) and \(\CNL=1\). Therefore, \eqref{Eq:construction-Pz-local} defines projections $\calP_z: H^{-1}(\omega_z) \to D_z \subset H^{-1}(\omega_z)$ satisfying 
\begin{equation*}
 \forall g \in H^{-1}(\Omega)
\quad
 \sum_{z\in\vertices}  \norm[H^{-1}(\omega_z)]{\calP_z g}^2
	\lesssim 
 \sum_{z\in\vertices}  \norm[H^{-1}(\omega_z)]{g}^2.
\end{equation*}
\end{lemma}

\begin{proof}
Let $z \in \vertices$. As $S_z \subset \mathring{H}^1(\omega_z) = V_z$, the simple test functions are locally conforming. This readily yields condition (i) of \eqref{A:construction-Pz-local} with $\tV_z =  \mathring{H}^1(\omega_z)$, condition (v) with $\CNL=1$, and condition (ii) in the light of \eqref{Eq:action-laplacian}. Condition (iii) holds because the functionals $\mathring{H}^1(\omega_z) \ni v \mapsto \int_K v$, $K \in \grid_z \cup \sides_z^i$, form a basis of $D_z$. Condition (iv) is proved in \cite[Theorems~8-10]{Kreuzer.Veeser:21} with $\CL \lesssim 1$; see also Lemma~\ref{L:CR-local-projections-2} below for a similar argument. Finally, Proposition~\ref{P:construction-Pz-local} provides the existence of $\calP_z$ and the claimed stability bound.        
\end{proof}

The approach represented by Lemmas \ref{L:CR-distance-to-conformity}, \ref{L:CR-partition-of-unity} and \ref{L:CR-local-projections-1} leads to the following strictly equivalent error estimator.

\begin{theorem}[Estimator for $\mathsf{CR}$@Poisson]
\label{T:CR-error-estimator-1}
Let $u \in \mathring{H}^1(\Omega)$  be the weak solution of the Poisson problem \eqref{Eq:Poisson} and $u_h \in \CR_1$ its quasi-optimal Crouzeix-Raviart approximation from \eqref{Eq:CR-qopt}. Given tuning constants \(C_1, C_2>0\), define
\begin{multline*}
	\Est_\mathsf{CR}^2
	:=
	\Ncf_\mathsf{CR}^2 + C_1^2 \eta_\mathsf{CR}^2 + C_2^2 \Osc_\mathsf{CR}^2,  \text{ where }
	\\
	\begin{aligned}
		\Ncf_\mathsf{CR}^2
		&:= 
		\norm[L^2(\Omega)]{ \nabla_h( u_h - \calA_\mathsf{CR} u_h ) }^2,
		\\ 
		\eta_\mathsf{CR}^2
		&:= 
		\sum_{T \in \grid} \eta_{\mathsf{CR},T}^2
		%\\
		\text{ with  }
		\eta_{\mathsf{CR},T}
		:=
		\max_{ K \in \mathcal{K}_T }
		 \frac{|\scp{f}{\Psi_K} - \int_\Omega \nabla_h u_h \cdot \nabla \Psi_K |}{\|\nabla \Psi_K\|_{L^2(\Omega)}},  
		\\
		\Osc_\mathsf{CR}^2
		&:= 
		\sum_{z\in\vertices} \norm[H^{-1}(\omega_z)]{f - \calP_z f}^2, 
	\end{aligned}
\end{multline*}
with the averaging operator $\calA_\mathsf{CR}$ from \eqref{Eq:CR1-averaging}, $\mathcal{K}_T =  \{T\} \cup \{F \in \sides^i \mid F \subset T \}$, the bubble functions $\Psi_T$ and $\Psi_F$ from \eqref{Eq:bubbles}, and the local projections \( \calP_z\) from Lemma~\ref{L:CR-local-projections-1}.
This estimator quantifies the error by
\begin{equation*}
 \underline{C} \Est_\mathsf{CR}
 \leq
 \norm[L^2(\Omega)]{\nabla_h(u-u_h)}
 \leq
 \overline{C} \Est_\mathsf{CR}
\quad\text{and}\quad
 \eta_{\mathsf{CR},T} 
 \leq \norm[L^2(\widetilde{\omega}_T)]{\nabla_h(u-u_h)}
\end{equation*}
where $\widetilde{\omega}_T = \bigcup_{F \in \sides^i, F \subset \partial T} \omega_F$  and the equivalence constants $\underline{C}$ and $\overline{C}$
depend only on the dimension \(d\),  the shape constant
\(\gamma_\grid\), and the tuning constants \(C_1, C_2\).

Furthermore, if $f \in L^2(\Omega)$, the oscillation indicator is bounded in terms of the classical $L^2$-oscillation:
\begin{equation*}
 \Osc_\mathsf{CR}^2
 \lesssim
 \sum_{T \in \grid} h_T^2 \inf_{c \in \R} \norm[L^2(T)]{f-c}^2.
\end{equation*}
\end{theorem} 

\begin{remark}[Estimator variants for  $\mathsf{CR}$@Poisson]
\label{R:CR-computable-indicator}
For the ease of implementation, the reduction of computational cost, or with the hope to improve the equivalence constants $\overline{C}, \underline{C}$, one may consider the following variants of the indicators in Theorem~\ref{T:CR-error-estimator-1}:
\begin{itemize}
\item In view of  Corollary~\ref{C:CR-jump-distance-to-conformity}, we may replace $\Ncf_\mathsf{CR}$ with properly scaled jumps and an additional tuning constant.
\item A simplification of $\eta_\mathsf{CR}$ is derived in Theorem~\ref{T:CR-error-estimator-2} below. 
\cite[Section~4]{Kreuzer.Veeser:21} discusses other alternatives for $\eta_\mathsf{CR}$ by using so-called local problems (this type is used also in Theorem~\ref{T:dG-error-estimator} below) and the standard residual technique. 
Moreover, in light of the equivalences $\norm[L^2(\Omega)]{\nabla \psi_K} \eqsim h_K^{d-2}$, one may replace $\eta_{\mathsf{CR},T}$ by
\begin{equation*}
%\label{simplified-hier-ind}
 \max_{ K \in \mathcal{K}_T }
	\frac{|\scp{f}{\Psi_K} - \int_\Omega \nabla_h u_h \cdot \nabla \Psi_K |}{h_K^{d-2}},  
\end{equation*}	
which is particularly simple for $d=2$. This simplified hierarchical indicator is closely related to Remark~\ref{R:computing-local-projections}, which in turn connects to local problems on $S_z$. Indeed, taking into account also $h_K \eqsim h_z$, we observe that the matrix $\mathbf{A}_z$ in Remark ~\ref{R:computing-local-projections} is spectrally equivalent to the diagonal matrix $\mathbf{diag}(h_z^{d-2},\dots, h_z^{d-2})$. We therefore have
\begin{equation*}
 \norm[H^{-1}(\omega_z)]{\calP_z \Res_h^{\mathtt{C}}}^2
 \eqsim
 \sum_{T\in\grid_z} \dfrac{|\scp{\Res_h^{\mathtt{C}}}{\Psi_T}|^2}{h_T^{d-2} }
 +
 \sum_{F \in \sides_z^i}  \dfrac{|\scp{\Res_h^{\mathtt{C}}}{\Psi_F}|^2}{h_F^{d-2}}.
\end{equation*}
and see that the `$\max$' in the above hierarchical indicators is mainly motivated by the `constant-free' lower bound in Theorem~\ref{T:CR-error-estimator-1}.
\item For alternative oscillation indicators, see the discussion in the Remarks~\ref{R:computability}, \ref{R:neglecting-data-osc}, and \ref{R:class-osc-as-surrogate}.
\end{itemize}
The hierarchical variant for $\eta_\mathsf{CR}$ in Theorem~\ref{T:CR-error-estimator-1} is best suited for a comparison with Theorem~\ref{T:CR-error-estimator-2}. 
\end{remark}

\begin{proof}[Proof of Theorem~\ref{T:CR-error-estimator-1}]
% equivalence between estimator and error
Thanks to Lemmas~~\ref{L:CR-distance-to-conformity},
\ref{L:CR-partition-of-unity}, and~\ref{L:CR-local-projections-1}, we
can apply Theorem~\ref{T:error-estimator} with
\eqref{setting-for-Poisson-and-CR1} and
\eqref{Eq:CR-partition-of-unity}. The overconsistency of the
Crouzeix-Raviart method \eqref{Eq:CR-qopt}, i.e.\ $\delta_h=0$,
simplifies the abstract estimator therein and   the claimed
equivalence follows with the help of 
 \begin{equation}
\label{CR-pde-indicator-1}
 \sum_{z \in \vertices} \norm[H^{-1}(\omega_z)]{\calP_z\Res_h^{\mathtt{C}}}^2
 \eqsim
 \sum_{z \in \vertices}  \norm[S_z']{\Res_h^{\mathtt{C}}}^2
 \eqsim
 \eta_\mathsf{CR}^2,
 \end{equation}
 where the first equivalence is a consequence of properties (iv) and (v) of
 \eqref{A:construction-Pz-local} and the second one can be shown with
 the arguments in \cite[\S4.1]{Kreuzer.Veeser:21}.
 
 % lower bound
 To verify the claimed lower bound, let $T \in \grid$ and $K \in \{T\} \cup \{ F \in \sides^i \mid F \subset \partial T\}$ be arbitrary. Thanks to the weak formulation of \eqref{Eq:Poisson} and $\operatorname{supp} \Psi_K \subset \widetilde{\omega}_T$, we readily obtain the constant-free lower bound
 \begin{equation*}
 	\frac{|\scp{f}{\Psi_K} - \int_\Omega \nabla_h u_h \cdot \nabla \Psi_K |}{\|\nabla \Psi_K\|_{L^2(\Omega)}}
 	=
 	\frac{|\int_\Omega \nabla_h (u-u_h) \cdot \nabla \Psi_K|}{\|\nabla \Psi_K\|_{L^2(\Omega)}}
 	\leq
 	\norm[L^2(\widetilde{\omega}_T)]{\nabla_h(u-u_h)}.
\end{equation*}

To show the bound for the oscillation, write $\bar{f}$ for the piecewise constant function with $\bar{f}_{|T} = |T|^{-1} \int_T f$ for all $T \in \grid$. Since $\bar{f}_{|\omega_z} \in D_z$ for any vertex $z \in \vertices$, the properties of the local projections and the Poincar\'e inequality imply
\begin{align*}
 \sum_{z \in \vertices} \norm[H^{-1}(\omega_z)]{f - \calP_z f}^2
 &=
 \sum_{z \in \vertices} \norm[H^{-1}(\omega_z)]{f - \bar{f} - \calP_z(f-\Bar{f})}^2
 \lesssim
 \sum_{z \in \vertices} \norm[H^{-1}(\omega_z)]{f - \bar{f}}^2
\\
 &\lesssim
 \sum_{z \in \vertices} h_z^2 \norm[L^2(\omega_z)]{f - \bar{f}}^2
 \lesssim
 \sum_{T \in \grid} h_T^2 \inf_{c \in \R} \norm[L^2(T)]{f-c}^2
\end{align*}
and the proof is finished.
\end{proof}

\begin{remark}[Generalization to higher order Crouzeix-Raviart elements]
\label{R:CRHiO}
\cite[section~3.3]{Veeser.Zanotti:19b}  derives quasi-optimal methods based upon Crouzeix-Raviart spaces of any order and dimension. Using techniques similar to those in Section~\ref{sec:dG}
for discontinuous Galerkin methods, Theorem~\ref{T:CR-error-estimator-1} can be generalized to these methods.
\end{remark}

%
% second approach with smooth CR basis fcts
%
We next derive an error estimator that does not  need the evaluations $\scp{\Res_h^{\mathtt{C}}}{\Psi_F}$, $F \in \sides^i$, of the conforming residual associated with interior faces. This  simplification is in line with \cite{Dari.Duran.Padra:95,Dari.Duran.Padra.Vampa:96} and, in the framework of Section \ref{sec:posteriori-analysis},  mainly relies  on suitably replacing the face-bubble functions $\psi_F$, $F \in \sides_z^i$, in the choice \eqref{Eq:CR-simple-test-functions} of the simple test functions. To individuate their alternatives,  denote by $(\Psi_F^{\mathsf{CR}})_{F \in \sides^i} \subseteq \CR_1$ the Crouzeix-Raviart basis functions, defined by the condition   
\begin{align}
	\label{Eq:CR-basis-functions}
	\Psi_F^{\mathsf{CR}}(m_{F'})=\delta_{FF'},\quad\forall F'\in\sides^i,
\end{align}
where $m_{F'}$ is the barycenter of $F'$. We observe from \eqref{Eq:CR-qopt} that
\begin{equation}
\label{Eq:CR-further-orthogonality}
	\forall F \in \sides^i 
	\qquad 
	\scp{\Res_h^{\mathtt{C}}}{\calE_{\mathsf{CR}} \Psi_F^\mathsf{CR}}
	=
	\scp{f}{\calE_\mathsf{CR} \Psi_F} - \int_\Omega \nabla_h u_h\cdot \nabla_h \Psi_F^\mathsf{CR}
	= 0.
\end{equation}
This orthogonality suggests to replace the face-bubble functions by the smoothed Crouzeix-Raviart basis functions, which have enlarged supports; cf.\ \eqref{Eq:CR-smoother-stability}.  In order to maintain the boundedness in terms of classical oscillation, these  enlarged supports need to be accompanied  by additional element-bubble functions; see Remark~\ref{R:nconf-elm-bubbles} below. Thus, the new simple test functions and functionals are
\begin{subequations}
\label{simple-CR1-nconf}
\begin{align}
\label{Eq:CR-simple-test-functions-2}
	S_z
	&=
	\mathrm{span}\{ \Psi_T \mid T \in \grid_z^+ \}
	\oplus
	\mathrm{span}\{\calE_{\mathsf{CR}}\Psi_F^\mathsf{CR} \mid F \in \sides_z^i\}
%	\not\subset \mathring{H}^1(\omega_z),
\\ \label{simple-fct-CR1-nconf}
	D_z
	&=
	\Big\{\chi\in \big( \mathring{H}^1(\omega_z) + S_z \big)' \mid  
	\scp{\chi}{v} = \sum_{T\in\grid_z^+} \int_T r_T v + \sum_{F\in\sides_z^i}\int_F r_F v
\\ \notag
   &\phantom{=\Big\{\chi\in H^{-1}(\Omega) \mid}
    \text{with}~r_T\in \R,\,  T \in \grid_z^+,  \text{ and } r_F\in\R,\, F\in\sides_z^i\Big\}.
\end{align} 
The simple test functions \eqref{Eq:CR-simple-test-functions-2} are locally nonconforming, i.e.\ $S_z \not\subset \mathring{H}^1(\omega_z)$, and the simple functionals \eqref{simple-fct-CR1-nconf} essentially differ from \eqref{Eq:CR-conf-simple-functionals} only by the additional characteristic functions $\chi_T$, $T \in \grid_z^+ \setminus \grid_z$.
\end{subequations}

\begin{lemma}[Local projections for $\mathsf{CR}$@Poisson with smoothed CR basis]
\label{L:CR-local-projections-2}
In the settings \eqref{setting-for-Poisson-and-CR1} and \eqref{Eq:CR-partition-of-unity},  the choices \eqref{simple-CR1-nconf} satisfy assumption \eqref{A:construction-Pz-local} with \(\CL\lesssim1\) and \(\CNL \lesssim1\). Hence, \eqref{Eq:construction-Pz-local} defines projections $\calP_z: \Dual{\tV_z} \to D_z \subset \Dual{\tV_z}$  with $\tV_z = \mathring{H}^1(\omega_z) + S_z$ and we have
\begin{equation*}
 \forall g \in H^{-1}(\Omega)
\quad
 \sum_{z\in\vertices}  \norm[H^{-1}(\omega_z)]{\calP_z g}^2 
 \lesssim
 \sum_{z\in\vertices}  \norm[H^{-1}(\omega_z)]{g}^2.
\end{equation*}
\end{lemma}

\begin{proof}
Let $z \in \vertices$. Clearly, the choices \eqref{simple-CR1-nconf} satisfy conditions (i) and (iii) of  \eqref{A:construction-Pz-local}. In order to verify (ii), let $ v_h \in V_h$, \( v \in \mathring{H}^1(\omega_z)\) and \(s \in S_z\). Then the representation \eqref{Eq:action-laplacian} and $\Delta (v_{h|T}) = 0$ in each $T \in \grid$ yield
\begin{equation*}
 \int_\Omega\nabla_hv_h\cdot\nabla (v+s)
 =
 \sum_{F \in \sides^i_z}  \int_F \jump{\nabla v_h} \cdot \normal (v+s)
 +
 \sum_ {F' \in \sides \setminus \sides^i_z} \int_{F'} \jump{\nabla v_h} \cdot \normal s
\end{equation*}
and we have to show that the second sum vanishes. Let $F '\in \sides \setminus \sides^i_z$ be any face appearing therein and observe,  for any $T \in \grid_z^+$ and any $F \in \sides_z^i$, that
\begin{equation}
\label{CR-nconf-sides}
 \Psi_T{}_{|F'}  = 0,
\quad
 \int_{F'} \calE_\mathsf{CR} \Psi_F^\mathsf{CR} =  \int_{F'} \Psi_F^\mathsf{CR} = \Psi_F^\mathsf{CR} (m_{F'})= 0
\end{equation}
thanks to the moment conservation \eqref{Eq:CR-smoother-moments} and the definition of the Crouzeix-Raviart basis. Therefore, $\int_{F'}  \jump{\nabla v_h} \cdot \normal s =  0$ and condition (ii) is verified.

We next show condition (iv) of \eqref{A:construction-Pz-local}. Given any $\chi \in D_z$ and $v \in \mathring{H}^1(\omega_z)$, we can write	
\begin{align*}
 \langle \chi, v\rangle
  =
  \sum_{T\in\grid_z^+}\int_T r_T v
  +
  \sum_{F\in\sides_z^i}\int_F r_F v
 \end{align*}
for some $r_T \in \poly_0(T)$, $T \in \grid_z^+$ and $r_F\in \poly_0(F)$,  $F\in\sides_z^i$.  Fixing a local test function $v$, we  choose the simple test function
\begin{equation*}
	s
	=
	\sum_{T\in\grid_z^+} s_T \Psi_T + \sum_{F \in \sides_z^i} s_F \calE_{\mathsf{CR}} \Psi_F^\mathsf{CR} \in S_z
\end{equation*}
where $s_T \in \poly_0(T)$, $T \in \grid_z^+$, and $s_F \in \poly_0(F)$, $F \in \sides_z^i$, solve the problems
\begin{align*}
		&\forall p_F \in \poly_0(F) \qquad \int_F s_F p_F \Psi_F^\mathsf{CR} = \int_F v p_F,
\\
		&\forall p_T \in \poly_0(T) \qquad \int_Ts_T p_T \Psi_T = \int_T v p_T - \sum_{F \in \sides_z^i} \int_T s_F p_T \calE_{\mathsf{CR}} \Psi_F^\mathsf{CR}.
\end{align*}
These problems are uniquely solvable because $\Psi_F^\mathsf{CR}$ and $\Psi_T$ are strictly positive in the interior of $F$ and $T$, respectively. We then have
$%\begin{equation*}
 \scp{\chi}{v} = \scp{\chi}{s}
$ %\end{equation*}
and
\begin{equation*}
	\norm[L^2(\omega_z^+)]{\nabla s} 
	\lesssim
	h_z^{-1}  \sum_{T\in\grid_z^+} \norm[L^2(T)]{s_T} + h_z^{-\frac{1}{2}}\sum_{F \in \sides_z^i} \norm[L^2(F)]{s_F} 
	\lesssim
	\norm[L^2(\omega_z)]{\nabla v},
\end{equation*}
where the first inequality follows from \eqref{Eq:CR-smoother-stability} and the scaling properties of the basis functions of $S_z$ and, for the second one, we use in addition norm equivalences with bubble functions~\cite[Proposition~1.4]{Verfuerth:13}, trace and Poincar\'e inequalities. Hence, we arrive
at 
\begin{equation*}
		\dfrac{\left\langle \chi, v\right\rangle }{\norm[L^2(\omega_z)]{\nabla v}}
		\lesssim 
		\dfrac{\left\langle \chi, s\right\rangle }{\norm[L^2(\omega_z^+)]{\nabla s}}
		\leq \norm[S_z']{\chi}
\end{equation*} 
and, taking the supremum over $v$, we  obtain the desired condition (iv).
        
Finally, for verifying condition (v)  in \eqref{A:construction-Pz-local}, assume $g \in H^{-1}(\Omega)$. Since we have $S_z \subset \mathring{H}^1(\omega_z^+)$ and so $\norm[\Dual{S_z}]{g} \leq \norm[H^{-1}(\omega_z^+)]{g}$ for all vertices $z \in \vertices$, the desired bound follows from
\begin{equation}
\label{stability-of-overlapping}
 \sum_{z\in\vertices}	 \norm[H^{-1}(\omega_z^+)]{g}^2
 \lesssim
 \sum_{z\in\vertices}	 \norm[H^{-1}(\omega_z)]{g}^2. 
\end{equation}
Given any vertex $z \in \vertices$ and $v \in \mathring{H}^1(\omega_z^+)$, we can write $ v = \sum_{y \in \vertices \cap \omega_z^+} v \Psi_y$ and have the stability bounds $\norm[L^2(\omega_y)]{\nabla(v \Psi_y)} \lesssim
\norm[L^2(\omega_z^+)]{\nabla v}$ thanks to scaling properties of the Courant basis and the Poincar\'e inequality with zero boundary values on $\omega_z^+$. Hence, similarly to the proof of Lemma~\ref{L:CR-partition-of-unity}, we derive
\begin{equation}
\label{E:omega+<omega}
\begin{aligned}
 \dfrac{\left\langle g, v\right\rangle }{\norm[L^2(\omega_z^+)]{\nabla v}}
 &\lesssim
 \sum_{y \in \vertices \cap \omega_z^+} \dfrac{\left\langle g, v\Psi_y\right\rangle }{\norm[L^2(\omega_y)]{\nabla(v \Psi_y)}}
 \lesssim 
 \sum_{y \in \vertices \cap \omega_z^+} \norm[H^{-1}(\omega_y)]{g}
\\
 & \lesssim
 \left(
  \sum_{y \in \vertices \cap \omega_z^+} \norm[H^{-1}(\omega_y)]{g}^2
 \right)^{1/2},
\end{aligned}
\end{equation}
where in the last inequality we used that $\#(\vertices\cap\omega_z^+)$ is bounded in terms of the shape constant $\gamma_\grid$ and $d$. We take the supremum over $v$, square and sum over all vertices $z$. This verifies \eqref{stability-of-overlapping}
because the number of patches $\omega_z^+$ containing a vertex $y$ is
again bounded in terms of shape constant \(\gamma_\grid\) and $d$.  
\end{proof}

\begin{theorem}[Simplified estimator for $\mathsf{CR}$@Poisson]
\label{T:CR-error-estimator-2}
Let $u \in \mathring{H}^1(\Omega)$  be the weak solution of the Poisson problem \eqref{Eq:Poisson} and $u_h \in \CR_1$ its quasi-optimal Crouzeix-Raviart approximation from \eqref{Eq:CR-qopt}. Given tuning constants \(C_1,C_2>0\), define
\begin{multline*}
	\Est_{\widetilde{\mathsf{CR}}}^2
	:=
	\Ncf_{\widetilde{\mathsf{CR}}}^2 + C_1^2 \eta_{\widetilde{\mathsf{CR}}}^2 + C_2^2 \Osc_{\widetilde{\mathsf{CR}}}^2,  \text{ where }
	\\
	\begin{aligned}
		\Ncf_{\widetilde{\mathsf{CR}}}^2
		&:= 
		\norm[L^2(\Omega)]{ \nabla_h( u_h - \calA_\mathsf{CR} u_h ) }^2,
		\\ 
		\eta_{\widetilde{\mathsf{CR}}}^2
		&:= 
		\sum_{T \in \grid} \eta_{{\widetilde{\mathsf{CR}}},T}^2
		%\\
		\text{ with  }
		\eta_{{\widetilde{\mathsf{CR}}},T}
		:=
		\frac{|\scp{f}{\Psi_T} - \int_\Omega \nabla_h u_h \cdot \nabla \Psi_T |}{\|\nabla \Psi_T\|_{L^2(\Omega)}},  \qquad\qquad
		\\
		\Osc_{\widetilde{\mathsf{CR}}}^2
		&:= 
		\sum_{z\in\vertices} \norm[H^{-1}(\omega_z)]{f - \calP_z f}^2, 
	\end{aligned}
\end{multline*}
with the averaging operator $\calA_\mathsf{CR}$ from \eqref{Eq:CR1-averaging}, the element-bubble functions $\Psi_T$  from \eqref{Eq:bubbles}, and the local projections \( \calP_z\) from Lemma~\ref{L:CR-local-projections-2}.
This estimator quantifies the error by
\begin{equation*}
	\underline{C} \Est_{\widetilde{\mathsf{CR}}}
	\leq
	\norm[L^2(\Omega)]{\nabla_h(u-u_h)}
	\leq
	\overline{C} \Est_{\widetilde{\mathsf{CR}}}
	\quad\text{and}\quad
	\eta_{{\widetilde{\mathsf{CR}}},T} 
	\leq \norm[L^2(T)]{\nabla_h(u-u_h)}
\end{equation*}
and the equivalence constants $\underline{C}$ and $\overline{C}$ depend only on the dimension \(d\),  the shape constant
\(\gamma_\grid\), and the tuning constants \(C_1,C_2\).

Furthermore, if $f \in L^2(\Omega)$, the oscillation indicator is bounded in terms of the classical $L^2$-oscillation:
\begin{equation*}
	\Osc_{\widetilde{\mathsf{CR}}}^2
	\lesssim
	\sum_{T \in \grid} h_T^2 \inf_{c \in \R} \norm[L^2(T)]{f-c}^2.
	\end{equation*}
\end{theorem} 

\begin{proof}
The proof is very similar to the one of Theorem~\ref{T:CR-error-estimator-1}. The main differences for the estimator equivalence are that we must invoke Lemma~\ref{L:CR-local-projections-2} instead of Lemma~\ref{L:CR-local-projections-1} and that  the smoothed CR basis functions do not appear in the estimator in view of their orthogonality \eqref{Eq:CR-further-orthogonality} to the residual. Also the  oscillation bound follows along the same lines, with the caveat that now the condition
\begin{equation}
\label{pw-constant-invariance-on-ex-patch}
 \bar{f}_{|\omega_z^+} \in D_z
\end{equation}
ensures the invariance $(\calP_z \bar{f})_{|\mathring{H}^1(\omega_z)} = \bar{f}_{|\omega_z}$, where we identify $\bar{f}_{|\omega_z}$ with the functional $\mathring{H}^1(\omega_z) \ni v \mapsto \int_\Omega \bar{f} v$. 
\end{proof}

To discuss Theorem~\ref{T:CR-error-estimator-2}, two remarks address the design and the generality of the approach behind it and three remarks compare its estimator with alternatives.

\begin{remark}[Nonconforming element-bubble functions]
\label{R:nconf-elm-bubbles}
One may be tempted to build the second approach on the simpler pairs
\begin{equation}
\label{wrong-DDPV}
 S_z
 =
 \mathrm{span}\{ \Psi_T \mid T \in \grid_z \} \oplus   \mathrm{span}\{\calE_\mathsf{CR}\Psi^\mathsf{CR}_F \mid F \in \sides_z^i\}, 
\quad
  D_z = \widehat{D}_z{}_{|(\mathring{H}^1(\omega_z) + S_z)},
\end{equation}
which just replace the face-bubble functions in \eqref{Eq:CR-simple-test-functions} by the smoothed Crouzeix-Raviart basis functions, thus avoiding the nonconforming element-bubble functions in $S_z$ along with their companions in $D_z$. In this case, the estimator equivalence in Theorem~\ref{T:CR-error-estimator-2} still holds, but the oscillation bound is not valid in general. Indeed, assuming \eqref{wrong-DDPV}, testing \eqref{Eq:construction-Pz-local} first with the element-bubble functions $\Psi_T$, $T \in \grid_z$, and then with the smoothed Crouzeix-Raviart basis functions $\calE_\mathsf{CR} \Psi_F^\mathsf{CR}$, $F \in \sides_z$ reveals that $\calP_z \bar{f} = \bar{f}$ holds only under additional assumptions on $\bar{f}_{|(\omega_z^+ \setminus\omega_z)}$ outside of the star $\omega_z$. This drawback is overcome by including the nonconforming element-bubble functions $\Psi_T$, $T \in \grid_z^+ \setminus \grid_z$, and their companions $\chi_T$,   $T \in \grid_z^+ \setminus \grid_z$, ensuring \eqref{pw-constant-invariance-on-ex-patch}, which in turn implies $(\calP_z \bar{f})_{|\mathring{H}^1(\omega_z)} = \bar{f}_{|\omega_z}$.

On the other hand, if we set the oscillation bound in Theorem~\ref{T:CR-error-estimator-2} aside, we can use the even simpler pair
\begin{equation*}
	S_z
	=
	\mathrm{span}\{\calE_\mathsf{CR}\Psi^\mathsf{CR}_F \mid F \in \sides_z^i\}, 
	\quad
	D_z = \widehat{D}_z{}_{|(\mathring{H}^1(\omega_z) + S_z)},
\end{equation*}
and obtain another strictly equivalent  estimator, which consists only
of an approximate distance to conformity and a data oscillation term
and which is a lower bound of the estimators in
\cite[Section~2.3]{Brenner:15} and~\cite[Section
3.2.3]{Carstensen.Graessle.Nataraj:24}, respectively. 
\end{remark}

\begin{remark}[Simplification for some higher-order Crouzeix-Raviart methods]
\label{R:CRHiO-2}
For \(d=2\) and odd polynomial degree $\ell$, a similarly simplified estimator as in Theorem~\ref{T:CR-error-estimator-2} can be derived for the corresponding Crouzeix-Raviart method in \cite[Section~3.3]{Veeser.Zanotti:19b}. In this case, \cite[Lemma~1]{Stoyan.Baran:06} shows that point evaluations at the order \(\ell\) Legendre  nodes on the element edges together with the order \(\ell\) Lagrange points in the interior of a given element determine  $\mathbb{P}_\ell$. The smoothed basis functions corresponding to the Legendre nodes on the edges are then orthogonal to the residual. Thus, the higher order counterpart of \eqref{simple-CR1-conf} can be modified to a higher order counterpart of \eqref{simple-CR1-nconf} in order to obtain a generalization of Theorem~\ref{T:CR-error-estimator-2}. 
\end{remark}

\begin{remark}[Comparison of strictly equivalent estimators with quasi-optimality]
\label{R:comp-strictly-eq-ests}
%
% intro and equivalences
Let us compare the two estimators in Theorems~\ref{T:CR-error-estimator-1} and \ref{T:CR-error-estimator-2}. The differing indicators are related as follows:
\begin{equation}
\label{CR-indicator-comparison}
 \eta_\mathsf{CR} \geq \eta_{{\widetilde{\mathsf{CR}}}},
\quad
 \eta_\mathsf{CR} \not\lesssim \eta_{{\widetilde{\mathsf{CR}}}},
\quad \textit{and} \quad
 \Osc_\mathsf{CR} \eqsim \Osc_{{\widetilde{\mathsf{CR}}}}.
\end{equation}
These properties are proved below. 

% implementation and computational cost
Clearly, the implementation of the indicator $\eta_{\widetilde{\mathsf{CR}}}$ is easier than the one of $\eta_{\mathsf{CR}}$, while the one of $\Osc_{\mathsf{CR}}$ is slightly easier that the one of $\Osc_{{\widetilde{\mathsf{CR}}}}$ because the latter in addition involves some smoothing; recall that smoothing itself has to be implemented in any case for assembling the discrete problem \eqref{Eq:DisEq} and that oscillations may be replaced by surrogates. In terms of computational costs, the situation is quite similar in the sense that `easier' is replaced with `less costly'.  To sum up, we may view $\Est_{\widetilde{\mathsf{CR}}}$ as a simplification of $\Est_\mathsf{CR}$.

% hidden constants
We finally  inspect the (hidden) constants in the proof. Since the derivations of the two upper bounds differ only in the second equivalence of \eqref{CR-pde-indicator-1}, the constants in front of the oscillations $\Osc_\mathsf{CR}$ and $\Osc_{\widetilde{\mathsf{CR}}}$ coincide and we expect that those in front of $\eta_\mathsf{CR}$ and $\eta_{\widetilde{\mathsf{CR}}}$ are different. In light of the first inequality in \eqref{CR-indicator-comparison},  the latter should turn out to be larger.   For the global lower bounds, the constants for $\eta_{\widetilde{\mathsf{CR}}}$ and $\Osc_{\widetilde{\mathsf{CR}}}$ are expected to be larger in view of the increased overlapping in the simple  functions, as suggested by the given proofs for (v) in \eqref{A:construction-Pz-local}. 
\end{remark}

\begin{proof}[Proof of \eqref{CR-indicator-comparison}]
We use the superscript `$\sim$' to distinguish the spaces and operators related to the estimator $\Est_{\widetilde{\mathsf{CR}}}$ from those related to $\Est_\mathsf{CR}$. 

\fbox{\scriptsize 1} The first claimed inequality readily follows by the respective definitions. Furthermore, the  invariance and stability properties of $\calP_z$ as well as $\widetilde{D}_{z}{}_{|\mathring{H}^1(\omega_z)} = D_z$ imply
 \begin{equation*}
	\| f - \calP_z  f\|_{H^{-1}(\omega_z)}
	=
	\| (\operatorname{id} - \calP_z) ( f - \widetilde{\calP}_{z} f)\|_{H^{-1}(\omega_z)}
	\lesssim
	\| f - \widetilde{\calP}_{z} f\|_{H^{-1}(\omega_z)}.
\end{equation*} 
Hence, summing over $z \in \vertices$, gives the oscillation bound $\Osc_\mathsf{CR} \lesssim \Osc_{\widetilde{\mathsf{CR}}}$.

\fbox{\scriptsize 2} To prepare the proof of inequality  $\Osc_{\widetilde{\mathsf{CR}}} \lesssim \Osc_\mathsf{CR}$, we introduce a further local projection acting over the extended star $\omega_z^+$: for any vertex $z \in \vertices$, the pair 
 \begin{gather*}
  S_z^+
  :=
  \mathrm{span}\{ \Psi_T \mid T \in \grid_z^+ \} \oplus \mathrm{span}\{\Psi_F \mid F \in \sides_z^{i+}\},
\\
 \begin{aligned}
  D_z^+
  :=
  \Big\{ \chi\in H^{-1}(\omega_z^+)&\mid 
   \scp{\chi}{v} =\sum_{T\in\grid_z^+} \int_T r_T v + \sum_{F\in\sides_z^{i+}}\int_F r_F v
\\
	&\phantom{\mid} \text{with}~r_T\in \R,\,  T \in \grid_z^+,  \text{ and } r_F\in\R,\, F\in\sides_z^{i+}\Big\},
\end{aligned}
\end{gather*}
induces through Proposition~\ref{P:construction-Pz-local} a local projection $\calP_z^+:H^{-1}(\omega_z^+) \to D_z^+$. For a given vertex $y \in \vertices \cap \omega_z$,  combining $D_z^+{}_{|\mathring{H}^1(\omega_y)} = D_y$, $S_y \subset S_z^+$, and  the uniqueness of \eqref{Eq:construction-Pz-local} for \eqref{simple-CR1-conf} yield that, for any $g \in H^{-1}(\Omega)$,
\begin{equation}
\label{P_z^+-orthogonality}
	( \calP_z^+ g)_{|\mathring{H}^1(\omega_y)} = \calP_y g.
\end{equation}
Also, there are constants $r_T$, $ T \in \grid_z^+$, and $r_F$, $F \in \sides_z^i$, such that, for all $ v \in \mathring{H}^1(
\omega_z)$ and $s \in \widetilde{S}_z$, we have
\begin{equation*}
  \scp{\calP_z^+ g}{v+s}
  =
  \sum_{T \in \grid_z^+} \int_T r_T (v+s)
  +
  \sum_{F \in \sides_z^i} \int_F r_F  (v+s)
\end{equation*}
thanks to \eqref{CR-nconf-sides}. In other words, $(\calP_z^+
g)_{|(\mathring{H}^1(\omega_z) + \widetilde{S}_z)}  \in
\widetilde{D}_z$ and the invariance of $\widetilde{\calP}_z$ ensure 
\begin{equation}
\label{Pzvar-and-Pz+}
 \big( \calP_z^+g \big)_{|(\mathring{H}^1(\omega_z) + \widetilde{S}_z)}
 =
 \widetilde{\calP}_z ( \calP_z^+g ).
\end{equation}
Note that this property hinges on the companions of the nonconforming
element-bubble functions in Remark~\ref{R:nconf-elm-bubbles}. 

\fbox{\scriptsize 3} We turn to the proper proof of
$\Osc_{\widetilde{\mathsf{CR}}} \lesssim \Osc_\mathsf{CR}$. Employing
\eqref{Pzvar-and-Pz+} restricted to $\mathring{H}^1(\omega_z)$, the
stability properties of $\widetilde{\calP}_z$, $\widetilde{S}_z
\subset \mathring{H}^1(\omega_z^+)$, \eqref{E:omega+<omega} and
\eqref{P_z^+-orthogonality},  we obtain 
\begin{align*}
	\norm[H^{-1}(\omega_z)]{ f - \widetilde{\calP}_{z} f}
	&=
	\norm[H^{-1}(\omega_z)]{ f - \calP_z^+ f - \widetilde{\calP}_{z} ( f - \calP_z^+ f)} 
	\\
	& \lesssim
	\norm[H^{-1}(\omega_z^+)]{ f - \calP_z^+ f} 
	\lesssim
	\left(
	\sum_{y \in \vertices \cap \omega_z^+} \norm[H^{-1}(\omega_y)]{ f -  \calP_y f }^2
	\right)^{1/2},
\end{align*}
and  so squaring, and summing over all vertices, yields the desired bound.

\fbox{\scriptsize 4} Finally, we verify that $\eta_{{{\mathsf{CR}}}}$ cannot be bounded by $\eta_{{\widetilde{\mathsf{CR}}}}$. For $f \in H^{-1}(\Omega)$ and $T \in \grid$, an integration by parts reveals 
\begin{equation*}
	\scp{f}{\Psi_T} - \int_\Omega \nabla_h u_h \cdot \nabla \Psi_T = \scp{f}{\Psi_T}.
\end{equation*} 
Moreover, for $F \in \sides^i$, the orthogonality property of the interpolant $\calI_\mathsf{CR}$ in \eqref{Eq:CR-partition-of-unity}, see \cite[Section~2.1]{Brenner:15}, and \eqref{Eq:CR-qopt} entail that
\begin{equation*}
	\scp{f}{\Psi_F} - \int_\Omega \nabla_h u_h \cdot \nabla \Psi_F = \scp{f}{\Psi_F - \calE_{\mathsf{CR}} \calI_\mathsf{CR} \Psi_F}.
\end{equation*} 
The construction of $\calE_{\mathsf{CR}}$ in \cite[Proposition~3.3]{Veeser.Zanotti:19b} reveals that there is a face $F$ such that  $(\Psi_F - \calE_{\mathsf{CR}} \calI_\mathsf{CR} \Psi_F)_{|F}$ is not the zero function. Therefore, fixing that face and taking the source $f$ so that $\scp{f}{v} = \int_F (\Psi_F - \calE_{\mathsf{CR}} \calI_\mathsf{CR} \Psi_F) v$ for $v \in \mathring{H}^1(\Omega)$, we have $\eta_{{{\mathsf{CR}}}} \neq 0$, whereas $\eta_{{\widetilde{\mathsf{CR}}}} = 0$. 
\end{proof} 

\begin{remark}[Comparison with classical estimators]
\label{R:comparison-with-classical-est}
We compare the estimators in Theorems~\ref{T:CR-error-estimator-1} and \ref{T:CR-error-estimator-2} with the classical ones in \cite{Dari.Duran.Padra.Vampa:96}, for the classical Crouzeix-Raviart method
\begin{equation}
\label{Eq:CR-without-smoothing}
	\text{find $u_h \in \CR_1$ such that\;}
	\forall v_h\in \CR_1\quad
	\int_\Omega \nabla_h u_h \cdot \nabla_h v_h = \scp{f}{v_h}.
\end{equation}
For $* \in \{ \mathsf{CR}, \widetilde{\mathsf{CR}} \}$, the classical estimators are given by
\begin{equation}
\label{CR-class-ests}
\begin{aligned}
 \Est_\mathrm{*}^2
 :=
 &\Ncf_\mathrm{cl, *}^2 + \eta_\mathrm{cl,*}^2 + \Osc_\mathrm{cl,*}^2,
 \quad\text{where}
\\
 	 \Ncf_{\mathrm{cl,*}}^2
 	 &:=
 	 \int_\Sigma \dfrac{|\jump{u_h}|^2}{h},
 \qquad
 	 \Osc_{\mathrm{cl},*}^2
 	 :=
 	 \sum_{T \in \grid} \inf_{c \in \R} \int_T h^2 |f-c|^2
  \\
 	 \eta_{\mathrm{cl,*}}^2
 	 &:=
 	 \sum_{T \in \grid}   \eta_\mathrm{cl,*}(T)^2
 	\quad\text{with}
\\
  &\eta_\mathrm{cl,*}(T)^2
 	 :=
 	 \begin{cases}
      \int_T h^2 |f|^2 + \sum_{F \subseteq \partial T} \int_F h |\jump{\nabla u_h}\cdot n|^2,
       &\text{if }\mathrm{*}=\mathsf{CR}, 
     \\
      \int_T h^2 |f|^2 , &\text{if }\mathrm{*}= \widetilde{\mathsf{CR}}.
    \end{cases}
\end{aligned}\end{equation}
Note that both, the classical method \eqref{Eq:CR-without-smoothing} and the estimators \eqref{CR-class-ests}, are not defined for all $f \in H^{-1}(\Omega)$. More crucially, there are no continuous extensions of their definitions to $H^{-1}(\Omega)$. As a consequence, the classical method \eqref{Eq:CR-without-smoothing}
is not quasi-optimal and the estimators cannot be strictly equivalent;  cf.\ \cite[Remark~4.9]{Veeser.Zanotti:18} and \cite{Kreuzer.Veeser.Zanotti:24}. It therefore comes as no surprise that, if $ f \in L^2(\Omega)$, we have
\begin{equation}
\label{est<cl-est}
	\Est_{*}
	\lesssim
	\Est_{\mathrm{cl},\mathrm{*}}
\quad
 * \in \{\mathsf{CR},\widetilde{\mathsf{CR}}\}.
\end{equation}
In fact, the respective inequalities for the  oscillations indicators are already shown in Theorems~\ref{T:CR-error-estimator-1} and \ref{T:CR-error-estimator-2},  while those for  the indicators of the approximate $\mathring{H}^1$-distance and the conforming residual follow from Corollary \ref{C:CR-jump-distance-to-conformity} and standard residual techniques. The fact that the dependence of $\Est_{\mathrm{cl,*}}$ on $f$ does not continuously extend to $H^{-1}(\Omega)$ entails that there is a sequence $(f_k)_k \subset L^2(\Omega)$ such that  the right-hand side of \eqref{est<cl-est} tends to $\infty$, while the left-hand side remains bounded or even tends to $0$.
\end{remark}

\begin{remark}[Quasi-optimality and strict estimator equivalence]
\label{R:role-of-qopt}
The quasi-op\-ti\-mal\-i\-ty of the nonconforming method is a useful
but not necessary ingredient for the strict equivalence of error
estimators.  In fact, the usefulness is clear from
Lemma~\ref{L:conforming-residual} and, to show the non-necessity, we
outline the derivation of  a strictly equivalent error estimator for
the classical Crouzeix-Raviart method \eqref{Eq:CR-without-smoothing}. Since
the quasi-optimality enters our approach only in Lemma
~\ref{L:conforming-residual}, we can generalize the bounds therein by
accepting the additional indicator 
\begin{equation*}
  \| \Res_h^\mathtt{C} \|_{\Dual{\calE_h(V_h)}}
 :=
 \sup_{w_h \in \calE_h(V_h), \| w_h \| \leq 1 } \scp{ \Res_h^\mathtt{C}}{w_h}.
\end{equation*}
This indicator can be computed by solving a discrete problem or quantified by using the correction of a solver with strict error reduction or by a hierarchical approach as illustrated in \cite{Fierro.Veeser:06}. Combing this idea, e.g., with Theorem~\ref{T:CR-error-estimator-1} and using the notation therein shows that the estimator
\begin{equation*}
 \left(
	 \Ncf_\mathsf{CR}^2 + \eta_\mathsf{CR}^2 +   \| \Res_h^\mathtt{C} \|_{\Dual{\calE_\mathsf{CR}(\CR_1)}} ^2+ \Osc_\mathsf{CR}^2
\right)^{1/2}
\end{equation*} 
is strictly equivalent to the error $\| \nabla_h (u-u_h) \|_{L^2(\Omega)}$ of the classical Crouzeix-Raviart method \eqref{Eq:CR-without-smoothing}.
\end{remark}

%
% \begin{remark}
% %
% \todo{A: Do we really need this remark?}
% The simplified estimators in Theorem~\ref{T:CR-error-estimator-2}
%   unveil the following general rule of thumb: Whenever some of the
%   information in \(D_z\) can be obtained by local testing with
%   (smoothed) functions from the finite element space, this information
%   can be skipped from the estimator without significantly enlarging
%   the oscillation.
%    This effect is most obvious for the hierarchical error estimator, 
%    where Galerkin orthogonality applies to indicators corresponding
%    to testfunctions from the finite 
%    element space.
% \end{remark}

% \begin{remark}[Volume contributions]
% 	\label{eq:DzAugmented}
% \todo{A: I think this is now already covered in (various) places and I would remove.}
% 	Actually \(\Delta {v_h}_{|T}=0\) for $v_h \in \CR_1$ and $T
% 	\in \grid$. Therefore, the identity
% 	\eqref{Eq:action-laplacian} reveals that including in $D_z$
% 	the volume functionals $v \mapsto \int_T r_T v$, $T \in
% 	\grid_z$, is not strictly necessary. Still, having those
% 	functionals in $D_z$ ensures that the oscillation in
% 	Theorems~\ref{T:CR-error-estimator-1}
% 	and~\ref{T:CR-error-estimator-2} below is bounded by the one
% 	in \eqref{Eq:oscillation-CR-classical} for $f\in L^2(\Omega)$
% 	and thus it can be expected to be negligible compared to the
% 	error for sufficiently smooth $f$, see also \cite[Section
% 	3.8]{Kreuzer.Veeser:21}. \end{remark}

%
%
\subsection{Discontinuous Galerkin methods of fixed arbitrary order}
\label{sec:dG}
In this section we apply the abstract framework of Sections~\ref{sec:qo-methods} and \ref{sec:posteriori-analysis} to derive strictly equivalent error estimators for quasi-optimal %symmetric- (SIP) and non-symmetric (NIP)
discontinuous Galerkin methods of fixed arbitrary order.  We shall restrict to the symmetric and non-symmetric interior penalty methods in \cite[Section~3.2]{Veeser.Zanotti:18b} for the ease of presentation, see also Remark~\ref{R:dG-smoother-ass}. To recall them, we use the notation of Section \ref{sec:doma-mesh-polyn}. 

% Introduce NIP SIP
% Although the methods may behave differently
% in some regards quite drastically, this does not affect the a
% posteriori analysis.

% DG methods to be considered
Given a polynomial degree $\ell \in \mathbb{N}$ %$\theta \in\{-1,1\}$,
and a stabilization parameter $\sigma$, we define a discrete bilinear form on the possibly discontinuous piecewise polynomials \(S_\ell^0\) by
\begin{equation}
\label{dG-bilinear-form}
\begin{multlined}
	a_h(u_h,v_h)
	=
	\int_\Omega \nabla_h u_h \cdot \nabla_h v_h
	+
	\int_\Sigma %\sum_{F\in\sides}\int_F
	 \frac{\sigma}{h} \jump{u_h} \jump{v_h}
\\
 {} + \int_\Sigma %\sum_{F\in\sides}\left(
	 \theta %\int_F
	 \jump{u_h}\mean{\partial_n v_h}
	 - %\int_F
	 \mean{\partial_n u_h} \jump{v_h}
	%\right)
\end{multlined}
\end{equation}
with $\theta = 1$ for the non-symmetric and $\theta = -1$ for the symmetric case. In order to ensure coercivity of the bilinear
form, the former case requires \(\sigma>0\) and the latter \(\sigma>\sigma^*\) for some sufficiently large \(\sigma^*\) depending on the dimension $d$, the shape constant \(\gamma_\grid\) and the polynomial degree \(\ell\); compare, e.g., with~\cite[Lemma~3.5 and (3.23)]{Veeser.Zanotti:18b}. The discrete problem then reads
\begin{equation}
\label{Eq:dG-qopt}
%	\begin{gathered} 
	\text{for \(f\in H^{-1}(\Omega)\)},~\text{find $u_h \in S_\ell^0$ such that }
	\forall v_h\in S_\ell^0\quad
	a_h(u_h,\,v_h) = \scp{f}{\calE_\ell v_h},
%	\end{gathered}
\end{equation}
where the smoothing operator $\calE_\ell:  S^0_\ell \to \mathring{H}^1(\Omega)$ is defined in \cite[Proposition~3.4]{Veeser.Zanotti:18b} and satisfies for the following \emph{moment conditions}  and  \emph{stability properties}: for all \(v_h\in S^0_\ell\), $F \in \sides^i$, $m_F\in\poly_{\ell-1}(F)$, $ T \in \grid $, $m_T\in\poly_{\ell-2}(T)$ 
\begin{subequations}
\label{Eq:dG-smoother}
\begin{gather} 
\label{Eq:dG-smoother-moments}
  \int_F \calE_\ell(v_h) m_F
  =
  \int_F \mean{v_h} m_F,
\qquad
  \int_T \calE_\ell(v_h) m_T 
  =
  \int_T v_h m_T,
\\ \label{Eq:dG-smoother-stability}
  \begin{aligned}%[t]
		h_T^{-2} \norm[L^2(T)]{v_h - \calE_{\ell} v_h}^2
		&+
		\|\nabla \calE_\ell v_h\|_{L^2(T)}^2
\\
		&\lesssim
		\|\nabla_h v_h\|_{L^2(\omega_T)}^2
		+
		\sum_{F\in \sides, F\cap T\neq\emptyset} \int_F \frac{1}{h} |\!\jump{v_h}\!|^2.
 \end{aligned}
\end{gather}
\end{subequations}

% setting for abstract framework
Together with the Poisson problem \eqref{Eq:Poisson}, method \eqref{Eq:dG-qopt} fits into the framework of Section~\ref{sec:qo-methods} by considering
\begin{equation}
\label{dG-setting}
\begin{gathered}
 	\tV = S^0_\ell + \mathring{H}^1(\Omega) \quad\text{with}\quad V = \mathring{H}^1(\Omega), \quad V_h = S^0_\ell,
\\
 \ta(v,w)
 =
 \int_{\Omega} \nabla_h v \cdot \nabla_h w
 +
 \int_\Sigma %\sum_{F\in\sides} \int_F
  \frac{\sigma}{h} \jump{v} \jump{w},
 \quad
 \norm{v}
 =
 \|v\|_{\texttt{dG}} := \ta(v,v)^{1/2},
\\
 a_h \text{ as in \eqref{dG-bilinear-form}},
\quad
 \text{and}\quad \calE_h = \calE_\ell.
\end{gathered} 
\end{equation}
We note that these methods are quasi-optimal but not overconsistent, i.e.\ we have $\delta_h \lesssim 1$ in Proposition~\ref{P:QuasiOptMethods}, where the hidden constants depend only on $d$, \(\gamma_\grid\), \(\ell\) and $\sigma$; see \cite[Theorems~3.6 and 3.7]{Veeser.Zanotti:18b}, which also shows that $\delta_h \searrow 0$ for $\sigma \nearrow \infty$. % This causes some additional
% conformity error e.g. when applying Corollary~\ref{Cor:AbsLocalization}.

% intro and verification of (H1)
To establish the main assumptions of Section~\ref{sec:posteriori-analysis}, we follow the arguments of the first approach for the Crouzeix-Raviart method in Section~\ref{sec:CR1}, focusing on the changes due to the higher order.
Regarding a suitable operator for \eqref{A:equiv-to-dist-to-conf}, recall that averaging the function values at the Lagrange points of order \(\ell\) instead of the vertices of the mesh generalizes the  operator
from~\eqref{Eq:CR1-averaging} to an averaging operator 
\begin{align}
\label{Eq:dG-averaging}
  \calA_\ell: S_\ell^0 \to \mathring{S}^1_\ell
 \qquad\text{such that}\qquad
  \forall w_h\in\mathring{S}_\ell^1=S_\ell^0\cap \mathring{H}^1(\Omega)
 \quad
  \calA_\ell w_h=w_h% \qquad\text{and}\qquad \|\nabla\calA_\ell
  % w_h\|\lesssim \|w_h\|_{\texttt{dG}}
  ;
\end{align}
% where the hidden constant depends only on the shape constant 
% \(\gamma\grid\); 
compare, e.g.,  with \cite[Section~2.1]{Karakashian.Pascal:03}.

\begin{lemma}[Approximate $\|\cdot\|_{\mathtt{dG}}$-distance to $\mathring{H}^1$ by averaging]
\label{L:dG-distance-to-conformity}
In the setting~\eqref{dG-setting},  the operator $\calA_\ell$  from \eqref{Eq:dG-averaging} satisfies assumption \eqref{A:equiv-to-dist-to-conf} and \(\Cav\gtrsim \sqrt{{\sigma}/{(1+\sigma)}}\). In particular, we have that,  for all $v_h \in S_\ell^0$, 
 \begin{align*}
		\sqrt{\frac{\sigma}{1+\sigma}}\|v_h - \mathcal{A}_\ell v_h\|_{\mathtt{dG}}
		\lesssim
		\inf_{v\in \mathring{H}^1(\Omega)}\|v_h - v\|_{\mathtt{dG}}
		\leq
			\|v_h - \mathcal{A}_\ell v_h\|_{\mathtt{dG}}.
\end{align*}
%The hidden constants depend only on the the dimension \(d\), the shape constant \(\gamma_\grid\) and the polynomial degree \(\ell\).
\end{lemma}

\begin{proof}
The second inequality holds because the operator $\calA_\ell$ verifies assumption~\eqref{A:equiv-to-dist-to-conf} and it remains to show the first one. Similar to Lemma~\ref{L:CR-distance-to-conformity}, \cite[(3.17)]{Veeser.Zanotti:18b} assures
\begin{equation*}
	\norm[\texttt{dG}]{v_h - \mathcal{A}_\ell v_h}^2
    \lesssim
    (1+\sigma) \int_\Sigma \dfrac{|\jump{v_h}|^2}{h}.
\end{equation*}
Since the right-hand side scaled by \(\frac\sigma{1+\sigma}\) is bounded by \(\|v-v_h\|_{\texttt{dG}}\) for any \(v\in \mathring{H}^1(\Omega)\), the proof is finished.
\end{proof}

In analogy to Section~\ref{sec:CR1}, the proof of Lemma~\ref{L:dG-distance-to-conformity} shows that properly scaled jumps of $v_h$ on the mesh skeleton may be used as approximate distance. 
\begin{corollary}[Approximate $\|\cdot\|_{\mathtt{dG}}$-distance to $\mathring{H}^1$ by jumps]
\label{C:dG-jump-distance-to-conformity}
For \(v_h\in S_\ell^0\),  we have % for the \(\ta\) orthogonal
% projection \(\PiV:\mathring{H}^1(\Omega)+ S_\ell^0 \to
 % \mathring{H}^1(\Omega)\) that
 \begin{align*}
  \int_\Sigma \frac{\sigma}{h} |\jump{v_h}|^2% \dS
  \le
  \inf_{v\in \mathring{H}^1(\Omega)} \|{v - v_h}\|_{\mathtt{dG}}^2
 % \eqsim
 % \normdG{\calE_\ell v_h - v_h}^2
 % +\sum_{F\in\sides} \int_F \frac{\sigma}{h_F} \jump{\cdot}^2
   \lesssim
   \int_\Sigma\frac {1+\sigma}{h} |\jump{v_h}|^2 % \dS
 % .
          % \end{align*}
          % % \begin{align*}
          % %   \norm[L^2(\Omega)]{\nabla_h(\PiV v_h - v_h)}^2\le
          % %   \norm[L^2(\Omega)]{\nabla_h(\calE_\ell v_h - v_h)}^2% +\sum_{F\in\sides} \int_F \frac{\sigma}{h_F} \jump{\cdot}^2
          % %   \lesssim\sum_{F\in\sides}\int_F\frac 1h \abs{\jump{v_h}}^2 
          .
\end{align*}
\end{corollary}

% verification of (H2)
Next, we turn to assumption \eqref{A:partition-of-unity}.  Motivated by  the inclusion \( \CR_1\subset S_\ell^0\), we choose 
\begin{equation}
\label{Eq:dG-partition-of-unity}
  \Index = \vertices,
\qquad
  \Phi_z = \Psi_z,
\qquad
  V_z = \mathring{H}^1(\omega_z),
 \qquad
  \calI_h = \calI_\mathsf{CR},
\end{equation}
which coincides with \eqref{Eq:CR-partition-of-unity} and, in particular, does not depend on the polynomial degree  $\ell$.

\begin{lemma}[Partition of unity in $H^1$]
\label{L:dG-partition-of-unity}	
In the setting \eqref{dG-setting}, the choices \eqref{Eq:dG-partition-of-unity} satisfy	assumption~\eqref{A:partition-of-unity} with constants $\Cloc, \Ccol \lesssim 1$. 
\end{lemma} 

\begin{proof}
Since the difference between the settings \eqref{setting-for-Poisson-and-CR1} and \eqref{dG-setting} does not affect the proofs of (i) and (ii) of \eqref{A:partition-of-unity} in Lemma~\ref{L:CR-partition-of-unity}, they  are still valid here. The proof  of condition (iii), however, requires a small change due to the different smoothing operators, namely invoking the stability properties \eqref{Eq:dG-smoother-stability} of $\calE_\ell$ instead of those of $\calE_\mathsf{CR}$ in \eqref{Eq:CR-smoother-stability}.  
\end{proof}

% verification of (H3)
In order to motivate the simple test functions and functionals for assumption~\eqref{A:construction-Pz-local}, we follow Remark~\ref{R:choosing-simple-fcts} and  observe that, for a discrete trial function \(v_h\in S_\ell^0\) and a test function \(v\in \mathring{H}^1(\Omega)\), we have
\begin{equation}
\label{Laplace-action-dG}
  \widetilde a( v_h,v)
  =
  -\sum_{T\in\grid}\int_{T}\Delta v_h \psi % \dx
  +
  \sum_{F\in\sides^i} \int_{F} \jump{\nabla v_h }\cdot\normal v
   % \dS
  % + \sum_{F\in\sides}\int_F\frac{\gamma}h \jump{v_h}\jump{\psi} 
  % \dS 
\end{equation}
thanks to $\jump{v}=0$ on all faces $F \in \sides$. Thus, although the extended bilinear form differs from Section~\ref{sec:CR1}, the left-hand side has the same piecewise structure as the one in \eqref{Eq:action-laplacian}, only with polynomial densities of higher order. Given any vertex $z\in\vertices$, this suggests to choose the simple functionals
\begin{subequations}
\label{dG-simple}
\begin{equation} \begin{aligned}
\label{Eq:dG-simple-functionals}
  D_z
  =
  \Big\{ \chi  &\in H^{-1}(\omega_z) \mid
   \chi(v) = \sum_{T\in\grid_z} \int_T r_T v %\dx
      + \sum_{F\in\sides_z^i} \int_F r_F v %, \quad v\in \mathring{H}^1(\Omega), %\dS
 \\
 & \phantom{\in}
 \text{ with } r_T \in \poly_{\ell-1}(T), \,  T \in \grid_z,  \text{ and }
  r_F\in \poly_{\ell-1}(F), F \in\sides_z^i
  \Big\}.
\end{aligned} \end{equation}
To define the accompanying simple functions, we need to extend the polynomials on faces into the volume. To this end, given a face \(F=\operatorname{conv\,hull}\{z_0\ldots,z_{d-1}\}\in\sides\) and its midpoint $z_F$,  we recall that the extension operator  $E_F: \poly_{\ell-1}(F)% C(F)\cap H^1(F)
\to H^1(\omega_F)$ given by
\begin{align*}
	(E_F v)(x) 
	:=
	v\left( z_F + \sum_{i=1}^{d-1} \Psi_{z_i}(x) (z_i - z_F) \right),
\quad
  x \in \omega_F,
\end{align*}
satisfies the stability bound 
\begin{align*}
	\norm[L^2(\omega_F)]{E_F v}
	\lesssim 
	h_F^{\frac12} \norm[L^2(F)]{v};
\end{align*}
see, e.g.,  \cite[Lemma 4.20]{Bonito.Canuto.Nochetto.Veeser:24}.  The simple functions are then
\begin{equation}
\label{Eq:dG-simple-test-functions}
\begin{aligned}
	S_z
	&=
  \operatorname{span} \big\{p_T \Psi_T\mid p_T\in\poly_{\ell-1}(T), T\in\grid_z \big\}
\\
  & \qquad \oplus
  	\operatorname{span} \big\{ (E_Fp_F) \Psi_F \mid p_F \in\poly_{\ell-1}(F), F\in\sides_z^i \big\}.
\end{aligned}
\end{equation}
\end{subequations}
These  choices coincide with those for conforming methods in~\cite{Kreuzer.Veeser.Zanotti:24}. 

\begin{lemma}[Local projections for $\mathsf{dG}$@Poisson]
\label{L:dG-local-projections}
In the settings \eqref{dG-setting} and \eqref{Eq:dG-partition-of-unity},  the choices \eqref{dG-simple} satisfy assumption	 \eqref{A:construction-Pz-local} with \(\CL\lesssim1\) and \(\CNL=1\). Hence, \eqref{Eq:construction-Pz-local} defines projections
$\calP_z: H^{-1}(\omega_z) \to D_z \subset H^{-1}(\omega_z)$ with  
\begin{equation*}
%\label{Eq:dG-local-projections}
 \forall g \in H^{-1}(\Omega)
\quad
 \sum_{z\in\vertices} 	\norm[H^{-1}(\omega_z)]{\calP_z g}^2
 \lesssim
 \sum_{z\in\vertices} \norm[H^{-1}(\omega_z)]{g}^2.
\end{equation*}
\end{lemma}

\begin{proof}
Given any vertex $z \in \vertices$, we have \(S_z\subset V_z\), viz.\ the simple functions are locally conforming, and therefore (i) of \eqref{A:construction-Pz-local} holds with $\tV_z = \mathring{H}^1(\omega_z)$, (v) with $\CNL=1$, and (ii) in view of \eqref{Laplace-action-dG}. Condition (iii) holds because $S_z$ determines $D_z$. Condition~(iv) with $\CL \lesssim
1$ is shown in  \cite{Kreuzer.Veeser.Zanotti:24} by using an argument similar to the corresponding one in Lemma~\ref{L:CR-local-projections-2}.  Then Proposition~\ref{P:construction-Pz-local} implies the existence of $\calP_z$ and the collective stability bound.
\end{proof}

Using the above lemmas in the abstract a~posteriori analysis of Section~\ref{sec:posteriori-analysis} leads to the following strictly equivalent estimator. 
\begin{theorem}[Estimator for $\mathsf{dG}$@Poisson]
\label{T:dG-error-estimator}
Let $u \in \mathring{H}^1(\Omega)$  be the weak solution of the Poisson problem \eqref{Eq:Poisson} and $u_h \in S_\ell^0$ its quasi-optimal discontinuous Galerkin approximation from \eqref{Eq:dG-qopt}. Given tuning constants $C_1, C_2,C_3 >0$,
define
\begin{equation*}
\begin{aligned}
		\Est_\ell^2
		:= 
		 C_1 &\Ncf_\ell^2 + C_2 \eta_\ell^2 + C_3 \Osc_\ell^2,
		\quad\text{where}
		\\
		\Ncf_\ell^2
		&:=
		\norm[\mathtt{dG}]{u_h - \calA_\ell u_h}^2,
		\quad
		\Osc_\ell^2
		:=
		\sum_{z \in \vertices} \norm[H^{-1}(\omega_z)]{f-\calP_z f}^2,
		\\
		\eta_\ell^2
		&:=
		\sum_{z \in \vertices}   \eta_{\ell,z}^2
		\quad\text{with}\quad
		\eta_{\ell,z}
		:=
		\sup_{ s \in S_z}
		 \frac{|\scp{f}{s} - \int_{\Omega} \nabla_h u_h \cdot \nabla s|}{\|\nabla s\|_{L^2(\omega_z)}}		
\end{aligned}\end{equation*}
with the averaging operator $\calA_\ell$ from \eqref{Eq:dG-averaging}, the local projections $\calP_z$ from Lemma~\ref{L:dG-local-projections}, the test simple functions $S_z$ from \eqref{Eq:dG-simple-test-functions}, and $\eta_{\ell,z}$ can be computed as indicated in Remark~\ref{R:computing-local-projections}. This estimator quantifies the error by
\begin{equation*}
  \underline{C} \Est_\ell
  \lesssim
  \norm[\mathtt{dG}]{u_h - u }
  \lesssim
   \underline{C} \Est_\ell
 \qquad
  \eta_{\ell,z} \leq \norm[L^2(\omega_z)]{\nabla_h(u_h-u)},
\end{equation*}
where the equivalence constants $\underline{C}, \overline{C}$ depend only on the dimension $d$, the shape constant $\gamma_\grid$, the polynomial degree $\ell$, the stabilization parameter $\sigma$, and the tuning constants $C_1, C_2$ and $C_3$. 

Furthermore, if $f \in L^2(\Omega)$, the oscillation indicator is bounded in terms of the classical higher-order $L^2$-oscillation:
\begin{equation*}
	\Osc_\ell^2
	\lesssim
	\sum_{T \in \grid} h_T^2 \inf_{p \in \poly_{\ell-1}} \norm[L^2(T)]{f-p}^2.
\end{equation*}
\end{theorem} 

\begin{remark}[Estimator variants for $\mathsf{dG}$@Poisson]
\label{R:dG-indicators}
The estimator in Theorem~\ref{T:dG-error-estimator} amounts to an approach based upon (discrete) local problems. Indeed, the solution of
\begin{equation*}
	\text{find } v _z \in S_z \text{ such that }
	\forall  s \in S_z \; \int_{\omega_z} \nabla v_z \cdot \nabla s = \scp{\Res_h^\mathtt{C}}{s}	 
\end{equation*}
satisfies the identity $\eta_{\ell,z} = \norm[L^2(\omega_z)]{\nabla v_z}$. For alternative indicators, we refer to Remark~\ref{R:CR-computable-indicator} as well as \cite[Section~4.9]{Bonito.Canuto.Nochetto.Veeser:24}.
\end{remark}

\begin{proof}[Proof of Theorem~\ref{T:dG-error-estimator}]
Supposing the settings \eqref{dG-setting} and \eqref{Eq:dG-partition-of-unity}, Lemmas~\ref{L:dG-distance-to-conformity}, \ref{L:dG-partition-of-unity}, and \ref{L:dG-local-projections} verify \eqref{A:equiv-to-dist-to-conf}, \eqref{A:partition-of-unity} and \eqref{A:construction-Pz-local}. Hence, we can apply Theorem~\ref {T:error-estimator}  and  the claimed equivalence follows from
  \begin{align*}
    \sum_{z \in \vertices} \norm[H^{-1}(\omega_z)]{\calP_z \Res_h^{\mathtt{C}}}^2
    \eqsim \sum_{z \in \vertices}
    \norm[S_z']{\Res_h^{\mathtt{C}}}^2,
\end{align*}
which is a consequence of (iv) and (v) in \eqref{A:construction-Pz-local}. The lower bound is an immediate consequence of the definition of $\eta_{\ell,z}$ and $\scp{\Res_h^\mathtt{C}}{s} = \int_\Omega \nabla_h (u-u_h) \cdot \nabla s$ for all $s \in S_z$. 
\end{proof}

\begin{remark}[Assumptions for smoothing operator]
\label{R:dG-smoother-ass}
Theorem~\ref{T:dG-error-estimator} assumes that the smoothing operator in \eqref{Eq:dG-qopt} verifies \eqref{Eq:dG-smoother}. We emphasize that this covers high-order smoothing operator which are only invariant  on $\mathring{S}^1_1$. Although such operators are not quasi-optimal  in the sense of \cite[Definition~2.1]{Carstensen.Nataraj:22}, they lead to quasi-optimal methods \cite[Theorem~3.10]{Veeser.Zanotti:18b}. In fact, the required stability property \eqref{Eq:dG-smoother-stability} for quasi-optimality and Theorem~\ref{T:dG-error-estimator} hinges only on an local invariance  within $\mathring{S}^1_1$, which is a proper subspace of the full conforming subspace $\mathring{S}^1_\ell$ for $\ell\geq2$.
\end{remark}

\section{Estimators for the biharmonic problem}
\label{sec:Biharmonic}
%
%
% intro
In this section, we continue to illustrate the generality of the guidelines in Section~\ref{sec:posteriori-analysis} for an a~posteriori analysis of nonconforming methods. To this end, we consider the fourth-order problem of the two-dimensional biharmonic equation
with clamped boundary conditions:
\begin{equation}
	\label{Eq:Biharmonic}
	\Delta \Delta u = f \quad\text{in }\Omega \subset \R^2,
	\qquad\text{and}\qquad
	u=\partial_{\normal} u=0 \quad\text{on }\partial\Omega.
\end{equation}
Generally speaking, fourth-order problems are challenging test cases for nonconforming techniques because, in contrast to second-order problems, the conforming subspaces of the underlying discrete spaces may be so small that they are unusable for theoretical or practical purposes. To illustrate how to handle this  adverse feature within the aforementioned guidelines, we derive strictly equivalent error estimators for a quasi-optimal \emph{$C^0$-interior penalty} method  of lowest order and a quasi-optimal \emph{Morely} method. % applied to \eqref{Eq:Biharmonic}.
% weak formulation
For convenience of the reader, we recall the weak formulation of~\eqref{Eq:Biharmonic} underlying these methods:
\begin{equation*}
	\text{ for } f \in H^{-2}(\Omega)
	\text{, find $u \in \mathring{H}^2(\Omega)$ such that }
	\forall v \in \mathring{H}^2(\Omega)\;
	\int_\Omega D^2u : D^2v = \scp{f}{v},
\end{equation*}
where, in addition to the assumptions and notation of Section~\ref{sec:doma-mesh-polyn}, 
\begin{gather*}
	\mathring{H}^2(\Omega)
	= 
	\{ v \in H^2(\Omega) \mid v = \partial_\normal v  = 0 \text{ on }\partial\Omega \},
	\quad
	H^{-2}(\Omega) = \mathring{H}^2(\Omega)', 
	\\ \text{and}\quad
	%\begin{align*}
	D^2 u : D^2 v
	=
	\sum_{i,j=1}^2\frac{\partial^2u}{\partial x_i\partial x_j}\frac{\partial^2v}{\partial x_i\partial x_j}.
	%\end{align*}
\end{gather*} 

\subsection{$\mathring{H}^2$-conforming quadratics and $\HCT$ elements}
\label{S:H20-quad-HCT}
The piecewise quadratics $S^0_2$ underlie both the Morley method \cite{Morley:68} and the $C^0$-interior penalty method \cite{Brenner.Sung:05} of lowest order. In this section, we recall that  their $\mathring{H}^2(\Omega)$-conforming subspace $S^0_2 \cap \mathring{H}^2(\Omega)$ may be trivial, whence other conforming discrete functions have to be used in the a~posteriori error analysis (and in a quasi-optimal method). Furthermore, for the later purpose, we recall properties of the $H^2$-conforming $\HCT$ elements.

The following example illustrates the well-known fact from  \cite[Theorem~3]{DeBoor.DeVore:1983} that smoothness requirements for the approximate may adversely affect their approximation properties in multidimensions.
\begin{example}[No $\mathring{H}^2$-conforming quadratics]
\label{E:no-H20-conf-quads}
Subdivide \(\Omega=(0,1)^2\) for a given \(n\in \mathbb{N}\) into \(n^2\) equal sized squares and obtain a triangulation \(\grid\) by inserting in each square the  diagonal parallel to the line \(\{(x,x)\mid x\in\mathbb{R}\}\). Then \(S_2^0\cap \mathring{H}^2(\Omega)=\{0\}\). To see this, let $v \in S_2^0\cap \mathring{H}^2(\Omega)$ and consider first the square containing the origin, denoting by $T_i$, $i=1,2$, its two triangles. Factorizing $v_{|T_i} \in \poly_2$ by means of the boundary conditions yields $v = c_i a_i^2$ on $T_i$ with $c_i \in \mathbb{R}$ and $a_i \in \poly_1$, $i=1,2$. The $C^1$-transition along the diagonal  \(\{(x,x)\mid x\in\mathbb{R}\}\) implies $c_1 = c_2$ and $\nabla a_1 = \nabla a_2$ and therefore $v_{|(T_1 \cup T_2)} = 0$. Inductively repeating this argument for the other squares provides the desired identity $v =0$. 
\end{example}

% consequences for approximate distance to conformity
In the Sections~\ref{sec:CR1} and \ref{sec:dG} for the Poisson problem, the operator $\calA_h$ to approximate the distance of conformity $\norm{v_h-\Pi_V v_h}$ has values in $\mathring{S}^1_\ell = V_h \cap V $. This is algorithmically convenient as there is no need to implement additional shape functions. However, in situations with $V_h \cap \mathring{H}^2(\Omega) = \{0\}$ as in Example~\ref{E:no-H20-conf-quads}, the crucial constant $\Cav$ in Proposition~\ref{P:inds-for-dist-to-conf} vanishes for $h\to0$ whenever the discrete spaces can approximate a nontrivial exact solution.  Indeed, under these assumptions, we have
\begin{align*}
	\Cav \norm{u}
	=
	\lim_{h\to 0} \Cav \norm{\Pi_{V_h}u}
	=
	\lim_{h\to 0}  \Cav \norm{(\calA_h-\mathrm{id}_{\widetilde{V}}) \Pi_{V_h}u}
	\le
	\lim_{h\to 0}   \norm{u-\Pi_{V_h}u} = 0.
\end{align*}
Thus, a generic approach for nonconforming quadratics has to employ additional shape functions with a sufficiently large $\mathring{H}^2$-conforming subspace.

% conforming subspace and localization
Conforming discrete functions also play a role in bounding the conforming residual $\Res_h^\texttt{C}$, namely trough $\calE_h(V_h)$ in the near orthogonality of Lemma~\ref{L:conforming-residual}.  Of course, Example~\ref{E:no-H20-conf-quads} excludes also that the smoothing operator $\calE_h$ of quasi-optimal methods with piecewise quadratics has values in $V_h \cap \mathring{H}^2(\Omega)$. In fact, the smoothers of the  quasi-optimal  methods in \cite{Veeser.Zanotti:18b,Veeser.Zanotti:19b} are based upon $\HCT$ elements. Since we shall also build here on these elements, we recall them and then prove some useful properties of their nodal basis.

% HCT elements
Let
\begin{align}\label{df:HCT}
	\HCT:=\left\{w_h\in C^1(\overline{\Omega})\mid 
	{w_h}_{|T}\in S_3^2(\grid_T),~ T\in\grid\right\} \subset H^2(\Omega)
\end{align}
be the \(H^2(\Omega)\)-conforming \emph{Hsieh-Clough-Tocher} space of \cite{Clough.Tocher:65}. Hereafter, for each \(T\in\grid\), the sub-triangulation  \(\grid_T\) is obtained by drawing three lines from the barycenter \(m_T\) to the vertices of \(T\) and consists of three triangles. Each function \(w_h\in\HCT\) is uniquely defined by the values \(w_h(z)\) of the function itself and its gradient \(\nabla w_h(z)\) in the vertices \(z\in\vertices\) of \(\grid\) and by the normal derivatives \(\nabla w_h(m_F)\cdot\normal_F\) at the midpoints \(m_F\) of each \(F\in\sides\). We denote the corresponding nodal basis functions
by $\Upsilon^0_z$ (function values), $\Upsilon^1_z$, $\Upsilon^2_z$ (respective partial derivatives), \(z\in\vertices\), and $\Upsilon_F$ (normal derivatives in \(m_F\)), \(F\in\sides\). 
\begin{lemma}[Properties of $\HCT$ basis]
	\label{L:ScalingUpsilon}
	We have
	$%\begin{equation*}
	\sum_{z \in \vertices} \Upsilon^0_z = 1 \text{ in }\Omega,
	$ %\end{equation*}
	where, for all $z \in \vertices$,
	\begin{equation*}
		\operatorname{supp}\Upsilon_z^0=\omega_z,
		\quad%\text{and}\quad
		h_z^2\norm[L^\infty(\Omega)]{D^2\Upsilon_z^0}\eqsim
		h_z\norm[L^\infty(\Omega)]{\nabla\Upsilon_z^0}\eqsim\norm[L^\infty(\Omega)]{\Upsilon_z^0}\eqsim1.
	\end{equation*}
	Furthermore, for all $z \in \vertices$, $j=1,2$,  all $T\in\grid$ with $T\ni z$,  as well as  all \(F\in\sides\), $T \in \grid$ with $T \supset F$,
	\begin{equation*}
		\operatorname{supp}\Upsilon_z^j = \omega_z,
		\quad
		\norm[L^2(T)]{D^2\Upsilon_z^j} \lesssim 1,
		\qquad 
		\operatorname{supp}\Upsilon_F = \omega_F,
		\quad
		\norm[L^2(T)]{D^2\Upsilon_F} \lesssim 1.
	\end{equation*}
	All hidden constants only depend on the shape constant \(\gamma_\grid\).
\end{lemma}

\begin{proof}
The functions $\Upsilon_z^0$, $z \in \vertices$, form a partition of unity because the function $1_\Omega$ that equals $1$ in $\Omega$ satisfies $1_\Omega \in \HCT$ and $\nabla 1_\Omega = 0$ in $\Omega$. The identities for the supports are also immediate. Regarding the scaling properties, we only prove those of $\Upsilon_z^0$, $z \in \vertices$,  while those for $\Upsilon_z^j$, $z\in\vertices$, $j=1,2$, and $\Upsilon_F$, $F \in \sides$, follow from a similar reasoning; compare also with \cite[Lemma~3.14]{Veeser.Zanotti:19b}.
	
	The first two equivalences for $\Upsilon_z^0$ immediately follow from the fact that \(\Upsilon^0_z{}_{|T}\in S_3^0(\grid_T)\), which is finite dimensional and affine equivalent to a corresponding reference space.
	
	For the third or last equivalence for $\Upsilon_z^0$, however, notice that the nodal variables involve normal derivatives and are thus not affine equivalent. To overcome this, we follow the idea in the proof of~\cite[Theorem 6.1.3]{Ciarlet:2002}, which involves a similar, affine equivalent finite element: Fix \(T\in\grid\) with \(z\in T\) and observe that \(S_3^2(\grid_T)\) is also determined by the evaluations of \(p(z)\), \(\nabla p(z)\cdot(y-z)\), for \(z,y\in\vertices\cap T\) with \(z\neq y\) and \(\nabla p(m_F)\cdot(m_T-m_F)\) for
	\(F\in\sides(T) :=\{F\in\sides\mid F\subset T\}\). We denote the corresponding nodal basis functions by \(\tilde\Upsilon_z\),
	\(\tilde \Upsilon_z^y\),  \(z,y\in\vertices\cap T\) with \(z\neq y\) and \(\tilde\Upsilon_F\) for \(F\in\sides(T)\), respectively. As this new finite element is affine equivalent, we derive in particular
	\begin{equation}
		\label{scaling-for-similar-affequiv}
		\norm[L^\infty(T)]{\tilde\Upsilon_z} \eqsim 1
		\quad\text{and}\quad
		\norm[L^\infty(T)]{\tilde\Upsilon_F} \eqsim 1 
	\end{equation}
	by standard arguments. To conclude, we relate the new basis functions with the original ones. Given \(F \in\sides(T)\) with \(z\in F\), let \(y_F\in\vertices\cap F\) with \(y\neq z\) and,  from the definition of \(\Upsilon_z^0\), deduce the representation
	\begin{align*}
		\Upsilon_z^0
		=
		\tilde \Upsilon_z + \frac32 \sum_{F\in\sides(T),z\in F} \frac{(z-y_F)}{|z-y_F|^2}% \cdot\tangent_F
		% \tangent_F
		\cdot(m_T-m_F)\tilde\Upsilon_F. 
	\end{align*}
	
	% Indeed, this is based on the observation that
	% \(\Upsilon_z^0(m_F)\cdot\normal_F=0\) implies that
	% \(\nabla\Upsilon_z^0(m_F)=\nabla\Upsilon_z^0(m_F)\cdot\texsf{t}_F\texsf{t}_F\).
	% Moreover, thanks to the nodal conditions on \(\Upsilon_z^0\), we have
	% that   \(3t^2-2t^3:=\Upsilon_z^0(tz+(1-t)y))\in\poly_3([0,1])\)
	% and thus
	% \begin{align*}
		%   \frac32=\frac{d}{dt}(3t^2-2t^3)_{|t=\frac12}=\nabla\Upsilon_z^0(m_F)\cdot(z-y_F)=\nabla\Upsilon_z^0(m_F)\cdot
		%   \texsf{t}_F(z-y_F)  \texsf{t}_F.
		% \end{align*}
	% Now the claimed representation follows by observing that
	% \((z-y_F)  \cdot\texsf{t}_F=\pm |x-y_F|\).
Hence the desired equivalence $\norm[L^\infty(T)]{\Upsilon_z^0} \eqsim 1$ follows from \eqref{scaling-for-similar-affequiv} and the relationship	\(|m_T-m_F|/|z-y_F|\eqsim1\), where, as before, the hidden constants depend only on the shape constant \(\gamma_\grid\).
\end{proof}

\subsection{A quasi-optimal $C^0$ interior penalty method}
\label{sec:qo-C0}
%
% intro
In this section, we derive a strictly equivalent error estimator for
the  quasi-optimal variant~\cite[(3.51)]{Veeser.Zanotti:18b} of the $C^0$ interior
penalty method~\cite{Brenner.Sung:05} for the 2-dimensional biharmonic
problem~\eqref{Eq:Biharmonic}. 

% quasi-optimal C0 interior penalty method
We start by recalling the above method. In the notation of Section~\ref{sec:doma-mesh-polyn} and given a stabilization parameter $\sigma>0$, define  a discrete bilinear form on the space $\mathring{S}^1_2=S^1_2 \cap \mathring{H}^1(\Omega)$ by
\begin{equation}
	\label{C0:a_h}
	\begin{multlined} 
		a_h(v_h,w_h)
		:=
		\int_\Omega D^2_h v_h : D^2_h w_h
		+
		\int_\Sigma %\sum_{F\in\sides} \int_F
		\frac\sigma h \jump{\partial_\normal v_h} \jump{\partial_\normal w_h}
		\\
		{} - \int_\Sigma %\sum_{F\in\sides}\left( \int_F
		\mean{\partial_{\normal}^2 v_h} \jump{\partial_\normal w_h}
		+
		% \int_F
		\jump{\partial_\normal v_h} \mean{\partial^2_\normal w_h}
		% \right)
		.
	\end{multlined} 
\end{equation}
Hereafter, the operator \(D_h^2\) evaluates the broken Hessian, i.e., for a suitable $v$, we have
\begin{align*}
	(D^2_hv)_{|T}=D^2(v_{|T})\quad\text{for all } T\in\grid.
\end{align*}
The bilinear form $a_h$ is coercive  provided \(\sigma>\sigma_*\) for some \(\sigma_*>0\) depending on the shape constant \(\gamma_\grid\); see \cite[Lemma~7]{Brenner.Sung:05}. The discrete problem then reads
\begin{equation}
	\label{Eq:C0-qopt}
	\text{for \(f\in H^{-2}(\Omega)\), find $u_h \in \mathring{S}_2^1$ such that }
	\forall v_h \in \mathring{S}^1_2\;
	a_h(u_h,v_h) = \scp{f}{\calE_{\COip} v_h},
\end{equation}
where the smoothing operator \(\calE_{\COip}:\mathring{S}_2^1\to \mathring{H}^2(\Omega)\) is constructed with the help of the $\HCT$ element and  satisfies the following properties: for all \(v_h\in \mathring{S}_2^1\), $F \in \sides$, $T \in \grid$, 
\begin{subequations}
	\label{eq:EC0}
	\begin{gather}
		\label{eq:EC0cons}
		\int_F \nabla\calE_{\COip} v_h=\int_F \mean{\nabla v_h},
		\\ \label{eq:EC0stab}
		\begin{aligned}
			h_T^{-4} \norm[L^2(T)]{v_h-\calE_{\COip} v_h}^2
			+
			h_T^{-2}\norm[L^2(T)]{\nabla(v_h-\calE_{\COip} v_h)}^2
			+
			\norm[L^2(T)]{D^2\calE_{\COip} v_h}^2
			\\
			\lesssim
			\norm[L^2(\omega_T)]{D^2_hv_h}^2
			+
			\sum_{F\in \sides, F\cap T\neq\emptyset} \int_{F}\frac{1}{h}|\jump{\partial_\normal v_h}|^2;
		\end{aligned}
	\end{gather}
\end{subequations}
see~\cite[Proposition 3.17]{Veeser.Zanotti:18b}. 
According to~\cite[Theorem~3.18]{Veeser.Zanotti:18b}, this method is quasi optimal whenever $\sigma > \sigma_*$ and, as the dG methods from Section~\ref{sec:dG}, not overconsistent, i.e., we have Proposition~\ref{P:QuasiOptMethods} with \(\delta_h \lesssim 1\), where the hidden constants depend on the shape constant \(\gamma_\grid\)  and the stabilization parameter $\sigma$.

% setting for abstract framework
The weak formulation of the biharmonic problem~\eqref{Eq:Biharmonic} and the method~\eqref{Eq:C0-qopt} fit into the abstract framework of
Section~\ref{sec:qo-methods} with 
\begin{equation}
	\label{setting-for-Biharmonic-and-C0}
	\begin{gathered}
		\tV=\mathring{H}^2(\Omega)+\mathring{S}_2^1
		\quad\text{with}\quad
		V=\mathring{H}^2(\Omega)
		\quad\text{and}\quad
		V_h=\mathring{S}_2^1,
		\\
		\ta(v,w)
		=
		\int_\Omega D^2_h v : D^2_h w
		+
		\int_\Sigma % \sum_{F\in\sides} \int_F
		\frac\sigma h \jump{\partial_\normal v} \jump{\partial_\normal w},
		\quad
		\|v\|=\norm[\COip]{v}:=\sqrt{\ta(v,v)}
		%\\
		%   \|v\|=\norm[\COip]{v}:=\left(\norm[L^2(\Omega)]{D_h^2v}^2+\sum_{F\in\sides}
		%   \int_F \frac\sigma h \jump{\partial_\normal v}^2
		%    \right)^{\frac12},
		\\
		%    \begin{aligned}
			a_h \text{ as in }\eqref{C0:a_h},
			%(v_h,w_h)&= \int_\Omega D^2_h v_h : D^2_h w_h  
			%   + \sum_{F\in\sides} \int_F \frac\sigma h \jump{\partial_\normal v_h} \jump{\partial_\normal w_h}  \\
			%   &\quad-\sum_{F\in\sides}\left\{\int_F \mean{\partial_{\normal}^2 v_h} \jump{\partial_\normal w_h}  
			%    +\int_F \jump{\partial_\normal v_h} \mean{\partial^2_\normal w_h} \right\},
			% \end{aligned}
		%  \\
		\quad\text{and}\quad
		\calE_h=\calE_{\COip}.
	\end{gathered}
\end{equation}

% intro to aposteriori analysis
To derive a strictly equivalent error estimator, we follow the guidelines of Section~\ref{sec:posteriori-analysis} and start by establishing the main assumptions therein.  For the first assumption \eqref{A:equiv-to-dist-to-conf} regarding the approximate distance to conformity, we take the observations of Section~\ref{S:H20-quad-HCT}  into account and base a simplified averaging operator upon the $\HCT$ element, already used for the smoother $\calE_\mathtt{C0}$. More precisely, we assign to each interior vertex \(z\in\vertices^i\) an element \(T_z\in\grid\) with \(z\in T\) , to each interior edge \(F\in\sides^i\) an element \(T_F\) with \(F\subset T\), and define
\begin{align}
	\label{df:AC0}
	\calA_{\COip}  v_h :=
	\!\sum_{z \in \vertices^i}\!\Big(v_h(z) \Upsilon^0_z +
	\sum_{j=1}^2 \frac{\partial}{\partial x_j}({v_h}_{|T_z})(z) \Upsilon^j_z \Big)
	+ \sum_{F\in\sides^i} \partial_{\normal}(v_{h|T_F})
	(m_F)\Upsilon_F. 
\end{align}
This  averaging operator is locally computable, thus in line with Remark~\ref{R:dist-to-conf-averaging}. It also coincides with the operator \cite[(3.48)]{Veeser.Zanotti:18b}, which is helpful in the following lemma specifying Proposition~\ref{P:inds-for-dist-to-conf}.

\begin{lemma}[Approximate $\|\cdot\|_\mathtt{C0}$-distance to $\mathring{H}^2$ by averaging]
	\label{L:C0dist2conf}
	In the setting \eqref{setting-for-Biharmonic-and-C0}, the operator  $\mathcal{A}_{\COip}$ in \eqref{df:AC0} satisfies assumption \eqref{A:equiv-to-dist-to-conf}  and \(  \Cav \gtrsim \sqrt{{\sigma}/{(1+\sigma)}} \): for all $v_h \in \mathring{S}_2^1$, we have
	\begin{align*}
		\sqrt{\frac{\sigma}{1+\sigma}}
		\norm[\COip]{\calA_{\COip}v_h-v_h}
		\lesssim
		 \inf_{v \in \mathring{H}^2(\Omega)}\norm[\COip]{v-v_h}
		 \leq
		 \norm[\COip]{\calA_{\COip}v_h-v_h}.
	\end{align*}
\end{lemma}

\begin{proof}
	We first verify \eqref{A:equiv-to-dist-to-conf} for $\calA_{\COip}$. Since $\HCT \subset H^2(\Omega)$ and definition \eqref{df:AC0} does not involve any basis functions associated with the boundary $\partial\Omega$, we have $\calA_{\COip} v_h \in \mathring{\HCT} := \HCT \cap \mathring{H}^2(\Omega)$ for all $v \in \mathring{S}^1_2$. Clearly, $\calA_{\COip}$ is invariant on $\mathring{S}^1_2 \cap \mathring{H}^2(\Omega) = \mathring{S}^1_2 \cap \mathring{\HCT}$. Assumption \eqref{A:equiv-to-dist-to-conf} is thus verified, as well as the second inequality of the claimed equivalence.
	
	To show the first inequality, we combine stability properties of $\calA_{\COip}$ and properties of the basis function in Lemma~\ref{L:ScalingUpsilon}. Given $v_h \in \mathring{S}^1_2 \subset C^0(\bar{\Omega})$ and an element \(T\in\grid\), we have $v_h{}_{|T} \in \poly_2(T) \subset S^2_3(\grid_T)$ and 
	\begin{multline*}
		%D^2_h
		(v_h - \calA_{\COip}(v_h))_{|T}
		=
		\sum_{z\in\vertices^i \cap T} \sum_{j=1}^2 \frac{\partial}{\partial x_j}({v_h}_{|T}-\calA_{\COip} v_h)(z) %D^2
		\Upsilon_{z|T}^{j}
		\\
		{} + \sum_{F\in\sides,\,F\subset T} \partial_{\normal}(v_{h|T} -\calA_{\COip}v_h) (m_F) %D^2
		\Upsilon_{F|T},
	\end{multline*}
	whence
	\begin{multline*}
		\norm[L^2(T)]{D^2_h(v_h - \calA_{\COip}(v_h))}
		\le
		\sum_{z \in \vertices^i \cap T} |\nabla({v_h}_{|T}-\calA_{\COip}v_h)(z)| \,
		\left( \sum_{j=1}^2\norm[L^2(T)]{D^2\Upsilon_z^{j}}^2\right)^{\frac12}
		\\
		{} +\sum_{F\in \sides,F\subset T} |\partial_\normal ({v_h}_{|T}-\calA_{\COip}v_h)(m_F)| \, \norm[L^2(T)]{D^2\Upsilon_F}.
	\end{multline*}
	Thanks to the definition \eqref{df:AC0}, we observe $ \nabla(\calA_{\COip}v_h)(z) = \nabla (v_h{}_{|T_z})(z)$ as well as $ \partial_\normal (\calA_{\COip}v_h)(m_F) = \partial_\normal (v_h{}_{|T_F})(m_F)$.  Furthermore, the continuity of \(v_h\) entails |\(\jump{\nabla v_h}| = |\jump{\partial_{\normal}v_h}| \) on the skeleton $\Sigma$.  We thus obtain
	\begin{subequations}
		\begin{align}
			\label{Eq:Dvh-DAvh<DJumps}
			|\nabla({v_h}_{|T}-\calA_{\COip}v_h)(z)|
			&\lesssim
			\sum_{F\in\sides_z^i} \|h^{-\frac12}\jump{\partial_{\normal}  v_h}\|_{L^2(F)},
			\quad
			z\in\vertices^i,
			\\
			% =
			% \sum_{F\in\sides,z\inF}\|h^{-\frac12}\jump{\partial_\normal
				% v_h}\|_{L^2(F)};
			%\intertext{and}
			|\nabla({v_h}_{|T}-\calA_{\COip}v_h)(m_F)|
			&\lesssim
			\|h^{-\frac12}\jump{\partial_{\normal} v_h}\|_{L^2(F)},
			\quad
			F\in\sides.
		\end{align}
	\end{subequations}
	We insert these two inequalities in the preceding one and use the scaling properties in Lemma~\ref{L:ScalingUpsilon}. As each side \(F\in\sides\) is contained in at most two of the sets \(\{\sides_z^i\}_{z\in\vertices^i}\), we conclude the proof by arriving at
	\begin{align*}
		\norm[\COip]{\calA_{\COip}v_h-v_h}^2
		\lesssim
		(1+\sigma)\sum_{F\in\sides}\|h^{-\frac12}\jump{\partial_{\normal} v_h}\|_{L^2(F)}^2
		\le
		\frac{1+\sigma}{\sigma}\norm[\COip]{v-v_h}^2,
	\end{align*}
	where  \(v\in \mathring{H}^2(\Omega)\) is arbitrary.
\end{proof}

The proof of Lemma~\ref{L:C0dist2conf} reveals that properly scaled jumps of the normal derivative of \(v_h\) provide an alternative estimator for the distance to \(\mathring{H}^2(\Omega)\),  which is widely used in the literature~\cite{Brenner.Gudi.Sung:10}.
% \begin{align}\label{Eq:local-C0dist2conf_by_jumps}
	%   \begin{aligned}
		%     \norm[L^2(T)]{D^2_h(v_h -
			%       \calA_{\COip}(v_h))}^2&\lesssim\sum_{z\in\vertices\cap
			%       T}\sum_{F\in\sides_z} \int_F \frac{1}{h} \left|\jump{\nabla_h
			%         v_h}\right|^2.
		%   \end{aligned}
	% \end{align}

%
\begin{corollary}[Approximate $\|\cdot\|_\mathtt{C0}$-distance to $\mathring{H}^2$ by jumps]
	\label{C:C0JumpDist2conf}
	For \(v_h\in\mathring{S}_2^1\), we have 
	\begin{align*}
		\int_\Sigma \frac{\sigma}{h} \left|\jump{\partial_\normal
			v_h}\right|^2
			\lesssim 
			\inf_{v\in
			\mathring{H}^2(\Omega)}\norm[\COip]{v_h-v}^2
			\lesssim
			\int_\Sigma \frac{1+\sigma}{h} \left|\jump{\partial_\normal
				v_h}\right|^2.
	\end{align*}
	
\end{corollary}

Next, we prepare the localization of the conforming residual by establishing assumption \eqref{A:partition-of-unity}. Motivated by Lemma~\ref{L:ScalingUpsilon} and the embedding $\mathring{H}^2(\Omega)\hookrightarrow C(\bar\Omega)$, we choose
\begin{equation}
	\label{Eq:C0-partition-of-unity}
	\Index = \vertices,
	\qquad
	\Phi_z = \Upsilon_z^0,
	\qquad
	V_z = \mathring{H}^2(\omega_z),
	\qquad
	\calI_h = \calI_{L},
\end{equation}
where \(\calI_{L}\) denotes the Lagrange interpolation operator onto \(\mathring{S}_2^1\).

\begin{lemma}[Partition of unity in $H^2$ for biharmonic $\COip$]
	\label{L:C0-partition-of-unity}
	In the  setting  \eqref{setting-for-Biharmonic-and-C0}, the choices  \eqref{Eq:C0-partition-of-unity} satisfy assumption \eqref{A:partition-of-unity} with constants $\Cloc,\Ccol \lesssim 1$. 
	%  For the \(C^0\)-interior penalty method~\eqref{Eq:C0ip-method}
	% the index set \(\Index=\vertices\), functions
	% \(\Phi_z=\Upsilon_z^0\), subspaces 
	% \(V_z=\mathring{H}^2(\omega_z)\), \(z\in\vertices\), and the nodal Lagrange interpolation
	% operator 
	% \(\mathcal{I}_h=\mathcal{I}_{L}\) into \(\mathring{S}_2^1\) satisfy
	% Assumption~\eqref{A:partition-of-unity} with constants
	% \(\Ccol =d+1\) and \(\Cloc>0\) only depending on the shape
	% constant \(\gamma_\grid\). 
\end{lemma}
\begin{proof}
	% checking (i) and (ii)
	Since \(\Upsilon^0_z\in W^{2,\infty}(\Omega)\) and  \(\operatorname{supp}\Upsilon_z^0=\omega_z\), condition~(i) of \eqref{A:partition-of-unity} holds with \(V_z=\mathring{H}^2(\omega_z)\) and \(\Index=\vertices\). Similarly to the proof in Lemma~\ref{L:CR-partition-of-unity}, condition~(ii) follows with \(\Ccol=3\) from the fact that  each mesh element is contained in three stars. % \(\omega_z\), \(z\in\vertices\).
	
	% checking (iii)
	It remains to check condition (iii) of \eqref{A:partition-of-unity}. Given any \(v\in \mathring{H}^2(\Omega)\) and $z \in \vertices$, we may write
	\begin{equation*}
		\norm[\COip]{\big((v - \calE_{\COip}(\mathcal{I}_{L} v) \big)\Upsilon_z^0} 
		\leq
		\norm[\COip]{(v-\mathcal{I}_{L} v) \Upsilon_z^0}
		+
		\norm[\COip]{ \big( \mathcal{I}_{L} v - \calE_{\COip} (\mathcal{I}_{L} v) \big) \Upsilon_z^0}.
	\end{equation*}
	The first term on the right-hand side can be bounded by means of trace inequalities, standard interpolation error estimates, the stability bound $\norm[L^2(\omega_z)]{D^2_h(\calI_ L v)} \lesssim  \norm[L^2(\omega_z)]{D^2 v}$ and the scaling properties of $\Upsilon_z^0$ in Lemma~\ref{L:ScalingUpsilon}. For the second term, we use again Lemma~\ref{L:ScalingUpsilon}, \eqref{eq:EC0stab}, a scaled trace inequality and  again error and stability bounds of $\calI_{L}$ to derive 
	\begin{multline*}
		\norm[\COip]{ \big( \mathcal{I}_{L} v - \calE_{\COip} (\mathcal{I}_{L} v) \big) \Upsilon_z^0}^2\\
		\begin{aligned}
			&\lesssim
			\norm[{L^2(\omega_z)}]{ h^{-2} \big( \mathcal{I}_{L}v-\calE_{\COip}(\mathcal{I}_{L} v) \big)}^2
			+
			\norm[{L^2(\omega_z)}]{ h^{-1} \nabla \big( \mathcal{I}_{L} v-\calE_{\COip}(\mathcal{I}_{L} v) \big)}^2
			\\
			&\quad
			{} +
			\norm[{L^2(\omega_z)}]{ D^2_h \big( \mathcal{I}_{L} v-\calE_{\COip}(\mathcal{I}_{L} v) \big)}^2
			+
			\sigma\sum_{F\in\sides_z}\int_{F}\frac1h\jump{\partial_\normal\mathcal{I}_{L} v}^2
			\\
			&\lesssim
			\norm[L^2(\omega_z^+)]{D^2 v}^2
%			+
%			\sigma \sum_{F\in\sides_z^{+}}\int_{F}\frac1h\jump{\partial_\normal\mathcal{I}_{L} v}^2,
		\end{aligned}
	\end{multline*}
	%  In particular, the last inequality follows as in \cite[Proposition 3.17]{Veeser.Zanotti:18b}.
	where \(\sides_z^{+}=\bigcup\{\sides_y\mid  y\in \vertices\cap\omega_z\}\). We collect the previous bounds and sum over all vertices $z \in \vertices$. This finishes the proof as  the stars \(\omega_z^+\), \(z\in\vertices\), overlap finitely many times and, for each edge \(F\in\sides\), we have $\#\{ z \in \vertices \mid \sides_z^+ \ni F\} \lesssim 1$. 
\end{proof}

% As a consequence we conclude
% the following localization result for the conforming residual.
% \begin{corollary}
	%  Let $u\in \mathring{H}^2(\Omega)$ be the weak solution of~\eqref{Eq:Biharmonic}, 
	%  and let $u_h\in\mathring{S}_2^1$ be the solution of the quasi-optimal
	%  \(C^0\) interior penalty method~\eqref{eq:qo-C0ip}. Then we have
	%  \begin{align*}
		%    \frac1{d+1}\sum_{z\in\vertices}\norm[H^{-2}(\omega_z)]{\Res_h^{\mathtt{C}}}^2
		%    &\le\norm[H^{-2}(\Omega)]{\Res_h^{\mathtt{C}}}
		%    \\
		%    &\le 2 \Cobl^2 \delta_{\COip}^2 \norm[\COip]{ u_h - \Pi_{\mathring{H}^2(\Omega)} u_h}^2+2\Cloc
		%    \sum_{z\in\vertices}\norm[H^{-2}(\omega_z)]{\Res_h^{\mathtt{C}}}^2,
		%  \end{align*}
	%  where the constants \(\Cobl\le \normtr{A_h}
	%  \normtr{\calI_h}\) and \(\Cloc\) and the consistency measure
	%  \(\delta_{\COip}=\delta_h>0\) from
	%  Proposition~\ref{P:QuasiOptMethods}\eqref{QuasiOptMethods-ConsistencyMeasure} 
	%  only depend on the shape constant 
	%  \(\gamma_\grid\). 
	% \end{corollary}

% \begin{proof}
	%   Thanks to Lemma~\ref{L:C0-partition-of-unity}, the assertion readily
	%   follows from Lemma~\ref{L:localizating-conf-res}. The boundedness of
	%   the constants \(\Cobl\le \normtr{A_h}
	%  \normtr{\calI_h}\), \(\Cloc\) and 
	%  \(\delta_{\COip}\) is proved in \cite[\S 3.4]{Veeser.Zanotti:18b}.
	% \end{proof}

% establishing (H3)
Our last preparatory step is the choice of simple functionals and test functions verifying \eqref{A:construction-Pz-local}. To this end,
planning for a conforming approach, i.e.\ \(S_z\subset \mathring{H}^2(\omega_z)\), we first analyze the structure of the space \(\tA(\mathring{S}_2^1))_{|\mathring{H}^2(\omega_z)}\). Let $v_h\in\mathring{S}_2^1$ be a discrete trial  function and $w\in \mathring{H}^2(\omega_z)$ be a local test function. Integration by parts gives
\begin{equation*}\begin{aligned}
		\label{Eq:BiHarmPI}
		\ta(v_h,w)
		&=
		\int_\Omega D_h^2 v_h:D^2 w
		% &= \sum_{T\in\grid_z}\int_T
		% \sum_{i,j=1}^2 \frac{\partial^2}{\partial x_i\partial x_j}
		% v_h \frac{\partial^2}{\partial x_i\partial x_j} v
		% \\
		=
		\sum_{F\in\sides^i_z} \int_F \jump{D^2_h v_h\normal }\cdot\nabla w
		-
		\sum_{T\in\grid_z}\int_T \nabla \Delta v_h\cdot\nabla w,
\end{aligned}\end{equation*}
where the second sum on the right-hand sides vanishes. To simplify the first sum, we shall use an orthogonal decomposition on the skeleton $\Sigma_z = \bigcup_{F \in \sides_z^i} F$ around $z$. For any edge $F \in \sides_z^i$, denote by $y_z^F \in \vertices$ the vertex of $F$ opposite to $z$ and let $\tangent_z:\Sigma_z \to \R^2$ be the tangent field defined by
\begin{equation}
\label{def-t_z}
 \tangent_z{}_{|F} = \frac{z-y_z^F}{|z-y_z^F|},
\quad
 F \in \sides_z^i.	
\end{equation}
We then may write
\begin{equation*}
	\jump{D^2_h v_h \normal}\cdot\nabla w
	=
	\jump{D_h^2 v_h \normal\cdot\normal}\nabla w\cdot\normal
	+
	\jump{D_h^2 v_h \normal\cdot\tangent_z}\nabla w\cdot\tangent_z,
\end{equation*}
where the jumps are edge-wise constant. Furthermore, thanks to $y_z^F \in \partial\omega_z$ and $w \in \mathring{H}^2(\omega_z) \subset C^0(\bar{\omega}_z)$, we observe $w(y_z^F)=0$, which leads to  $\int_F \nabla w \cdot \tangent_z = w(z)$. Combining these observations, we arrive at the representation
\begin{equation}
\label{eq:StructOfDC0}
	\ta(v_h,w)% \int_\Omega D_h^2 v_h:D^2 v
	=
	\left(
	\int_{\Sigma_z}  \jump{D_h^2 v_h \normal\cdot\tangent_z} 
	\right)  w(z)
	+
	\sum_{F\in\sides_z^i} \jump{D_h^2 v_h \normal\cdot\normal}_{|F}\int_F \partial_{\normal} w,
\end{equation}
which, as a functional of $w \in \mathring{H}^2(\omega_z)$,  is given by constants associated with the vertex \(z\) and the edges in $\sides_z^i$.  Note that if $z \in \vertices \setminus \vertices^i$ is a boundary vertex, then we always have $w(z)=0$.
We thus define % This suggests for the simple functionals 
% \(z\in\vertices\), the choice
\begin{equation}
	\label{Eq:hatDzBi}
	\begin{multlined}
		\widehat D_z
		:=
		\Big\{\chi\in H^{-2}(\Omega) \mid
		\chi(v) = r_zv(z)+\sum_{F \in \sides_z^i} r_F\int_F \partial_\normal v,
		\\
		\text{with}~ r_F \in \R,~F\in \sides_z^i,~\text{and}~r_z\in
		\R~\text{if}~z\in\vertices^i~\text{and}~r_z=0~\text{else} %\text{if}~z\in\vertices\setminus\vertices^i
		\Big\}
	\end{multlined}
\end{equation}
and take
\begin{subequations}
	\label{Eq:DzSzBiC0}
	\begin{align}\label{Eq:DzBiC0}
		D_z=\widehat D_z{}_{|\mathring{H}^2(\omega_z)}
	\end{align}
	as simple functionals. To choose the simple test functions, we employ  some basis functions of the $\HCT$ space in \eqref{df:HCT}, writing $\Upsilon_z := \Upsilon_z^0$ for notational convenience.  More precisely, we set
	\begin{equation}
		\label{Eq:SzBiC0}
		S_z
		:=
		\operatorname{span}\left\{\Upsilon_F\mid F\in \sides_z^i\right\}
		\oplus
		\begin{cases}
			\operatorname{span} \Upsilon_z &\text{if } z\in\vertices^i,
			\\
			\{0\} &\text{if } z \in \vertices \setminus \vertices^i,
		\end{cases}
	\end{equation}
	which is locally conforming, i.e.\  we have $S_z  \subset \mathring{H}^2(\omega_z)$. 
\end{subequations}
As the following remark shows, this is actually  a convenient choice because, using the  $H^{-2}(\omega_z)$-functionals
\begin{equation}
\label{def-of-chi_z,chi_F}
	\langle{\chi_z}, v\rangle
	:=
	v(z),
	\qquad\text{and}\qquad
	\langle\chi_F,v\rangle
	:=
	\int_F\partial_\normal v,\quad F\in\sides_z^i,
\end{equation}
we can choose a basis of $D_z$ that is almost dual to the basis functions in $S_z$ suggested by \eqref{Eq:SzBiC0}.
\begin{remark}[Dual bases in simple pairs for biharmonic $\COip$]
	\label{R:C0BiorthogonalSystem}
	%
	% interior vertex
	Consider first the case of an interior vertex $z \in \vertices^i$. Then $\chi_z$, $\chi_F$, $F \in \sides_z^i$ is a basis of $D_z$ and
	\(\Upsilon_z\), \(\Upsilon_F\), \(F\in \sides_z^i\), is one in $S_z$.  The properties of the $\HCT$ nodal variables and straight-forward calculations yield, for all $F, F' \in \sides_z^i$,
	\begin{alignat*}{2}
		\langle \chi_z,\Upsilon_z\rangle
		&=
		\Upsilon_z(z)=1,
		&\qquad
		\langle\chi_z,\Upsilon_F\rangle
		&=
		\Upsilon_F(z)=0,
		\\
		\langle \chi_F,\Upsilon_z\rangle
		&=
		\int_F\partial_\normal\Upsilon_z=0,
		& 
		\langle \chi_F,\Upsilon_{F'}\rangle
		&=
		\int_F\partial_\normal\Upsilon_{F'} = \frac23 h_F\delta_{FF'}.
	\end{alignat*}
	Thus, the bases $\chi_z$, $\chi_F$, $F \in \sides_z^i$, and  \(\Upsilon_z\), \(\frac{3}{2}h_F^{-1}\Upsilon_F\), \(F\in \sides_z^i\),  are dual.
	%
	%  boundary vertex
	If $z\in\vertices\setminus\vertices^i$ is a boundary vertex, we only have to discard $\chi_z$ and $\Upsilon_z^0$ from the previous case.
\end{remark}

\begin{lemma}[Local projections for biharmonic $\COip$]
\label{L:C0construction-Pz-local}
In the settings \eqref{setting-for-Biharmonic-and-C0} and \eqref{Eq:C0-partition-of-unity},  the choices \eqref{Eq:DzSzBiC0} satisfy assumption \eqref{A:construction-Pz-local} with \(\CL\lesssim1\) and \(\CNL=1\). Consequently, \eqref{Eq:construction-Pz-local} defines projections $\calP_z: H^{-2}(\omega_z) \to D_z \subset H^{-2}(\omega_z)$ with  
	\begin{equation*}
		\forall g \in H^{-2}(\Omega)
		\quad
		\sum_{z \in \vertices}  \norm[H^{-2}(\omega_z)]{\calP_z g}^2
		\lesssim
		\sum_{z \in \vertices} \norm[H^{-2}(\omega_z)]{g}^2.
	\end{equation*}
\end{lemma}

\begin{proof}
Conditions (i), (ii) and (v) of \eqref{A:construction-Pz-local} immediately follow from \(S_z\subset \mathring{H}^2(\omega_z)\) and~\eqref{eq:StructOfDC0}. In view of Remark~\ref{R:C0BiorthogonalSystem}, we have \(\dim S_z= \dim D_z < \infty\) for all \(z\in\vertices\), which implies (iii) and it remains to prove (iv). To this end, we combine the duality in Remark~\ref{R:C0BiorthogonalSystem} with the scaling properties of the involved basis functions, considering only the case of an interior vertex \(z\in\vertices^i\). Given a simple functional \(\chi\in D_z\) and $v \in \mathring{H}^2(\omega_z)$, we write 
	\begin{align*}
		\langle{\chi,v}\rangle
		=
		r_zv(z)+\sum_{F \in \sides_z^i} r_F\int_F \partial_\normal v,
	\end{align*}
	and, by means of scaled trace and Poincar\'e inequalities, derive
	\begin{align}
		\label{Eq:BiScaledResiduals}
		\langle\chi,v\rangle
		\lesssim
		\Big(h^2_z |r_z|^2 +\sum_{F\in\sides_z^i} h_z^2 |r_F|^2 \Big)^{\frac12} \norm[L^2(\omega_z)]{D^2v}.
	\end{align}
	In addition, the duality of Remark~\ref{R:C0BiorthogonalSystem} and the scaling properties in Lemma~\ref{L:ScalingUpsilon} ensure 
	\begin{equation*}
		h_z |r_z|
		=
		h_z|\langle \chi,\Upsilon_z \rangle|
		\lesssim
		\norm[\Dual{S_z}]{\chi} ,
		\qquad \text{and} \qquad
		h_z^{\frac12} |r_F|
		=
		h_z|\langle \chi,\tfrac32 h_F^{-1}\Upsilon_F\rangle|
		\lesssim
		\norm[\Dual{S_z}]{\chi},
	\end{equation*}
	with hidden constants depending only on the shape constant \(\gamma_\grid\). Inserting these estimates in~\eqref{Eq:BiScaledResiduals} and taking the supremum over all \(v\in \mathring{H}^2(\omega_z)\) finishes the proof of \eqref{A:construction-Pz-local}.
\end{proof}

Given the preparation by Lemmas~\ref{L:C0dist2conf}, \ref{L:C0-partition-of-unity}, and \ref{L:C0construction-Pz-local}, we are now in position to derive an error estimator that is strictly equivalent to the error of the $C^0$ interior penalty method.

\begin{theorem}[Estimator for biharmonic $\COip$]
	\label{TH:mainC0}
	Let $u\in \mathring{H}^2(\Omega)$ be the weak solution of the biharmonic problem \eqref{Eq:Biharmonic} and   $u_h\in\mathring{S}_2^1$ be its quasi-optimal approximation of the $C^0$ interior penalty method in \eqref{Eq:C0-qopt} with sufficiently large stabilization parameter $\sigma$. Given tuning constants \(C_1,C_2>0\), define
	\begin{multline*}
		\Est_{\COip}^2
		:=
		\Ncf_{\COip}^2 + C_1^2 \eta_{\COip}^2 + C_2^2\Osc_{\COip}^2,  \text{ where }
		\\
		\begin{aligned}
			\Ncf_{\COip}^2
			&:= 
			\norm[\mathtt{C0}]{u_h - \calA_{\COip} u_h }^2,
			\\ 
			\eta_{\COip}^2
			&:= 
			\sum_{T \in \grid} \eta_{\COip,T}^2
			\text{ with  }
			\eta_{\COip,T}
			:=
			\max_{ K \in \vertices_T^i \cup \sides_T^i}
			\frac{|\scp{f}{\Upsilon_K} - \int_\Omega D^2_h u_h : D^2 \Upsilon_K |}{\|D^2 \Upsilon_K\|_{L^2(\Omega)}},  
			\\
			\Osc_{\COip}^2
			&:= 
			\sum_{z\in\vertices} \norm[H^{-2}(\omega_z)]{f - \calP_z f}^2, 
		\end{aligned}
\end{multline*}
swith the averaging operator $\calA_{\COip}$ from \eqref{df:AC0}, $\vertices_T^i = \vertices^i \cap T$, $\sides_T^i = \{ F \in \sides^i \mid F \subset T \}$, the $\HCT$ nodal basis functions $\Upsilon_z := \Upsilon_z^0$ and $\Upsilon_F$, as well as the
	local projections \( \calP_z:H^{-2}(\omega_z) \to D_z \) from Lemma~\ref{L:C0construction-Pz-local}.
	
	This estimator quantifies the error by
	\begin{equation*}
		\underline{C} \Est_{\COip}
		\leq
		\norm[\COip]{u-u_h}
		\leq
		\overline{C} \Est_{\COip}
		\quad\text{and}\quad
		\eta_{\COip,T} 
		\leq
		\norm[L^2(\widetilde{\omega}_T)]{D^2_h(u-u_h)}
	\end{equation*}
	where $\widetilde{\omega}_T = \bigcup_{z  \in \vertices_T^i} \omega_z $ and the equivalence
	constants $\underline{C}$, $\overline{C}$ depend only on the shape constant \(\gamma_\grid\), the stabilization constant $\sigma$,  and the tuning constants \(C_1, C_2\).
	
Finally, if $f \in L^2(\Omega)$, we have
\begin{equation*}
		\Osc_{\COip}^2
		\lesssim
		\sum_{T \in \grid} h_T^4 \norm[L^2(T)]{f}^2,
\end{equation*}
where the latter is formally of higher order.
\end{theorem}

\begin{remark}[Estimator variants of biharmonic $\COip$]
\label{R:oscHiOC0}
%
% alternative approximate distance to conformity
In view of  Corollary~\ref{C:C0JumpDist2conf}, we may remove the volume term from the indicator $\Ncf_{\COip}$ in Theorem~\ref{TH:mainC0} at the expense of an additional tuning constant.
	
% alternative local projection with `faster' oscillation
In \eqref{Eq:DzSzBiC0}, we could augment the simple functionals $D_z$ with piecewise constants over $\omega_z$ and the simple functions $S_z$ with corresponding \(H^2\)-element-bubble functions. Then the new oscillation bound involves on the right-hand side the oscillation in \cite{Brenner.Gudi.Sung:10,Carstensen.Graessle.Nataraj:24}, which is formally of third order. The new indicator of the approximate residual, however, is more expensive.
\end{remark}

\begin{proof}[Proof of Theorem~\ref{TH:mainC0}]
Recall that, according to \cite[Theorem 3.18]{Veeser.Zanotti:18b}, we have that \(\delta_h\lesssim 1\) depending only on the shape constant  \(\gamma_\grid\) and the penalty parameter \(\sigma>\sigma_*\). Thanks to Lemmas~\ref{L:C0dist2conf},~\ref{L:C0construction-Pz-local} and~\ref{L:C0-partition-of-unity}, we can apply Theorem~\ref{T:error-estimator} and the claimed equivalence as well as the local lower bound follow with similar arguments as in the proof of Theorem~\ref{T:CR-error-estimator-1}; note that for the oscillation bound, just stability and no invariance of $\calP_{z}$ is invoked.  
\end{proof}

\subsection{A quasi-optimal Morley method}
\label{sec:morley}
In this section, we derive a strictly equivalent error estimator for the  quasi-optimal Morley method  from \cite[Section~4,3]{Veeser.Zanotti:19b} approximating the two-dimensional biharmonic problem~\eqref{Eq:BiHarmPI}. The estimator is simplified in the spirit of Theorem~\ref{T:CR-error-estimator-2}, reducing this time to an approximate distance to $\mathring{H}^2(\Omega)$ and data oscillation that is classically bounded and formally of higher order.  Again, the simplification is obtained by locally nonconforming simple test functions that are orthogonal to the residual. 

% Morley method
In the notation of Section~\ref{sec:doma-mesh-polyn}, the Morley space \cite{Morley:68} is
\begin{align*}
	\begin{aligned}
		\MR := \bigg\{ v_h \in S^0_2 &\mid v_h 
		\text{ is continuous in \(\vertices^i\), } v_h=0 \text{
			in }
		\vertices\setminus\vertices^i,
		\\
		&\quad\text{and}~\forall F \in \sides ~\int_F \jump{\partial_\normal v_h}  = 0 \bigg\}
	\end{aligned}
\end{align*}
and the discrete problem reads
\begin{equation}
\label{eq:qo-MR}
\begin{multlined} 
		\text{for \(f\in H^{-2}(\Omega)\)},~\text{find $u_h \in \MR$ such that}
	\\
		\forall v_h\in \MR\quad \int_\Omega D^2_h
		u_h: D_h^2 v_h=\langle f, \calE_{\mr} v_h\rangle.
\end{multlined}
\end{equation}
The smoothing operator \(\calE_{\mr}:\MR\to \mathring{H}^2(\Omega)\) has the following properties:
for all \(v_h\in \MR\), $F\in\sides$, $z \in \vertices$, and $T \in \grid$, 
\begin{subequations}\label{eq:EMR}
\begin{gather}
		\label{eq:EMRcons}
		\calE_{\mr}v_h(z)
		=
		v_h(z)
	\quad\text{and}\quad
		\int_F \partial_\normal\calE_{\mr} v_h
		=
		\int_F \partial_\normal v_h,
	\\ \label{eq:EMRstab}                                           
		\begin{align}
			\begin{aligned}
				h_T^{-4}\norm[L^2(T)]{v_h-\calE_{\mr} v_h}^2
			+
			h_T^{-2}\norm[L^2(T)]{\nabla(v_h-\calE_{\mr} v_h)}^2
			\\
			{} + \norm[L^2(T)]{D^2 \calE_{\mr} v_h}^2\lesssim \norm[L^2(\omega_T)]{D^2_hv_h}^2;
			\end{aligned}
		\end{align}
\end{gather}
see \cite[Propostition~3.17]{Veeser.Zanotti:19b}.	
\end{subequations}
% setting at hand
We note that the weak formulation of the biharmonic problem \eqref{Eq:BiHarmPI} and this method fit into the abstract framework of Section~\ref{sec:qo-methods} with
\begin{equation}
\label{setting-for-Biharmonic-and-MR}
	\begin{gathered}
		\tV = \mathring{H}^2(\Omega) + \MR
		\quad\text{with}\quad
		V = \mathring{H}^2(\Omega),
		\quad%\text{and}\quad
		V_h = \MR,
		\\
		\ta(v,w) = \int_\Omega D^2_h v : D_h^2 w,
		\quad
		\norm{v} =\|D^2_h v\|_{L^2(\Omega)} ,
		%\quad
		% v,w\in \mathring{H}^1(\Omega) + \CR_1,
		\\
		a_h = \ta_{|\MR\times\MR},
		\quad\text{and}\quad
		\calE_h = \calE_{\mr}.
	\end{gathered}
\end{equation}
The conservation \eqref{eq:EMRcons} of point values and face means of the normal derivatives, together with element-wise integration by parts, implies
\begin{align}
\label{Eq:MR-smoother-overconsistency}
	\forall v_h,w_h\in\MR\quad \int_\Omega D_h^2w_h:D_h(v_h-\calE_{\mr}v_h)=0.
\end{align}
Since both discrete and continuous bilinear forms are restrictions of \(\ta\), the method is thus overconsistent, i.e.\ its consistency measure vanishes, \(\delta_h=0\). Hence, Proposition~\ref{P:QuasiOptMethods} and the stability~\eqref{eq:EMRstab} ensure that the method is quasi-optimal with the constant $\Cqo=\normtr{\calE_{\mr}}$ depending on the shape constant $\gamma_\grid$. 

% establishing (H1)
To apply the abstract a~posteriori analysis in Section~\ref{sec:posteriori-analysis}, we establish its main assumptions. As we proceed similarly to Section~\ref{sec:qo-C0} , we focus on the differences and start with \eqref{A:equiv-to-dist-to-conf} regarding the distance to conformity. Motivated by Section~\ref{S:H20-quad-HCT}, we  define a linear and bounded operator \(\calA_{\mr}:\MR\to \mathring{H}^2(\Omega)\) with the help of \(\HCT\) elements by 
\begin{equation}
\label{df:AHCT}
	\calA_{\mr}  v_h
	:=
	\sum_{z \in \vertices^i}\Big(v_h(z) \Upsilon^0_z
	 +
	\sum_{j=1}^2 \frac{\partial}{\partial x_j}({v_h}_{|T_z})(z) \Upsilon^j_z \Big)
	+
	\sum_{F\in\sides^i} \partial_{\normal}{v_h}% _{|T_F}
	 (m_F)\Upsilon_F,
\end{equation}
which in contrast to the operator $\calA_{\COip}$ in \eqref{df:AC0} for the $C^0$ interior penalty method exploits continuity of the normal derivatives in edge midpoints. The next lemma specifies Proposition~\ref{P:inds-for-dist-to-conf} for $\calA_{\mr}$.

\begin{lemma}[Approximate $\|D^2_h\cdot\|_{L^2}$-distance to $\mathring{H}^2$ by averaging]
\label{L:MRdist2conf}
In the setting~\eqref{setting-for-Biharmonic-and-MR}, the operator \(\calA_{\mr}\) in~\eqref{df:AHCT} satisfies
assumption~\eqref{A:equiv-to-dist-to-conf}  and \(\Cav\gtrsim 1\):  for all $v_h \in \MR$,
\begin{align*}
	\norm[L^2(\Omega)]{D_h^2(\calA_{\mr}v_h-v_h)}
	% \lesssim
	% \left(\sum_{F\in\sides} \int_F \frac{1}{h} \left|\jump{\nabla_h
	% v_h}\right|^2\right)^{\frac12}
	\lesssim \inf_{v\in \mathring{H}^2(\Omega)}\norm[L^2(\Omega)]{D_h^2(v-v_h)}
	\leq
	\norm[L^2(\Omega)]{D_h^2(\calA_{\mr}v_h-v_h)}.
\end{align*}
\end{lemma}

\begin{proof}
The operator $\calA_{\mr}$ maps \(\MR\) into $ \mathring{\HCT} = \HCT \cap \mathring{H^2}(\Omega) $ and is invariant on \( \MR\cap \mathring{H}^2(\Omega)\subset \mathring{\HCT}\). Hence, \eqref{A:equiv-to-dist-to-conf} and the second claimed inequality hold.

% second inequality
It remains to show the first claimed inequality. Given $v_h \in \MR$ and a fixed  element \(T\in\grid\), we have $v_h{}_{|T} \subset S^2_3(\grid_T)$ and thus
\begin{align*}
		D^2_h \big( v_h - \calA_{\mr}(v_h) \big)_{|T}
		=
		%\\
		\sum_{z\in\vertices^i\cap T} \sum_{j=1}^2
		\partial_j({v_h}_{|T}-\calA_{\mr} v_h)(z) D^2\Upsilon_z^{j}{}_{|T}.
\end{align*}
Furthermore, arguing similarly as for~\eqref{Eq:Dvh-DAvh<DJumps},
% from \(\operatorname{supp}\Upsilon_z^i=\omega_z\), \(i=1,2\), that
% \begin{align*}
%   \norm[L^2(T)]{D^2_h(v_h - \calA_{\mr}(v_h))}\le
%   \sum_{z\in T\cap\vertices}|\nabla_h({v_h}_{|T}-\calA_{\mr} v_h)(z)|\, \norm[L^2(T)]{D^2\Upsilon_z^{i}}
% \end{align*}
% Using that \(\nabla_h({v_h}-\calA_{\mr} v_h)\) is piecewise
% polynomial, we obtain with classical equivalence of norms and scaling arguments that
\begin{align}
\label{Eq:Dvh-DAvh<DJumps;MR}
		|\nabla_h({v_h}_{|T}-\calA_{\mr} v_h)(z)|
		\lesssim 
		\sum_{F\in\sides_z^i}\|h^{-\frac12}\jump{\nabla_h v_h}\|_{L^2(F)},
		% = \sum_{F\in\sides,z\inF}\|h^{-\frac12}\jump{\partial_\normal
		%   v_h}\|_{L^2(F)};
\end{align}
% compare with the proof of~\cite[Proposition
% 3.17]{Veeser.Zanotti:19b}.
% It follows from the scaling of
% the averaged HCT basis functions~\cite[Lemma
% 3.14]{Veeser.Zanotti:19b} that
% In order to prove the second inequality, 
where each face contribution has mean value zero, i.e.	\(\int_F \jump{\nabla_h v_h}=0\) for all $F \in \sides_z^i$.  Indeed, \(\int_F \jump{\nabla_h v_h\cdot\normal}=0\) follows explicitly from the definition of \(\MR\), while \(\int_F \jump{\nabla_h v_h\cdot\tangent_F}=0\), where the unit vector $\tangent_F = (-\normal_2,\normal_1)_{|F}$ is  tangent to $F$, is a consequence of  the continuity of \(v_h\) in \(\vertices\). Hence,  a Poincar\'e inequality implies 
\begin{align*}
		\|h^{-\frac12}\jump{\nabla_h v_h}\|_{L^2(F)}
		\lesssim
		\|h^{\frac12}\jump{D^2_h v_h}\tangent_F\|_{L^2(F)}.
\end{align*}
As in the proof of \cite[Proposition 2.3]{Gallistl:15}, we introduce \(\phi_F := \jump{D^2_h v_h} \tangent_F \Psi_F / \int_F\Psi_F \) with the piecewise quadratic edge-bubble function \(\Psi_F\) from~\eqref{Eq:bubbles} and observe that,  for \(v \in \mathring{H}^2(\Omega)\),
\begin{align*}
		\|\jump{D^2_h v_h}\tangent_F\|_{L^2(F)}^2
		=
		\int_F \jump{D^2_h	v_h} \tangent_F \cdot \phi_F
		=
		\int_{\omega_F}D^2_h(v_h-v):\operatorname{Curl} \phi_F
\end{align*}
with
\begin{align*}
		\operatorname{Curl} \phi_F=
		\begin{pmatrix}
			-\partial\phi_{F,1}/\partial
			x_2&\partial\phi_{F,1}/\partial x_1\\
			-\partial\phi_{F,2}/\partial
			x_2&\partial\phi_{F,2}/\partial x_1
		\end{pmatrix}.
\end{align*}
Cauchy and inverse inequalities thus lead to
\begin{align}
\label{eq:MRlocal-jumpEst}
		\|h^{-\frac12}\jump{\nabla_h v_h}\|_{L^2(F)}^2
		\lesssim
		\norm[L^2(\omega_F)]{D^2_h(v_h-v)}^2.
\end{align}
We combine the above bounds with the scaling properties of $\Upsilon_z^{j}$, $j=1,2$, in Lemma~\ref{L:ScalingUpsilon} , sum over all elements $T\in\grid$,  and  take the infimum over all \(v\in \mathring{H}^2(\Omega)\). This concludes the second claimed inequality thanks to  the fact that the overlapping of the sets $\omega_F $ with $F \in \sides_z^i$, $z \in \vertices^i \cap T$, and $T \in \grid$ is bounded in terms of the shape constant $\gamma_\grid$.
\end{proof}

The proof of Lemma~\ref{L:MRdist2conf} reveals an alternative approximate distance to $\mathring{H}^2(\Omega)$.

\begin{corollary}[Approximate $\|D^2_h\cdot\|_{L^2}$-distance to $\mathring{H}^2$ by jumps]
\label{C:MRJumpDist2conf}
For \(v_h\in\MR\),
\begin{align*}
	\int_\Sigma \frac{1}{h} \left|\jump{\nabla_h v_h}\right|^2
	\eqsim
	\inf_{ v \in \mathring{H}^2(\Omega)} \norm[L^2(\Omega)]{D_h^2(v_h-v)}^2.
\end{align*}
\end{corollary}

We next establish assumption~\eqref{A:partition-of-unity} for the localization of the conforming residual. Motivated by Section~\ref{S:H20-quad-HCT}, we choose 
\begin{equation}
	\label{Eq:MR-partition-of-unity}
	\Index = \vertices,
	\qquad
	\Phi_z = \Upsilon_z^0,
	\qquad
	V_z = \mathring{H}^2(\omega_z),
	\qquad
	\calI_h = \Pi_{\mr},
\end{equation}
where \(\Pi_{\mr}\) denotes the Morley projection, viz.\ the \(\ta\)-orthogonal projection of \(\mathring{H}^2(\Omega)+\MR\) projection onto \(\MR\). Notice that \eqref{Eq:MR-partition-of-unity} differs from \eqref{Eq:C0-partition-of-unity} for the $C^0$ interior penalty method only in the choice of the operator $\calI_h$.

\begin{lemma}[Partition of unity in $H^2$ for biharmonic $\mr$]
\label{L:MR-partition-of-unity}
In the  setting  \eqref{setting-for-Biharmonic-and-MR}, the choices \eqref{Eq:MR-partition-of-unity} satisfy assumption
\eqref{A:partition-of-unity} with constants $\Cloc,\Ccol \lesssim 1$.  
\end{lemma}

\begin{proof}
Conditions (i) and (ii) in~\eqref{A:partition-of-unity} follow as in the proof of Lemma~\ref{L:C0-partition-of-unity} and we only need to verify condition~(iii). Let \(v\in \mathring{H}^2(\Omega)\).  The properties \eqref{eq:EMR} ensure that \(\calE_{\mr}\) is a bounded right inverse of \(\Pi_{\mr}\); see \cite[Lemma~3.13]{Veeser.Zanotti:19b}. We thus have \(v - \calE_{\mr}\Pi_{\mr} v = w -\Pi_{\mr} w \) with \(w=v-\calE_{\mr}\Pi_{\mr} v\). Using the scaling properties of $\Upsilon_z^0$ in Lemma~\ref{L:ScalingUpsilon}, stability and approximation properties of the Morley interpolation (see also \cite[Proposition 2.3]{Gallistl:15}),  and the stability~\eqref{eq:EMRstab} of the smoothing operator, we derive
\begin{multline*}
  \norm[L^2(\Omega)]{D^2_h((v - \calE_{\mr}\Pi_{\mr} v)\Upsilon^0_z)}
  =
  \norm[L^2(\omega_z)]{D_h^2\big((w - \Pi_{\mr} w)\Upsilon^0_z\big)}
\\
  \lesssim h^{-2}_z \norm[L^2(\omega_z)]{w - \Pi_{\mr} w} 
  +
  h^{-1}_z \norm[L^2(\omega_z)]{\nabla_h (w - \Pi_{\mr} w)}
  +
  \norm[L^2(\omega_z)]{D_h^2(w-\Pi_{\mr} w)}
\\
	\lesssim \norm[L^2(\omega_z)]{D^2w}
	=
	\norm[L^2(\omega_z)]{D^2 (v - \calE_{\mr}\Pi_{\mr} v)}
	\lesssim
	\norm[L^2(\omega_z^+)]{D^2v}.
\end{multline*}
Consequently,  summing over all \(z\in\vertices\) and accounting for the finite overlapping of the enlarged stars \(\omega_z^+\), \(z\in\vertices\), concludes the proof.
\end{proof}

% establishing (H3)
Let us now turn to choose the simple test functions and functionals for assumption \eqref{A:construction-Pz-local}. Our departure point is the local structure of the conforming residual, whose local contributions are represented in
\begin{equation*}
	\widehat{D}_z
	=
	\operatorname{span} \{ \widehat{\chi}_F \mid F \in \sides_z^i \}
	\oplus
	\begin{cases}
		\operatorname{span} \widehat{\chi}_z, &\text{if } z \in \vertices^i,
    \\
		\{0\}, &\text{if } z \in \vertices \setminus \vertices^i,
	\end{cases}
\end{equation*}
where $\widehat{\chi}_z$, $\widehat{\chi}_F$, $F \in \sides_z^i$, are the $H^{-2}(\Omega)$-variants of $\chi_z$, $\chi_F$, $F \in \sides_z^i$ in \eqref{def-of-chi_z,chi_F}; cf.\ \eqref{eq:StructOfDC0}.  Notably, the functionals  $\widehat{\chi}_z$, $z \in \vertices^i$, $\widehat{\chi}_F$, $F \in \sides^i$, are the degrees of freedom of $\MR$. Moreover, the nodal basis functions $\Psi^{\mr}_F$, $F\in\sides^i$, $\Psi^{\mr}_z$, $z\in\vertices^i$, given by
\begin{equation}
\label{eq:nodalBasisMR}
 \Psi^{\mr}_K \in \MR
 \quad\text{and}\quad
 \scp{\widehat{\chi}_{K'}}{\Psi^{\mr}_{K}} = \delta_{KK'},
 \qquad
 K, K' \in \vertices^i \cup \sides^i,
\end{equation}
satisfy the orthogonality
\begin{equation}
	\scp{\Res_h}{\calE_{\mr}\Psi^{\mr}_K} = 0
\end{equation}
thanks to the overconsistency of \eqref{eq:qo-MR}. Taking into account that the smoothing operator $\calE_{\mr}$ conserves the nodal variables $\widehat{\chi}_K$, $K \in \vertices^i \cup \sides^i$, see \eqref{eq:EMRcons}, we choose the simple functions
\begin{subequations}
\label{Eq:DzSzMR}
\begin{align}
\label{Eq:SzBiMR}
	S_z
	:=
	\operatorname{span} \{ \calE_{\mr}\psi^{\mr}_F \mid F \in \sides_z^i \}
	\oplus
	\begin{cases}
		\operatorname{span} \calE_{\mr}\psi^{\mr}_z, &\text{if } z \in \vertices^i,
		\\
		\{0\}, &\text{if } z \in \vertices \setminus \vertices^i,
	\end{cases}
\end{align}
and the simple functionals
\begin{align}
\label{Eq:DzBiMR}
		D_z
		:=
		\widehat{D}_z{}_{|\tV_z}\qquad\text{with}\qquad\tV_z=\mathring{H}^2(\omega_z)+S_z
\end{align}
\end{subequations}
and define, for convenience, $\chi_K := \widehat{\chi}_K{}_{|\tV_z}$ for $K \in \mathcal{K}_z := \sides_z^i \cup ( \{z\} \cap \vertices^i)$.

\begin{remark}[Dual bases in simple pairs for biharmonic $\mr$]
\label{R:MRBiorthogonalSystem}
In view of the conservation \eqref{eq:EMRcons}, the duality in \eqref{eq:nodalBasisMR} implies
\begin{equation*}
	\scp{\chi_{K'}}{\calE_{\mr}\psi^{\mr}_K} = \delta_{KK'},
\quad
   K, K' \in \mathcal{K}_z,
\end{equation*}
so that these bases of $D_z$ and $S_z$ are dual.
\end{remark}

\begin{lemma}[Local projections for biharmonic $\mr$]
\label{L:MRconstruction-Pz-local}
In the settings \eqref{setting-for-Biharmonic-and-MR} and \eqref{Eq:MR-partition-of-unity},  the choices \eqref{Eq:DzSzMR} satisfy assumption \eqref{A:construction-Pz-local} with \(\CL,\CNL\lesssim1\). In particular, \eqref{Eq:construction-Pz-local} defines projections $\calP_z: \Dual{\tV_z}\to D_z \subset \Dual{\tV_z}$ with \(\Dual{\tV_z}= \Dual{(\mathring{H}^2(\omega_z)+S_z)}\)~and  
\begin{equation*}
%\label{Eq:MR-local-projection}
  \forall g \in H^{-2}(\Omega)
 \quad
  \sum_{z \in \vertices}
  \norm[H^{-2}(\omega_z)]{\calP_z g}^2
  \lesssim
  \sum_{z \in \vertices}
  \norm[H^{-2}(\omega_z)]{g}^2.
\end{equation*}
\end{lemma}

\begin{proof}
% checking (i)
Condition~(i) of \eqref{A:construction-Pz-local} is clear from the definitions \eqref{Eq:DzSzMR} of the respective spaces. 
% checking (ii)
To verify condition~(ii), %of \eqref{A:construction-Pz-local},
let $z \in \vertices$ be arbitrary. Observations
similar to the proof of~\eqref{eq:StructOfDC0} imply \(\tA(\MR)_{|\mathring{H}^2(\omega_z)}\subset D_z\) and it  remains to verify  \(\tA(\MR)_{|S_z}\subset D_z\). To this end, we let \(v_h\in \MR\)  and proceed similarly as for~\eqref{eq:StructOfDC0} and denote by $\tangent_z^+$ an extension of the tangent field $\tangent_z$ in \eqref{def-t_z} to $\Sigma_z^+ = \cup_{F \in \sides_z^{i+}} F$, where the orientation of the unit vectors $\tangent_z^+{}_{|F}$, $F \in \sides_z^{i+} \setminus \sides_z^+$,  is arbitrary but fixed. For \(F\in\sides_z^i\), we have $\operatorname{supp} \calE_{\mr} \Psi^{\mr}_F \subset \omega_z^+$ and thus
\begin{align*}
 \ta(v_h, \calE_{\mr}\Psi^{\mr}_F)
 &=
 \int_{\Sigma_z^+} \jump{D_h^2 v_h \normal\cdot\normal} \partial_{\normal} \calE_{\mr}\Psi^{\mr}_F
 +
 \int_{\Sigma_z^+} \jump{D_h^2 v_h  \normal\cdot\tangent_z^+} \nabla \calE_{\mr} \Psi^{\mr}_F \cdot \tangent_z^+
\end{align*}
The second integral vanishes thanks to the fundamental theorem of calculus since \(\jump{D_h^2 v_h  \normal\cdot\tangent_z^+}\) is piecewise constant on \(\sides_z^{i+}\) and \( \calE_{\mr}\Psi^{\mr}_F(y) = \Psi^{\mr}_F(y) = 0\) for all \(y\in \vertices\).  In the light of \( \scp{\chi_{F'}}{\calE_{\mr}\Psi^{\mr}_F} = \scp{\chi_{F'}}{\Psi^{\mr}_F} = \delta_{FF'} \), the first  integral reduces to an integral over $F$ and we arrive at
\begin{align*}
	\ta( v_h, \calE_{\mr}\Psi^{\mr}_F)
	=
   \int_{F} \jump{D_h^2 v_h \normal\cdot\normal}\partial_{\normal} \calE_{\mr}\Psi^{\mr}_F.
\end{align*}
If $z \in \vertices^i$, a similar reasoning yields
\begin{align*}
	\ta( v_h, \calE_{\mr}\Psi^{\mr}_z )
	=
	\sum_{F\in\sides_z^i}\jump{D_h^2 v_h \normal\cdot \tangent_z}.
\end{align*}
This shows \(\tA(\MR)_{|S_z} \subset D_z\) and thus (ii) of \eqref{A:construction-Pz-local} is verified.
	
% checking (iii)
Remark~\ref{R:MRBiorthogonalSystem} readily implies \(\operatorname{\dim} D_z= \operatorname{\dim} S_z<\infty\) and
thus (iii) of \eqref{A:construction-Pz-local}.
%Note that this already implies the existence of \(\calP_z:\Dual{\tV_z}\to D_z\) as defined	in~\eqref{Eq:construction-Pz-local}. 
% checking (iv)
To prove (iv), fix \(z\in\vertices\),  without loss of generality $z \in \vertices^i$, and let   $\chi \in D_z$ with
\begin{align*}
		\scp{\chi}{v}
		=
		r_z v(z)
		+
		\sum_{F\in\sides_z^i} \int_F r_F \partial_\normal v  \quad\text{for}~v\in \mathring{H}^2(\omega_z).
\end{align*}
Standard arguments with scaled trace and Poincar\'e inequalities yield
\begin{align*}
		\scp{\chi}{v} % &\lesssim \sum_{F\in\sides_z^i}
		% \norm[L^2(F)]{r_F} \norm[L^2(F)]{\partial_\normal v} 
		% + h \abs{r_z} \norm[H^2(\pat)]{v} \\
		% &\lesssim
		% \sum_{F\in\sides_z^i} 
		% \norm[L^2(F)]{r_F} h^{\frac{1}{2}}\norm[H^2({\pat[F]})]{v}
		% + h \abs{r_z}  \norm[H^2(\pat)]{v}\\
		&\lesssim 
		\Big(\sum_{F\in\sides_z^i} h_z^2 |{r_F}|^2
		+
		h^2_z |r_z|^2 \Big)^{\frac12} \norm[L^2(\omega_z)]{D^2v}.
\end{align*}
To further bound the right-hand side, we use the duality of the bases of \(D_z\) and \(S_z\) and their scaling properties.
For \(F\in\sides_z^i\), we obtain
\begin{align*}
		|r_F|
		&=
		\Big|\int_F r_F \partial_\normal\calE_{\mr}(\Psi^{\mr}_F)\Big|
		\leq \norm[\Dual{S_z}]{\chi} \norm[L^2(\omega_z^+)]{D^2\calE_{\mr}(\Psi^{\mr}_F)}
		% \lesssim \norm[\Dual{{S^{\mr}_z}} ]{\chi}
		%   \norm[L^2(\omega_F)]{D_h^2\Psi^{\mr}_F}
		\lesssim
		\norm[\Dual{S_z}]{\chi} h_F^{-1},
\end{align*}
where we use the stability of \(\calE_{\mr}\) and  \(\norm[L^2(\Omega)]{D_h^2\Psi^{\mr}_F} \lesssim h_F^{-1}\), which follows from the fact that the Morley element is almost affine~\cite[Theorem~6.1.3]{Ciarlet:2002}. A similar reasoning yields
\begin{align*} 
	|r_z|
	=
	|r_z \calE_{\mr}\Psi^{\mr}_z(z)|
	=
	\left|\langle \chi, \calE_{\mr}\Psi^{\mr}_z\rangle\right| 
	&\leq
	\norm[\Dual{S_z}]{\chi} \norm[L^2(\omega_z^+)]{D^2 \calE_{\mr}{\Psi^{\mr}_z}}
	\lesssim
	h^{-1}_z\norm[\Dual{S_z} ]{\chi}. 
\end{align*}
% Here again, we used the stability of \(\calE_{\mr}\) and the scaling
% property \(\norm[L^2(\Omega)]{D_h^2\Psi^{\mr}_z}\lesssim
% h_z^{-1}\) using again techniques
% for almost affine elements.
Combining the above results, and  taking the supremum over all \(v\in \mathring{H}^2(\omega_z)\) proves the desired assertion.

% checking local and global, ie (v) of (H2), stability	
The last step is to verify the collective stability property of $\calP_z$, $z \in \vertices$.  Given  any vertex \(z\in\vertices\) and \(g\in H^{-2}(\omega_z^+)\), we have
\begin{align*}
 S_z\subset \mathring{H}^2(\omega_z^+)
\quad\text{and thus}\quad
  \norm[{\Dual{S_z}}]{g}
  \le
  \norm[H^{-2}(\omega_z^+)]{g}.
\end{align*}
Furthermore, we recall from Lemma~\ref{L:ScalingUpsilon} that the
functions \(\Upsilon_y^0\), \(y\in\vertices\), form a partition of unity with $\Upsilon_y^0 \in W^{2,\infty}(\Omega)$ and \(\operatorname{supp} \Upsilon_y^0=\omega_y\). Consequently, for local test functions \(v\in \mathring{H}^2(\omega_z^+)\), we derive 
\begin{align*}
	 \langle g, v\rangle
	 &=
	 \sum_{y\in\vertices\cap\omega_z^+} \langle g,v\Upsilon_y^0\rangle
	\le
	\sum_{y\in\vertices\cap\omega_z^+} \norm[H^{-2}(\omega_y)]{g} \norm[L^2(\omega_y)]{D^2(v\Upsilon_y^0)}
	\\
	&\lesssim
	\sum_{y\in\vertices\cap\omega_z^+}\norm[H^{-2}(\omega_y)]{g} \norm[L^2(\omega_z^+)]{D^2v},
\end{align*}
where, similarly to	\eqref{E:omega+<omega}, we use the scaling properties of $\Upsilon_y^0$, the Poincar\'e inequality with vanishing boundary values, and \(\operatorname{diam}(\omega_z^+)\eqsim h_y \) for all \(y\in\vertices\cap\omega_z^+\). As $\# \{ z \in V \mid \omega_z^+ \ni y\} \lesssim 1$ uniformly in $y \in \vertices$, we arrive at
\begin{align*}
		\sum_{z\in\vertices}\norm[H^{-2}(\omega_z^+)]{g}^2
		\lesssim
		\adjustlimits{\sum}_{z\in\vertices}{\sum}_{y\in\vertices\cap\omega_z^+}\norm[H^{-2}(\omega_y)]{g}^2
		\lesssim
		\sum_{z\in\vertices}\norm[H^{-2}(\omega_z)]{g}^2.
\end{align*}
This proves (v) of \eqref{A:construction-Pz-local} with $\Ccol \lesssim 1$. The collective stability bound follows from Proposition~\ref{P:construction-Pz-local}.
\end{proof}

Building on Lemmas~\ref{L:MRdist2conf},~\ref{L:MR-partition-of-unity}
and~\ref{L:MRconstruction-Pz-local}, the following theorem derives an estimator that is strictly equivalent to the error of the above Morley method.

\begin{theorem}[Estimator for biharmonic $\mr$]
\label{TH:mainMR}
	Let $u\in \mathring{H}^2(\Omega)$ be the weak solution
	of~\eqref{Eq:Biharmonic} and   $u_h\in\MR$ be its quasi-optimal
	Morley approximation from~\eqref{eq:qo-MR}.
	Given a tuning constant \(C>0\), define
	\begin{equation*}
		\Est_{\mr}^2
		:=
		\Ncf_{\mr}^2 + C^2 \Osc_{\mr}^2,  \text{ where} \quad
		\begin{aligned}[t]
			\Ncf_{\mr}^2
			&:= 
			\norm[L^2(\Omega)]{ D^2_h( u_h - \calA_{\mr} u_h )}^2,
		\\ %\quad\text{and}\quad
			\Osc_{\mr}^2
			&:= 
			\sum_{z\in\vertices} \norm[H^{-2}(\omega_z)]{f - \calP_z f}^2, 
  \end{aligned}
\end{equation*}
with the averaging operator $\calA_{\mr}$ from \eqref{df:AHCT}  and the local projections \( \calP_z
%:H^{-2}(\omega_z^+) \to D_z
\) from Lemma~\ref{L:MRconstruction-Pz-local}.
This estimator quantifies the error by
\begin{gather*}
		\underline{C} \Est_{\mr}
		\leq
		\norm[L^2(\Omega)]{D_h^2(u-u_h)}
		\leq
		\overline{C} \Est_{\mr}.
%		\intertext{and}
%		\underline{C} \Ncf_{\mr,T}:= \underline{C} \norm[L^2(T)]{ D^2_h( u_h - \calA_{\mr} u_h )
%		}
%		\le \norm[L^2(\widetilde{\omega}_T)]{D^2_h(u-u_h)}
\end{gather*}
%	where $\widetilde{\omega}_T = \{ T \} \cup \bigcup_{T'\in \grid, T ' \cap T \in \sides^i} T'$.
The equivalence constants $\underline{C}$, and $\overline{C}$ depend only on  the shape constant \(\gamma_\grid\) and the tuning constant \(C\).

Finally, if $f \in L^2(\Omega)$, we have
\begin{equation*}
	\Osc_{\mr}^2
	\lesssim
	\sum_{T \in \grid} h_T^4 \norm[L^2(T)]{f}^2,
\end{equation*}
where the latter is formally of higher order.
\end{theorem}

\begin{remark}[Estimator variants of biharmonic $\mr$]
\label{R:Variants-of-MR-estimator}
In view of  Corollary~\ref{C:MRJumpDist2conf}, we may replace $\Ncf_{\mr}$ with properly scaled jumps and a second tuning constant. This variant is widely used in the literature; see e.g.~\cite{Carstensen.Graessle.Nataraj:24}.  As in Remark~\ref{R:oscHiOC0}, the order of the oscillation can be increased by suitably augmenting simple functionals and test functions. 
\end{remark}

\begin{proof}[Proof of Theorem~\ref{TH:mainMR}]
The claimed equivalence is  a direct consequence of Theorems~\ref{T:error-estimator} and the observation
\begin{align*}
	\norm[H^{-2}(\omega_z)]{\calP_z\Res^{\mathtt{C}}_h}
	\lesssim
	\norm[\Dual{S_z}]{\calP_z\Res^{\mathtt{C}}_h}=\norm[\Dual{S_z}]{\Res^{\mathtt{C}}_h}
	=
	0,
\end{align*} 
which follows from (iv) of \eqref{A:construction-Pz-local},  \eqref{Eq:construction-Pz-local} and~\eqref{eq:qo-MR}. The oscillation bounds follows as in the proof of Theorem~\ref{TH:mainC0}.
%The local lower bound follows from~\eqref{Eq:Dvh-DAvh<DJumps;MR} and~\eqref{eq:MRlocal-jumpEst}. 
\end{proof}

\section{A numerical benchmark with rough source term}
\label{sec:numerics}
This section presents some numerical results obtained from the discretization of the Poisson problem \eqref{Eq:Poisson}, on an open polygon $\Omega \subseteq \R^2$, with Crouzeix-Raviart finite elements. Our implementation is realized with the library ALBERTA 3.1; see~\cite{Heine.Koester.Kriessl.Schmidt.Siebert,Schmidt.Siebert:05}.

Recall that the Crouzeix-Raviart discretization \eqref{Eq:CR-without-smoothing} without smoothing is defined only for sufficiently smooth source terms like, e.g., $f \in L^2(\Omega)$. This assumption, in turn, implies that the solution of the Poisson problem is in $H^2(\Omega)$, provided that $\Omega$ is convex. Hence, adaptive mesh refinement can provide a higher error decay rate than uniform refinement only in presence of re-entrant corners, as for the so-called L-shaped or slit domains. 

In contrast, the quasi-optimal discretization \eqref{Eq:CR-qopt} is defined for general sources $f \in H^{-1}(\Omega)$, entailing that $H^2(\Omega)$-regularity of the solution cannot be expected in general. Therefore, adaptive mesh refinement can asymptotically outperform the uniform refinement also for convex $\Omega$. We propose one such benchmark and test the ability of the estimators in Theorems~\ref{T:CR-error-estimator-1} and~\ref{T:CR-error-estimator-2} to quantify the error and to drive adaptive mesh refinement.

Let $\Omega = (0,1)^2$ be the unit square. For $\lambda \in (0,1)$, consider the manufactured weak solution of \eqref{Eq:Poisson} defined as
\begin{subequations}
	\label{E:benchmark}
	\begin{equation}
		\label{E:benchmark-solution}
		u(x_1, x_2) := \begin{cases}
			x_1(\lambda - x_1)x_2(1-x_2), & \text{for $x_1 \in (0,\lambda)$},\\
			(1-x_1)(x_1 - \lambda)x_2(1-x_2), & \text{for $x_1 \in (\lambda, 1)$}.
		\end{cases}
	\end{equation}
	This function is in $\mathring{H}^1(\Omega)$ and the gradient is discontinuous along the `critical' line $\Lambda = \{(x_1, x_2) \in \Omega \mid x_1 = \lambda\}$. Thus, the corresponding source $f = -\Delta u$ in \eqref{Eq:Poisson} consists of a regular and a singular component and acts on $v \in \mathring{H}^1(\Omega)$ as follows
	\begin{equation}
		\label{E:benchmark-source}
		\scp{f}{v} = \int_\Omega f_\mathrm{reg} v + \int_{\Lambda} w(x_2) v(\lambda, x_2) \mathrm{d} x_2,
	\end{equation}
	where $w(x_2) := x_2(1-x_2)$ and $f_\mathrm{reg} \in L^2(\Omega)$ is given by
	\begin{equation}
		\label{E:benchmark-source-singular}
		f_\mathrm{reg}(x_1, x_2) = \begin{cases}
			2x_1(\lambda - x_1) + 2x_2(1 - x_2), & \text{for $x_1 \in (0,\lambda)$},\\
			2(1 - x_1)(x_1 - \lambda) + 2x_2(1 - x_2), & \text{for $x_1 \in (\lambda, 1)$}.
		\end{cases} 
	\end{equation}
\end{subequations}

To obtain the initial mesh $\grid_0$, we subdivide $\Omega$ into four triangles, by drawing the two main diagonals, cf.\ Figure~\ref{F:benchmark-adaptive-meshes}(A). Each successive mesh $\grid_1, \grid_2, \dots$ is obtained from the previous one subdividing each marked triangle into four triangles by bisection and performing completion to preserve conformity. We set $\lambda = \frac{2}{3}$ in \eqref{E:benchmark}, to make sure that the critical line $\Lambda$ is not contained in the skeleton of any mesh $\grid_k$, $k \geq 0$, cf. Remark~\ref{R:critical-line} below. According to Remarks~\ref{R:computability} and~\ref{R:class-osc-as-surrogate}, we replace the data oscillation in the estimators from Theorems~\ref{T:CR-error-estimator-1} and~\ref{T:CR-error-estimator-2} by the surrogate
\begin{equation}
	\label{E:benchmark-surrogate-osc}
	\mathrm{sur}_{*,k}^2 := 
	\sum_{T \in \grid_k} h_T^2 \inf_{c \in \R} \norm[L^2(T)]{f_\mathrm{reg}-c}^2
	+
	\sum_{T \in \grid_k, T \cap \Lambda \neq \emptyset} h_T^2 \left( \max_{T \cap \Lambda} w \right)^2 
\end{equation}
where $* \in \{\mathsf{CR}, \widetilde{\mathsf{CR}}\}$. The second summand is inspired by \cite[Lemma~7.19]{Bonito.Canuto.Nochetto.Veeser:24}. For both estimators, we make use of the tuning constants $C_1 = 1$ and $C_2 = 0.3$.

\begin{table}[h!]
	\centering
	\begin{tabular}{|c|c|c|c|c|c|c|c|}
		\hline
		\multicolumn{2}{|c|}{Mesh} & \multicolumn{2}{|c|}{Error} & \multicolumn{2}{|c|}{Theorem~\ref{T:CR-error-estimator-1}} & \multicolumn{2}{|c|}{Theorem~\ref{T:CR-error-estimator-2}}\\
		\hline
		$k$ & $\#\grid_{k} $ & $\mathrm{err}_k$ & $\mathrm{eoc_k}$ & $\mathrm{est}_{\mathsf{CR},k}$ & $\mathrm{eff}_{\mathsf{CR},k}$ & $\mathrm{est}_{\widetilde{\mathsf{CR}},k}$ & $\mathrm{eff}_{\widetilde{\mathsf{CR}},k}$\\
		\hline
		0 & 4 &6.55e-02 &  & 1.37e-01 & 2.09 & 1.33e-01 & 2.03 \\ 
		1 & 16 &5.91e-02 & 0.07 & 1.07e-01 & 1.81 & 1.05e-01 & 1.78 \\ 
		2 & 64 &4.10e-02 & 0.26 & 7.33e-02 & 1.79 & 7.13e-02 & 1.74 \\ 
		3 & 256 &2.71e-02 & 0.30 & 4.89e-02 & 1.80 & 4.77e-02 & 1.76 \\ 
		4 & 1024 &1.86e-02 & 0.27 & 3.36e-02 & 1.81 & 3.30e-02 & 1.77 \\ 
		5 & 4096 &1.30e-02 & 0.26 & 2.35e-02 & 1.81 & 2.31e-02 & 1.78 \\ 
		6 & 16384 &9.11e-03 & 0.26 & 1.65e-02 & 1.82 & 1.62e-02 & 1.78 \\ 
		7 & 65536 &6.42e-03 & 0.25 & 1.17e-02 & 1.82 & 1.15e-02 & 1.78 \\ 
		8 & 262144 &4.53e-03 & 0.25 & 8.24e-03 & 1.82 & 8.09e-03 & 1.79 \\ 
		9 & 1048576 &3.20e-03 & 0.25 & 5.82e-03 & 1.82 & 5.72e-03 & 1.79 \\ 
		\hline
	\end{tabular}
	\caption{Errors, experimental orders of convergence, estimators and corresponding effectivity indices for uniform mesh refinement.}
	\label{F:benchmark-uniform-refinements}
\end{table}

We first consider the uniform mesh refinement, i.e.\ we obtain $\grid_{k+1}$ from $\grid_{k}$, $k \geq 0$, by marking each triangle for refinement. Since $u \in H^{1+s}(T)$ for all $T \in \grid_k$ only if $s < 0.5$, due to the choice of $\lambda$ above, we expect that the error $\mathrm{err}_k$ on $\grid_k$, measured in the broken $H^1$-norm from \eqref{setting-for-Poisson-and-CR1}, decays to zero as $(\#\grid_{k})^{-0.25}$. We test our expectation by computing the experimental order of convergence
\begin{equation*}
	\label{E:benchmark-EOC}
	\mathrm{eoc}_k := \log(\mathrm{err}_k/\mathrm{err}_{k-1}) / \log(\#\grid_{k-1}/\#\grid_{k}), \quad k \geq 1.
\end{equation*}
We compute also the effectivity index
\begin{equation*}
	\label{E:benchmark-effectivity-index}
	\mathrm{eff}_{*,k} := \mathrm{est}_{*,k} / \mathrm{err}_k, \quad k \geq 0
\end{equation*}
to assess the quality of the estimator $\mathrm{est}_{*,k}$, with $* \in \{\mathsf{CR}, \widetilde{\mathsf{CR}}\}$, computed on $\grid_{k}$. The data in Table~\ref{F:benchmark-uniform-refinements} suggest that, for $k\to \infty$, we have $\mathrm{eoc}_k \to 0.25$ as expected, the two estimators yield quite similar results and their effectivity indices are bounded between $1$ and $2$.

The decay rate $(\#\grid_{k})^{-0.25}$ of $\mathrm{err}_k$ above hinges on the discontinuity of $\nabla u$ along the critical line $\Lambda$. Therefore, the error is expected to concentrate along $\Lambda$. Remarkably, this property is captured by each indicator in the estimators, as illustrated in Figure~\ref{F:benchmark-estimator-components} for $\mathrm{est}_{\widetilde{\mathsf{CR}},k}$. In particular, all indicators decay to zero at the same rate and the surrogate oscillation is not of higher order compared to the error.

\begin{figure}[htp]
	%\captionsetup[]{labelformat=empty}
	\hfill
	\subfloat[$\mathrm{ncf}_{1,4}$]{\includegraphics[width=0.32\hsize]{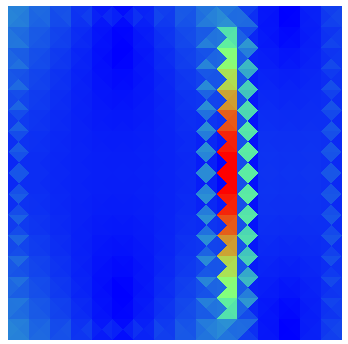}}
	\hfill
	\subfloat[$\eta_{1,4}$]{\includegraphics[width=0.32\hsize]{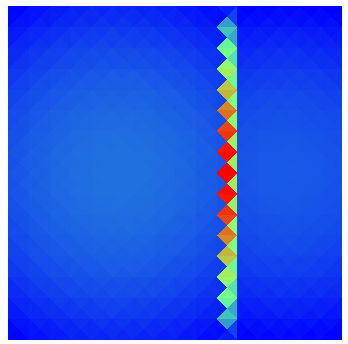}}
	\hfill
	\subfloat[$\mathrm{sur}_{1,4}$]{\includegraphics[width=0.32\hsize]{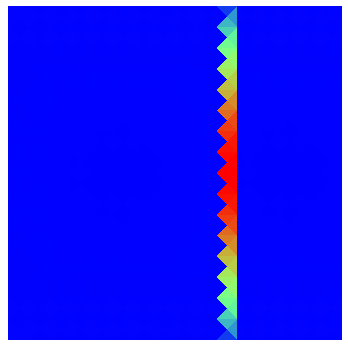}}
	\hfill
	\caption{Distribution in space of the indicators in the estimator $\mathrm{est}_{\widetilde{\mathsf{CR}},4}$, with oscillation replaced by surrogate, on the mesh $\grid_4$ obtained by uniform refinement. Lower and higher values correspond to cold and warm colors, respectively.} 
	\label{F:benchmark-estimator-components}
\end{figure}

\begin{figure}[htp]
	%\captionsetup[]{labelformat=empty}
	\hfill
	\subfloat[$\grid_0$]{\includegraphics[width=0.32\hsize]{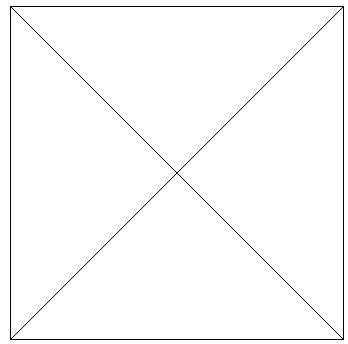}}
	\hfill
	\subfloat[$\grid_5$]{\includegraphics[width=0.32\hsize]{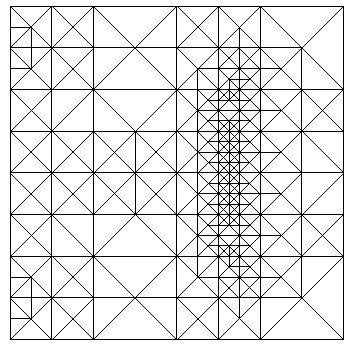}}
	\hfill
	\subfloat[$\grid_{10}$]{\includegraphics[width=0.32\hsize]{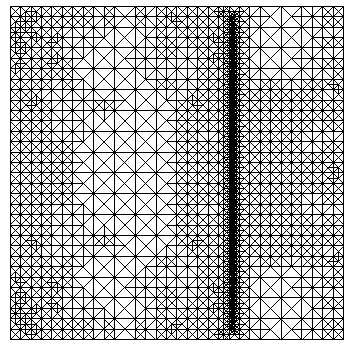}}
	\hfill
	\caption{Initial mesh $\grid_0$ and subsequent meshes $\grid_5$ and $\grid_{10}$ generated by adaptive refinement.} 
	\label{F:benchmark-adaptive-meshes}
\end{figure}

The above results suggest that adaptive mesh refinement could asymptotically outperform uniform refinement for this benchmark. To check this, we restrict to the use of the simplified estimator $\mathrm{est}_{\widetilde{\mathsf{CR}},k}$ as input for D\"{o}rfler marking with bulk parameter $\theta = 0.7$. As expected from Figure~\ref{F:benchmark-estimator-components}, such meshes are highly graded along the line $\Lambda$ and relatively coarse elsewhere, see Figure~\ref{F:benchmark-adaptive-meshes}. Moreover, owing to Figure~\ref{F:benchmark-errorplot}, both error and estimator decay to zero as $(\#\grid_k)^{-0.5}$, that is the best possible decay rate for a first-order method.

\begin{figure}[htp]
	\includegraphics[width=0.5\hsize]{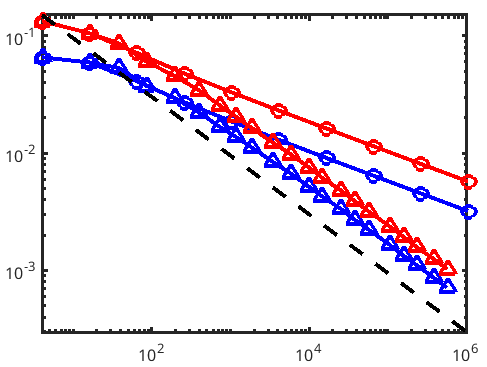}
	\caption{Error $\mathrm{err}_k$ (blue) and estimator $\mathrm{est}_{\widetilde{\mathsf{CR}},k}$ (red) for uniform ($\bigcirc$) and adaptive ($\triangle$) mesh refinement versus $\#\grid_k$. The dashed line indicates the decay rate $(\#\grid_k)^{-0.5}$.}
	\label{F:benchmark-errorplot}
\end{figure}

\begin{remark}[Critical line]
	\label{R:critical-line}
	The position of the line $\Lambda$ is crucial in this benchmark. Indeed, if $\Lambda$ was contained in the skeleton of, e.g., the initial mesh, then we would have $u \in H^2(T)$ for all $T \in \grid_{k}$, $k \geq 0$, and we would observe the error decay rate $(\#\grid_k)^{-0.5}$ also with uniform refinement. Moreover, we should modify the surrogate oscillation, because the simple functionals in \eqref{simple-fct-CR1-nconf} would approximate the singular component of the source term on $\Lambda$ more accurately than predicted by the latter term in \eqref{E:benchmark-surrogate-osc}.
\end{remark}

\subsection*{Funding}
Christian Kreuzer acknowledges funding by the Deutsche
For\-schungs\-ge\-mein\-schaft (DFG, German Research Foundation) --
321270008. 
Pietro Zanotti was supported by the GNCS-INdAM project CUP E53C23001670001.
\bibliographystyle{myalpha}
\bibliography{ms}

\end{document}